\documentclass[12pt,reqno,psamsfonts]{amsart}
\usepackage[alphabetic]{amsrefs}
\usepackage{feynmp}
\usepackage{latexsym}
\usepackage{amsfonts}
\usepackage{amssymb}
\usepackage{amscd}
\usepackage{amsbsy}
\usepackage{amsmath}
\usepackage{bbm}
\usepackage{euscript}
\usepackage{mathrsfs}
\usepackage{yhmath}

\usepackage{tensor}

\usepackage{dashbox}
\usepackage{environ}

\usepackage{textcomp}

\usepackage[dvips]{graphicx}
\usepackage{epsfig}    
\usepackage{xcolor}
\usepackage{tikz}
\usetikzlibrary{shapes,fit,intersections}
\usetikzlibrary{decorations.markings}
\usetikzlibrary{decorations.pathreplacing}
\usetikzlibrary{patterns}
\usetikzlibrary{matrix,arrows}
\usepgflibrary{decorations.pathmorphing}

\usepackage{tikz-cd}

\usepackage[all]{xy}
\xyoption{knot}
\xyoption{graph}
\xyoption{tile}
\xyoption{arc}
\xyoption{poly}

\usepackage{setspace}

\def\etc{{\it etc.\ }}
\def\ie{{\it i.e.\ }}
\def\eg{{\it e.g.\ }}
\def\cf{{\it cf.\ }}
\def\rhs{{\it r.h.s.\ }}
\def\lhs{{\it l.h.s.\ }}

\def\End{\mathop{{\rm End}}\nolimits}

\def\Hom{\mathop{{\rm Hom}}\nolimits}
\def\Ext{\mathop{{\rm Ext}}\nolimits}

\def\deg{ \mathop{{\rm deg}}\nolimits }

\def\max{\mathop{{\rm max}}\nolimits}

\def\End{\mathop{{\rm End}}\nolimits}

\parskip=\medskipamount                 
\thicklines                     
\def\inbar{\vrule height1.5ex width.4pt depth0pt}
\def\IC{\relax\,\hbox{$\inbar\kern-.3em{\rm C}$}}
\def\IN{\relax{\rm I\kern-.18em N}}
\def\IQ{\relax\,\hbox{$\inbar\kern-.3em{\rm Q}$}}
\def\IR{\relax{\rm I\kern-.18em R}}
\def\ZZ{\relax{\sf Z\kern-.4em Z}}
\def\etc{{\it etc\/}}

\newtheorem{theorem}{Theorem}[section]

\newtheorem{corollary}[theorem]{Corollary}

\newtheorem{lemma}[theorem]{Lemma}

\newtheorem{remark}[theorem]{Remark}

\catcode`\@=11

\newif\if@fewtab\@fewtabtrue

\catcode`\@=11

\newif\if@fewtab\@fewtabtrue

{\count255=\time\divide\count255 by 60
\xdef\hourmin{\number\count255} \multiply\count255
by-60\advance\count255 by\time
\xdef\hourmin{\hourmin:\ifnum\count255<10 0\fi\the\count255}}
\def\ps@draft{\let\@mkboth\@gobbletwo
    \def\@oddhead{}
    \def\@oddfoot
      {\hbox to 7 cm{\footnotesize {\em Draft of \jobname:} \draftdate
       \hfil}\hskip -7cm\hfil\rm\thepage \hfil}
    \def\@evenhead{}\let\@evenfoot\@oddfoot}

\def\ceqno{\global\@fewtabfalse
    \ifcase\@eqcnt \def\@tempa{& & &}\or \def\@tempa{& &}
      \or \def\@tempa{&}
      \or\def\@tempa{}\fi\@tempa
{\rm(\theequation)}}

\def\aeqno#1{\global\@fewtabfalse
    \ifcase\@eqcnt \def\@tempa{& & &}\or \def\@tempa{& &}
      \or \def\@tempa{&}
      \or\def\@tempa{}\fi\@tempa
{\rm(\theequation,#1)}}

\def\label#1{\ifnum\draftcontrol=1
 \global\def\draftnote{$\scriptstyle #1$}\fi
 \@bsphack\if@filesw {\let\thepage\relax
   \def\protect{\noexpand\noexpand\noexpand}%
\xdef\@gtempa{\write\@auxout{\string
      \newlabel{#1}{{\@currentlabel}{\thepage}}}}}\@gtempa
   \if@nobreak \ifvmode\nobreak\fi\fi\fi
  \@esphack}

\def\alabel#1#2{\label{#1}\global\@fewtabfalse
    \ifcase\@eqcnt \def\@tempa{& & &}\or \def\@tempa{& &}
      \or \def\@tempa{&}
      \or\def\@tempa{}\fi\@tempa
{\hbox to 3cm{\phantom{\rm(\theequation,#2)} \draftnote
\hfil}\hskip -3cm {\rm(\theequation,#2)}}}

\def\clabel#1{\label{#1}\global\@fewtabfalse
    \ifcase\@eqcnt \def\@tempa{& & &}\or \def\@tempa{& &}
      \or \def\@tempa{&}
      \or\def\@tempa{}\fi\@tempa
{\hbox to 3cm{\phantom{\rm(\theequation)} \draftnote \hfil}\hskip
-3cm{\rm(\theequation)}}}

\def\eqnarray{\def\draftnote{{}}\global\@fewtabtrue
\stepcounter{equation}\let\@currentlabel=\theequation
\global\@eqnswtrue
\global\@eqcnt\z@\tabskip\@centering\let\\=\@eqncr
$$\halign to \displaywidth\bgroup\@eqnsel\hskip\@centering\@eqcnt\z@
  $\displaystyle\tabskip\z@{##}$&\global\@eqcnt\@ne
  \hskip 1\arraycolsep \hfil$\displaystyle{##}$\hfil
  &\global\@eqcnt\tw@ \hskip 1\arraycolsep
$\displaystyle\tabskip\z@{##}$ \hfil
\tabskip\@centering&\global\@eqcnt\thr@@\llap{##}\tabskip\z@ \cr}

\def\endeqnarray{\@@eqncr\egroup
      \global\advance\c@equation\m@ne$$\global\@ignoretrue}

\def\@eqnnum{\hbox to 3cm{\phantom{\rm(\theequation)} \draftnote
                         \hfil}\hskip -3cm {\rm(\theequation)}}

\def\@@eqncr{\let\@tempa\relax
    \ifcase\@eqcnt \def\@tempa{& & &}\or \def\@tempa{& &}
      \or \def\@tempa{&}
      \or\def\@tempa{}
\fi\@tempa \if@eqnsw \if@fewtab\@eqnnum\fi
\stepcounter{equation}\fi\global
\@eqnswtrue\global\@eqcnt\z@\global\@fewtabtrue\cr}

\def\draftcite#1{\ifnum\draftcontrol=1#1\else{}\fi}

\def\@lbibitem[#1]#2{\item{}\hskip -3cm \hbox to 2cm
{\hfil$\scriptstyle\draftcite{#2}$}\hskip
1cm[\@biblabel{#1}]\if@filesw
     {\def\protect##1{\string ##1\space}\immediate
      \write\@auxout{\string\bibcite{#2}{#1}}}\fi\ignorespaces}

\def\@bibitem#1{\item\hskip -3cm \hbox to 2cm
{\hfil $\scriptstyle\draftcite{#1}$}\hskip 1cm \if@filesw
\immediate\write\@auxout
       {\string\bibcite{#1}{\the\value{\@listctr}}}\fi\ignorespaces}

\def\draftdate{\number\month/\number\day/\number\year\ \ \ \hourmin }
 \global\def\draftcontrol{0}
\catcode`\@=12

\def\theequation{{\thesection.\arabic{equation}}}

\def\qq{\begin{eqnarray}}
\def\qqq{\end{eqnarray}}
\def\ee{\begin{eqnarray}}
\def\eee{\end{eqnarray}}
\def\rx#1{~(\ref{#1})}

\def\ex#1{eq.\hspace*{-3pt}\rx{#1}}

\def\xlee#1{ \begin{eqnarray} \label{#1} }
\def\xeee{ \end{eqnarray} }
\def\ylee#1{ \begin{eqnarray}\nonumber }
\def\yeee{ \end{eqnarray} }
\def\zlee#1{ \begin{displaymath} }
\def\zleee{ \end{displaymath} }
\def\wlee#1{ $ }

\newlength{\shiftwidth}
\def\shift#1{&&\hbox to \shiftwidth{\hfill $\displaystyle#1$}}
\newlength{\sshiftwidth}
\def\sshift#1{\lefteqn{\hbox to
\sshiftwidth{\hfill$\displaystyle#1$}}}

\def\qbezier{\bezier{120}}
\def\DottedCircle{
\bezier{4}(0.966,-0.259)(1.04,0)(0.966,0.259)
\bezier{4}(0.966,0.259)(0.897,0.518)(0.707,0.707)
\bezier{4}(0.707,0.707)(0.518,0.897)(0.259,0.966)
\bezier{4}(0.259,0.966)(0,1.04)(-0.259,0.966)
\bezier{4}(-0.259,0.966)(-0.518,0.897)(-0.707,0.707)
\bezier{4}(-0.707,0.707)(-0.897,0.518)(-0.966,0.259)
\bezier{4}(-0.966,0.259)(-1.04,0)(-0.966,-0.259)
\bezier{4}(-0.966,-0.259)(-0.897,-0.518)(-0.707,-0.707)
\bezier{4}(-0.707,-0.707)(-0.518,-0.897)(-0.259,-0.966)
\bezier{4}(-0.259,-0.966)(0,-1.04)(0.259,-0.966)
\bezier{4}(0.259,-0.966)(0.518,-0.897)(0.707,-0.707)
\bezier{4}(0.707,-0.707)(0.897,-0.518)(0.966,-0.259) }
\def\Endpoint[#1]{
\ifcase#1 \put(1,0){\circle*{0.15}}
\or\put(0.866,0.5){\circle*{0.15}}
\or\put(0.5,0.866){\circle*{0.15}} \or\put(0,1){\circle*{0.15}}
\or\put(-0.5,0.866){\circle*{0.15}}
\or\put(-0.866,0.5){\circle*{0.15}} \or\put(-1,0){\circle*{0.15}}
\or\put(-0.866,-0.5){\circle*{0.15}}
\or\put(-0.5,-0.866){\circle*{0.15}} \or\put(0,-1){\circle*{0.15}}
\or\put(0.5,-0.866){\circle*{0.15}}
\or\put(0.866,-0.5){\circle*{0.15}} \fi}
\def\Arc[#1]{
\thicklines         
\ifcase#1 \bezier{25}(0.966,-0.259)(1.04,0)(0.966,0.259) \or
\bezier{25}(0.966,0.259)(0.897,0.518)(0.707,0.707) \or
\bezier{25}(0.707,0.707)(0.518,0.897)(0.259,0.966) \or
\bezier{25}(0.259,0.966)(0,1.04)(-0.259,0.966) \or
\bezier{25}(-0.259,0.966)(-0.518,0.897)(-0.707,0.707) \or
\bezier{25}(-0.707,0.707)(-0.897,0.518)(-0.966,0.259) \or
\bezier{25}(-0.966,0.259)(-1.04,0)(-0.966,-0.259) \or
\bezier{25}(-0.966,-0.259)(-0.897,-0.518)(-0.707,-0.707) \or
\bezier{25}(-0.707,-0.707)(-0.518,-0.897)(-0.259,-0.966) \or
\bezier{25}(-0.259,-0.966)(0,-1.04)(0.259,-0.966) \or
\bezier{25}(0.259,-0.966)(0.518,-0.897)(0.707,-0.707) \or
\bezier{25}(0.707,-0.707)(0.897,-0.518)(0.966,-0.259) \fi}
\def\DottedArc[#1]{
\ifcase#1 \bezier{4}(0.966,-0.259)(1.04,0)(0.966,0.259) \or
\bezier{4}(0.966,0.259)(0.897,0.518)(0.707,0.707) \or
\bezier{4}(0.707,0.707)(0.518,0.897)(0.259,0.966) \or
\bezier{4}(0.259,0.966)(0,1.04)(-0.259,0.966) \or
\bezier{4}(-0.259,0.966)(-0.518,0.897)(-0.707,0.707) \or
\bezier{4}(-0.707,0.707)(-0.897,0.518)(-0.966,0.259) \or
\bezier{4}(-0.966,0.259)(-1.04,0)(-0.966,-0.259) \or
\bezier{4}(-0.966,-0.259)(-0.897,-0.518)(-0.707,-0.707) \or
\bezier{4}(-0.707,-0.707)(-0.518,-0.897)(-0.259,-0.966) \or
\bezier{4}(-0.259,-0.966)(0,-1.04)(0.259,-0.966) \or
\bezier{4}(0.259,-0.966)(0.518,-0.897)(0.707,-0.707) \or
\bezier{4}(0.707,-0.707)(0.897,-0.518)(0.966,-0.259) \fi}
\def\Chord[#1,#2]{
\thinlines \ifnum#1>#2\Chord[#2,#1] \else\ifnum#1<#2 \ifcase#1
\ifcase#2 \or\qbezier(1,0)(0.516,0.138)(0.866,0.5)
\or\qbezier(1,0)(0.45,0.26)(0.5,0.866)
\or\qbezier(1,0)(0.327,0.327)(0,1)
\or\qbezier(1,0)(0.179,0.311)(-0.5,0.866)
\or\qbezier(1,0)(0.0536,0.2)(-0.866,0.5) \or\put(1, 0){\line(-2,
0){2}} \or\qbezier(1,0)(0.0536,-0.2)(-0.866,-0.5)
\or\qbezier(1,0)(0.179,-0.311)(-0.5,-0.866)
\or\qbezier(1,0)(0.327,-0.327)(0,-1)
\or\qbezier(1,0)(0.45,-0.26)(0.5,-0.866)
\or\qbezier(1,0)(0.516,-0.138)(0.866,-0.5) \fi \or\ifcase#2\or
\or\qbezier(0.866,0.5)(0.378,0.378)(0.5,0.866)
\or\qbezier(0.866,0.5)(0.26,0.45)(0,1)
\or\qbezier(0.866,0.5)(0.12,0.446)(-0.5,0.866)
\or\qbezier(0.866,0.5)(0,0.359)(-0.866,0.5)
\or\qbezier(0.866,0.5)(-0.0536,0.2)(-1,0) \or\put(0.866,
0.5){\line(-5, -3){1.73}}
\or\qbezier(0.866,0.5)(0.146,-0.146)(-0.5,-0.866)
\or\qbezier(0.866,0.5)(0.311,-0.179)(0,-1)
\or\qbezier(0.866,0.5)(0.446,-0.12)(0.5,-0.866)
\or\qbezier(0.866,0.5)(0.52,0)(0.866,-0.5) \fi \or\ifcase#2\or\or
\or\qbezier(0.5,0.866)(0.138,0.516)(0,1)
\or\qbezier(0.5,0.866)(0,0.52)(-0.5,0.866)
\or\qbezier(0.5,0.866)(-0.12,0.446)(-0.866,0.5)
\or\qbezier(0.5,0.866)(-0.179,0.311)(-1,0)
\or\qbezier(0.5,0.866)(-0.146,0.146)(-0.866,-0.5) \or\put(0.5,
0.866){\line(-3, -5){1}} \or\qbezier(0.5,0.866)(0.2,-0.0536)(0,-1)
\or\qbezier(0.5,0.866)(0.359,0)(0.5,-0.866)
\or\qbezier(0.5,0.866)(0.446,0.12)(0.866,-0.5) \fi
\or\ifcase#2\or\or\or \or\qbezier(0,1.)(-0.138,0.516)(-0.5,0.866)
\or\qbezier(0,1.)(-0.26,0.45)(-0.866,0.5)
\or\qbezier(0,1.)(-0.327,0.327)(-1,0)
\or\qbezier(0,1.)(-0.311,0.179)(-0.866,-0.5)
\or\qbezier(0,1.)(-0.2,0.0536)(-0.5,-0.866) \or\put(0, 1){\line(0,
-2){2}} \or\qbezier(0,1.)(0.2,0.0536)(0.5,-0.866)
\or\qbezier(0,1.)(0.311,0.179)(0.866,-0.5) \fi
\or\ifcase#2\or\or\or\or
\or\qbezier(-0.5,0.866)(-0.378,0.378)(-0.866,0.5)
\or\qbezier(-0.5,0.866)(-0.45,0.26)(-1,0)
\or\qbezier(-0.5,0.866)(-0.446,0.12)(-0.866,-0.5)
\or\qbezier(-0.5,0.866)(-0.359,0)(-0.5,-0.866)
\or\qbezier(-0.5,0.866)(-0.2,-0.0536)(0,-1) \or\put(-0.5,
0.866){\line(3, -5){1}}
\or\qbezier(-0.5,0.866)(0.146,0.146)(0.866,-0.5) \fi
\or\ifcase#2\or\or\or\or\or
\or\qbezier(-0.866,0.5)(-0.516,0.138)(-1,0)
\or\qbezier(-0.866,0.5)(-0.52,0)(-0.866,-0.5)
\or\qbezier(-0.866,0.5)(-0.446,-0.12)(-0.5,-0.866)
\or\qbezier(-0.866,0.5)(-0.311,-0.179)(0,-1)
\or\qbezier(-0.866,0.5)(-0.146,-0.146)(0.5,-0.866) \or\put(-0.866,
0.5){\line(5, -3){1.73}} \fi \or\ifcase#2\or\or\or\or\or\or
\or\qbezier(-1,0)(-0.516,-0.138)(-0.866,-0.5)
\or\qbezier(-1,0)(-0.45,-0.26)(-0.5,-0.866)
\or\qbezier(-1,0)(-0.327,-0.327)(0,-1)
\or\qbezier(-1,0)(-0.179,-0.311)(0.5,-0.866)
\or\qbezier(-1,0)(-0.0536,-0.2)(0.866,-0.5) \fi
\or\ifcase#2\or\or\or\or\or\or\or
\or\qbezier(-0.866,-0.5)(-0.378,-0.378)(-0.5,-0.866)
\or\qbezier(-0.866,-0.5)(-0.26,-0.45)(0,-1)
\or\qbezier(-0.866,-0.5)(-0.12,-0.446)(0.5,-0.866)
\or\qbezier(-0.866,-0.5)(0,-0.359)(0.866,-0.5) \fi
\or\ifcase#2\or\or\or\or\or\or\or\or
\or\qbezier(-0.5,-0.866)(-0.138,-0.516)(0,-1)
\or\qbezier(-0.5,-0.866)(0,-0.52)(0.5,-0.866)
\or\qbezier(-0.5,-0.866)(0.12,-0.446)(0.866,-0.5) \fi
\or\ifcase#2\or\or\or\or\or\or\or\or\or
\or\qbezier(0,-1.)(0.138,-0.516)(0.5,-0.866)
\or\qbezier(0,-1.)(0.26,-0.45)(0.866,-0.5) \fi
\or\ifcase#2\or\or\or\or\or\or\or\or\or\or
\or\qbezier(0.5,-0.866)(0.378,-0.378)(0.866,-0.5) \fi\fi\fi\fi}
\def\FullChord[#1,#2]{
\Endpoint[#1] \Endpoint[#2] \Arc[#1] \Arc[#2] \Chord[#1,#2] }
\def\EndChord[#1,#2]{
\Endpoint[#1] \Endpoint[#2] \Chord[#1,#2] }
\def\Picture#1{
\begin{picture}(2,1)(-1,-0.167)
#1
\end{picture}
}
\def\DottedChordDiagram[#1,#2]{
\Picture{\DottedCircle \FullChord[#1,#2]} }
\def\ZZ{ \mathbb{Z} }
\def\IQ{ \mathbb{Q} }
\def\IC{ \mathbb{C} }
\def\IR{ \mathbb{R} }

\def\bfx{ \mathbf{x} }
\def\bfy{ \mathbf{y} }

\def\Hom{ \mathop{\mathrm{Hom}}\nolimits }

\def\xId{ \mathbbm{1} }
\def\xIdv#1{ \xId_{#1} }
\def\xIdbstr{ \xIdv{\bstr} }

\def\tmcn{multi-cone}

\def\tmtdg{multi-degree}

\def\xdmm{ - }

\def\Conv#1{ \mathrm{Cone} (#1) }

\def\Hm{ \mathrm{H} }
\def\Hmiev#1{ \Hm_{\mathrm{#1}} }
\def\Hmint{ \Hmiev{in} }
\def\Hmext{ \Hmiev{out} }

\def\dgq{ \deg_{\mathrm{q}} }
\def\dgb{ \deg_{\mathrm{b}} }
\def\dga{ \deg_{\mathrm{a}} }
\def\dgt{ \deg_{\mathrm{t}} }
\def\dgtot{ \deg_{[\mathrm{tot}]} }

\def\tqdgr{$q$-degree}

\def\tadgr{$a$-degree}

\def\Zgrdg{$\ZZ$-grading}

\def\xIdv#1{ \xId_{#1} }

\def\xL{ L }

\def\xLb{ \bar{\xL} }

\def\sgmm{ s }

\def\betbr{ \beta }
\def\xLv#1{ \xL_{#1} }
\def\xLb{ \xLv{\betbr} }

\def\xfLbo{ \xLb }
\def\xfLbmo{ \xLb }

\def\Sgl{Soergel}
\def\Kszl{Koszul}

\def\bmdl{bimodule}
\def\Sbmdl{\Sgl\ \bmdl}
\def\Rq{Rouquier}
\def\Rqc{\Rq\ complex}

\def\tgrd{triply graded}
\def\tghm{\tgrd\ homology}
\def\ftghm{filtered \tghm}

\def\compl{complex}
\def\fcompl{filtered \compl}
\def\lkhm{link homology}
\def\tglkhm{\tgrd\ \lkhm}
\def\HMFp{HOMFLY-PT polynomial}

\def\frmd{framed}
\def\orntd{oriented}
\def\orfr{\orntd\ \frmd}

\def\Hch{Hochschild}
\def\Hchh{\Hch\ homology}

\def\brwd{braid word}
\def\brwdm{\brwd\ monoid}

\def\vrcr{virtual crossing}
\def\SOtN{\mathrm{SO}(2N)}
\def\tSOtN{$\SOtN$}
\def\SUN{\mathrm{SU}(N)}
\def\tSUN{$\SUN$}

\def\brwb{ \beta }
\def\brwbv#1{ \brwb_{#1} }
\def\brwbo{ \brwbv{1} }
\def\brwbt{ \brwbv{2} }
\def\brwbp{ \brwb' }
\def\brwbpp{ \brwb'' }

\def\brdb{ \beta }
\def\brdbv#1{ \brdb_{#1} }
\def\brdbo{ \brdbv{1} }
\def\brdbt{ \brdbv{2} }

\def\sg{ \sigma }
\def\sgn{ \sg^{-1} }
\def\sgv#1{ \sg_{#1} }
\def\sgnv#1{ \sgn_{#1} }
\def\sgi{ \sgv{i} }
\def\sgni{ \sgnv{i} }
\def\sgnio{ \sgnv{i+1} }
\def\sgo{ \sgv{1} }

\def\sgomo{ \sgo^{-1} }

\def\sqbr{ \sqcup }

\def\bstr{ n }
\def\blbn{ \Ngr }

\def\bfx{ \mathbf{x} }
\def\bfy{ \mathbf{y} }
\def\bfz{ \mathbf{z} }
\def\bfw{ \mathbf{w} }

\def\Qv#1{ \IQ[#1] }

\def\Qxyo{ \Qv{x_1,y_1} }
\def\Qz{ \Qv{z} }
\def\Qw{ \Qv{w} }

\def\otQv#1{ \otimes_{\Qv{#1}} }
\def\otQby{ \otQv{\bfy} }
\def\otQbx{ \otQv{\bfx} }
\def\otQbz{ \otQv{\bfz} }

\def\zcn{ \mathrm{cn} }

\def\zfr{ \otimes }

\def\xctv#1{ [\![ #1 ]\!] }
\def\xectv#1{ \left[\!\left[ #1 \right]\!\right] }

\def\zectv#1{ \left[\!\left[ #1 \right]\!\right] }
\def\zectcnv#1{ \zectv{#1}^{\zcn} }
\def\yectv#1{ \left[\!\!\left[ #1 \right]\!\!\right] }
\def\yectcnv#1{ \yectv{#1}^{\zcn}}

\def\yectvv#1#2#3{ \tensor[_{#2}]{\yectv{#1}}{_{#3}} }
\def\xxectvv#1#2#3{ \tensor[_{#2}]{\xectv{#1}}{_{#3}} }

\def\yectfrvv#1#2#3{ \tensor*[_{#2}]{\yectv{#1}}{_{#3}^{
\zfr}} }
\def\yectfrbzx#1{  \yectfrvv{#1}{\bfz}{\bfx} }

\def\yectvv#1#2#3{ \tensor[_{#2}]{\yectv{#1}}{_{#3}} }
\def\yectvxyt#1{ \yectvv{#1}{y_2}{x_2} }
\def\yectvyxt#1{ \yectvv{#1}{x_2}{y_2} }

\def\xctfrv#1{ \xctv{#1}^{\zfr} }
\def\xctcnv#1{ \xctv{#1}^{\zcn } }
\def\xctmv#1{ \xctv{#1}^{-} }

\def\bfxp{ \bfx' }
\def\bfyp{ \bfy' }

\def\xctvv#1#2#3{ {}_{#2}\xctv{#1}_{#3} }
\def\xectvv#1#2#3{ {}_{#2}\xectv{#1}_{#3} }
\def\xctvbyx#1{ \xctvv{#1}{\bfy}{\bfx} }
\def\xctvbyxp#1{ \xctvv{#1}{\bfyp}{\bfxp} }
\def\xctvbxy#1{ \xctvv{#1}{\bfx}{\bfy} }
\def\xctvbzx#1{ \xctvv{#1}{\bfz}{\bfx} }
\def\xctvbzy#1{ \xctvv{#1}{\bfz}{\bfy} }
\def\xctvbwz#1{ \xctvv{#1}{\bfw}{\bfz} }
\def\xctvbwx#1{ \xctvv{#1}{\bfw}{\bfx} }
\def\xctvbyxo#1{ \xctvv{#1}{y_1}{x_1} }
\def\xctvbyxt#1{ \xctvv{#1}{y_2}{x_2} }
\def\ctsgi{ \xctv{\sgi} }
\def\ctsgni{ \xctv{\sgni} }
\def\ctbrwb{ \xctv{\brwb} }
\def\ctbrwbo{ \xctv{\brwbo} }
\def\ctbrwbt{ \xctv{\brwbt} }

\def\xctcnvv#1#2#3{ {}_{#2}\xctcnv{#1}_{#3} }

\def\xctcnvbyx#1{ \xctcnvv{#1}{\bfy}{\bfx} }

\def\ctdm{ \xctv{\dmmy} }

\def\xLbv#1{ \xL_{#1} }
\def\xLbb{ \xLbv{\brwb} }
\def\xLbbo{ \xLbv{\brwbo} }
\def\xLbbt{ \xLbv{\brwbt} }
\def\xLbbp{ \xLbv{\brwbp} }
\def\xLbbpp{ \xLbv{\brwbpp} }

\def\xLsgmtn{ \xLbv{\ysgmmtn} }
\def\xLsgmtno{ \xLbv{\ysgmmtno} }
\def\xLsgmtnmo{ \xLbv{\ysgmmtnmo} }
\def\ctxLbb{ \xctv{\xLbb} }
\def\ctxLbbo{ \xctv{\xLbbo} }
\def\ctxLbbt{ \xctv{\xLbbt} }

\def\ctxLbbp{ \xctv{\xLbbp} }
\def\ctxLbbpp{ \xctv{\xLbbpp} }
\def\ctxtLbb{ \xctv{\xtLbb} }
\def\ctxtLbbxyo{ \xctvbyxo{\xtLbb} }
\def\ctxtLbbxyt{ \xctvbyxt{\xtLbb} }
\def\ctbrwbbxy{ \xctvbyx{\brwb} }

\def\xtLbv#1{ \widetilde{\xL}_{#1} }
\def\xtLbb{ \xtLbv{\brwb} }

\def\HHm{ \Hm\Hm }

\def\xIdbrn{ \xIdv{\bstr} }
\def\xIdbro{ \xIdv{1} }

\def\Rdt{R2}
\def\Rdta{\Rdt a}
\def\Rdtb{\Rdt b}
\def\Rdth{R3}

\def\Mo{M1}
\def\Mta{M2a}
\def\Mtb{M2b}

\def\Rmv{Reidemeister move}
\def\sRm{second \Rmv}
\def\tRm{third \Rmv}

\def\het{homotopy equivalent}
\def\hec{homotopy equivalence}

\def\fhet{filtered homotopy equivalent}
\def\fhec{filtered homotopy equivalence}

\def\mfltd{multi-filtered}
\def\nfltd{$\Ngr$-filtered}

\def\smspl{semi-split}
\def\mfspl{\mfltd\ \smspl}

\def\nfspl{\nfltd\ \smspl}

\def\mcpl{multi-complex}

\def\Htg{ \Hm^{(3)} }
\def\dmmy{ - }

\def\lkeq{ = }

\def\ccv#1{#1}
\def\ccA{ \ccv{A} }
\def\ccB{ \ccv{B} }

\def\fhmvv#1#2{ f_{#1#2} }
\def\fAB{\fhmvv{\ccA}{\ccB} }
\def\fBA{\fhmvv{\ccB}{\ccA} }
\def\hmv#1{ h_{#1} }
\def\hmA{ \hmv{\ccA} }
\def\hmB{ \hmv{\ccB} }

\def\FAB{ F_{AB} }
\def\FBA{ F_{BA} }
\def\HmA{ H_A }
\def\HmB{ H_B }

\def\dcv#1{ d_{#1} }
\def\dcA{ \dcv{\ccA} }
\def\dcB{ \dcv{\ccB} }

\def\cwdv#1#2{ (#1,#2) }
\def\cwdA{ \cwdv{\ccA}{\dcA} }
\def\cwdB{ \cwdv{\ccB}{\dcB} }

\def\xIA{ \xIdv{\ccA} }
\def\xIB{ \xIdv{\ccB} }

\def\ctC{ \mathcal{C} }

\def\nmv#1{ |#1| }

\def\Qv#1{ \IQ[#1] }
\def\Qbxy{ \Qv{\bfx,\bfy} }
\def\Qbx{ \Qv{\bfx} }
\def\Qby{ \Qv{\bfy} }
\def\Qxy{ \Qv{x,y} }

\def\Qmxy{ \Qv{\xm,\ym} }

\def\gq{q}

\def\ga{a}

\def\qgrdg{$\gq$-grading}
\def\agrdg{$\ga$-grading}

\def\xdg{ \mathrm{deg} }
\def\xdgv#1{ \xdg_{#1} }
\def\xdgq{ \xdgv{\gq} }
\def\xdga{ \xdgv{\ga} }

\def\qgrdd{$\gq$-graded}

\def\Zgrdd{$\ZZ$-graded}

\def\mcat#1{ \mathbf{#1} }
\def\Drv{ \mcat{D} }
\def\ctCh{ \mcat{Ch} }
\def\ctK{ \mcat{K} }

\def\ctChv#1{ \ctCh(#1) }
\def\ctKv#1{ \ctK(#1) }

\def\ctChuv#1{ \ctCh^{[#1]} }
\def\ctKuv#1{ \ctK^{[#1]} }
\def\Chum{ \ctChuv{\blbn} }
\def\Kum{ \ctKuv{\blbn} }
\def\Chuo{ \ctChuv{1} }
\def\Kuo{ \ctKuv{1} }
\def\ChumA{ \Chum(\ctA) }
\def\KumA{ \Kum(\ctA) }
\def\ChuoA{ \Chuo(\ctA) }

\def\ctChbuv#1{ \ctCh^{#1} }
\def\Chbum{ \ctChbuv{\blbn} }

\def\ChbumA{ \Chbum(\ctA) }

\def\ctA{ \mcat{A} }

\def\ctsg{ \mcat{sg} }
\def\ctsgv#1{ \ctsg^{#1} }
\def\ctsgvv#1#2{ #1-\ctsgv{#2} }
\def\ctsgmv#1{ \ctsgvv{#1}{\blbn} }
\def\ctsgov#1{ \ctsgvv{#1}{1} }
\def\ctsgcv#1{ #1-\ctsg }

\def\cfmod{ \mcat{fm} }
\def\cfmod{ \mcat{fi,g} }
\def\cfree{ \mcat{fr,g} }

\def\ctpt{ \mcat{Pt} }

\def\DQfv#1{ \Drv(\Qv{#1}-\cfmod) }

\def\DQfbxy{ \DQfv{\bfx,\bfy} }

\def\ChDQfbxy{ \ctChv{\DQfbxy} }

\def\KQfb{ \ctKv{\IQ-\cfmod}}
\def\ChQfb{ \ctChv{\IQ-\cfmod}}
\def\ChKQfb{ \ctChv{\KQfb}}

\def\ChFnA{ \ChumA }
\def\ChA{ \ctChv{\ctA} }

\def\Qvfr#1{ \Qv{#1}-\cfree }
\def\Qbxyfr{ \Qvfr{\bfx,\bfy} }

\def\ChFnQbxyfr{ \Chum\bigl(\Qbxyfr\bigr) }
\def\KFnQbxyfr{ \Kum\bigl(\Qbxyfr\bigr) }
\def\ChFoQbxyfr{ \Chuo\bigl(\Qbxyfr\bigr) }
\def\KFoQbxyfr{ \Kuo\bigl(\Qbxyfr\bigr) }
\def\ChbmFnQbxyfr{ \Chbum\bigl(\ChFnQbxyfr\bigr) }
\def\ChboFnQbxyfr{ \ctCh\bigl(\ChFnQbxyfr\bigr) }

\def\xcpAp{ B }

\def\FnA{ \ctsgmv{\ctA} }
\def\FA{ \ctsgcv{\ctA} }
\def\FoA{ \ctsgov{\ctA} }

\def\ptob#1{ (#1) }
\def\pto{ \ptob{1} }
\def\ptt{ \ptob{2} }
\def\ptn{ \ptob{\bstr} }
\def\ptnv#1{ \ptob{\bstr_{#1} } }
\def\ptno{ \ptnv{1} }
\def\ptnt{ \ptnv{2} }

\def\ptpr{ \sqcup }

\def\Eptv#1{ \End\ptob{#1} }
\def\Eptn{ \Eptv{\bstr} }
\def\Epto{ \Eptv{1} }
\def\Eptt{ \Eptv{2} }

\def\Smgv#1{ \mathrm{S}_{#1} }
\def\Smgn{ \Smgv{\bstr} }

\def\Smgt{ \Smgv{2} }

\def\tgct{target category}
\def\tgcts{target categories}

\def\Idpto{ \xIdv{\pto} }
\def\Idptn{ \xIdv{\ptn} }

\def\ene{\epsilon}

\def\grnv#1{ \mathbf{#1} }
\def\sha{ \grnv{a}}
\def\shq{ \grnv{q}}
\def\sht{ \grnv{t} }
\def\shs{ \grnv{s} }

\def\shst{ \shs^2 }
\def\shsh{ \shs^3 }

\def\gnst{ s }
\def\gnstp{ \gnst' }
\def\gnstv#1{ \gnst_{#1} }
\def\gnstot{ \gnstv{12} }
\def\gnstij{ \gnstv{ij} }
\def\gnsti{ \gnstv{i} }

\def\vert{ v }
\def\vertp{ \vert' }
\def\gnstvert{ \gnstv{\vert} }
\def\gnstvertp{ \gnstv{\vertp} }

\def\wert{ v }
\def\wertp{ \vert' }

\def\edg{ e }
\def\gnstev#1{ \gnstv{(#1)} }
\def\gnstedg{ \gnstev{\edg} }

\def\tsf{semi-free}
\def\tfsf{filtered \tsf}

\def\xFl{\mathrm{F}}
\def\xFlv#1{ \xFl_{#1} }
\def\xFli{ \xFlv{i} }
\def\xFlz{ \xFlv{0} }
\def\xFlo{ \xFlv{1} }

\def\xQl{\mathrm{Q}}

\def\xFlmv#1{ \xFl^{(#1)} }
\def\xFlmi{ \xFlmv{i} }

\def\xFlmimo{ \xFlmv{i-1} }
\def\xFlmvv#1#2{ \xFlmv{#1}_{#2} }
\def\xFlmij{ \xFlmvv{i}{j} }

\def\xQlmv#1{ \xQl^{(#1)} }

\def\xQlmvv#1#2{ \xQlmv{#1}_{#2} }
\def\xQlmij{ \xQlmvv{i}{j} }

\def\tot{ \mathrm{tot} }

\def\xFlmtotv#1{ \xFlmvv{\tot}{#1} }
\def\xFlmtoti{ \xFlmtotv{i} }
\def\xFlmtotio{ \xFlmtotv{i+1} }

\def\xGl{\mathrm{G}}

\def\xGlmv#1{ \xGl^{(#1)} }
\def\xGlmvv#1#2{ \xGlmv{#1}_{#2} }
\def\xGlmtotv#1{ \xGlmvv{\tot}{#1} }
\def\xGlmtoti{ \xGlmtotv{i} }

\def\blb{blob}
\def\blgn{\blb\ generator}
\def\blbm{\blb\ \bmdl}
\def\Prr{\mathrm{P}}
\def\Prrv#1{ \Prr(#1) }
\def\mcan{ \mathrm{cn} }
\def\Prcnxv#1{ #1^{\mathrm{cn}} }

\def\xCbv#1{ \bigl[ #1 \bigr] }

\def\xCsv#1{ \left[ #1 \right] }

\def\fres{filtered resolution}
\def\cres{canonical resolution}
\def\cfres{canonical \fres}

\def\xpar{ \shortparallel }

\def\xvcr{\!\times }
\def\xblb{\text{\textcurrency}}

\def\mMv#1{ M_{#1} }

\def\mMblb{ \mMv{\xblb} }

\def\tvsd{virtual saddle}
\def\vsd{ \phi }
\def\vsdv#1{ \vsd_{#1} }
\def\vsdp{ \vsdv{\xpar} }
\def\vsdm{ \vsdv{\xvcr} }
\def\vsde{ \vsdv{(\edg)} }

\def\vsdi{ \vsdv{i} }

\def\vsdvv#1#2{ \vsd_{#1}^{(#2)}}
\def\vsdijs{ \vsdvv{ij}{\gnst}}

\def\dgb{ \Delta }
\def\dgbv#1#2{ {}_{#1}\dgb_{#2} }

\def\dgbyx{ \dgbv{y}{x} }

\def\dgbxyt{ \dgbv{y_2}{x_2} }
\def\dgbyxp{ \dgbv{\yp}{\xp} }

\def\xp{ x_+ }
\def\xm{ x_- }
\def\yp{y_+}
\def\ym{ y_- }

\def\tctfn{categorification functor}

\def\dgbm{diagonal bimodule}

\def\xMcm{ M_{\mathrm{cm}} }

\def\zcn{ \mathrm{cn} }

\def\tflcn{filtered-canonical}
\def\tflcnm{\tflcn\ morphism}

\def\chp{ \chi_+ }
\def\chm{\chi_- }

\def\ssqt{ \shs_{q} }
\def\ssqti{ \ssqt^{-1} }
\def\saqt{ \sha_{q} }
\def\stat{ \sht_{a} }
\def\stot{ \sht_{\mathrm{tot}} }
\def\stati{ \stat^{-1} }
\def\stas{ \sht_{s} }
\def\stashf{\stas}
\def\xsht{ \stas }
\def\stasaqt{ \xsht\saqt }
\def\stasaqti{ (\stasaqt)^{-1} }

\def\xsmmt#1{
\left(\!
\begin{smallmatrix}
#1
\end{smallmatrix}
\!
\right)
}

\def\tbrgr{braid-graph}
\def\fdg{filtration diagram}

\def\brgr{\gamma}
\def\brgrv#1{ \brgr_{#1} }
\def\brgrw{ \brgrv{\wert} }
\def\brgrwp{ \brgrv{\wertp} }
\def\Ngr{m}

\def\cbrgr{ \xctv{\brgr} }
\def\cbrwb{ \xctv{\brwb} }
\def\cfrbr{ \xctfrv{\brgr} }
\def\cfrbraw{ \xctfrv{\brgrw} }
\def\cfrbrawp{ \xctfrv{\brgrwp} }
\def\cfrbw{ \xctfrv{\brwb} }
\def\cfrbrvr{ \cfrbr_\vert }
\def\ccnbr{ \xctcnv{\brgr} }
\def\ccnbri{ \xctcnv{\brgr_i} }
\def\ccnbrw{ \xctcnv{\brgrw} }

\def\ccnbw{ \xctcnv{\brwb} }
\def\cgnst{ \xctv{\gnst} }
\def\cgnstvert{ \xctv{\gnstvert} }
\def\cgnstvertp{ \xctv{\gnstvertp} }

\def\cmbai#1{
\underbrace{\zverlv{\lncf}{5}\cdots \zverlv{\lncf}{5}}_{i-1}\;
#1\;
\underbrace{\zverlv{\lncf}{5}\cdots \zverlv{\lncf}{5}}_{\bstr -i-1}
}

\def\fblb{ b^+ }
\def\fblg{ b^- }
\def\fblbv#1{ \fblb_{#1} }
\def\fblgv#1{ \fblg_{#1} }
\def\fblbi{ \fblbv{i} }
\def\fblgi{ \fblgv{i} }
\def\fblbj{ \fblbv{j} }
\def\fblgj{ \fblgv{j} }

\def\ctcnblbi{ \xctcnv{\fblbi} }
\def\ctcnblgi{ \xctcnv{\fblgi} }
\def\ctcnIdptn{ \xctcnv{\Idptn} }
\def\ctcngnsti{ \xctcnv{\gnsti} }

\def\sgm{ \sigma }
\def\sgmp{ \sgm }
\def\sgmm{ \sgm^- }

\def\Kr{Koszul resolution}
\def\cKr{canonical \Kr}
\def\varc{vertical arc}

\def\Zbstr{\ZZ^\blbn}

\def\cpv#1{ #1 }
\def\cpA{ A }
\def\cpB{ B }
\def\cpC{ C }
\def\cpCsm{ \cpC_{\mathrm{sym}} }

\def\cpd{ \cpv{d} }
\def\cqd{ d }
\def\cqdv#1{ \cqd_{#1} }
\def\cqdo{ \cqdv{1} }
\def\cqdm{ \cqdv{\blbn} }
\def\cqdi{ \cqdv{i} }
\def\cqdj{ \cqdv{j} }
\def\cqde{ \cnqdte }

\def\cnqd{ \cqd^{\mathrm{cn} } }
\def\cnqdv#1{ \cnqd_{#1} }
\def\etdg{ e }
\def\cnqdte{ \cnqdv{\etdg} }

\def\xcmpv#1#2{ (#1,#2) }
\def\xcmpAd{ \xcmpv{\cpA}{\cpd} }

\def\xcmpAdz{ \xcmpv{A}{\xbdz} }

\def\xcmpAdv{ \xcmpv{\yAvrt}{\yydvrt} }

\def\xcmqAd{ \xcmpv{\cpA}{\cqd} }

\def\yAv#1{ A_{#1} }
\def\yAvrt{ \yAv{\vert} }
\def\yydv#1{ d_{#1} }
\def\yydvrt{ \yydv{\vert} }
\def\yydvrtp{ \yydv{\vertp} }

\def\xdv#1{ d_{#1} }
\def\xdA{ \xdv{A} }
\def\xdB{ \xdv{B} }
\def\xdAv#1{ \xdv{A_{#1} } }

\def\xdAo{ \xdAv{1} }
\def\xdAt{ \xdAv{2} }

\def\cpfbvv#1#2{ \bigl( #1,#2 \bigr) }

\def\ttdg{total degree}

\def\tdgv#1{ |#1| }

\def\xbdv#1{ d_{(#1)} }
\def\xbdz{ \xbdv{0} }
\def\xbdo{ \xbdv{1} }
\def\xbdt{ \xbdv{2} }
\def\xbdi{ \xbdv{i} }
\def\xbdj{ \xbdv{j} }

\def\xbde{ d_{\edg} }

\def\xdot{ d^{\otimes} }
\def\xdotv#1{ \xdot_{(#1)} }
\def\xdoto{ \xdotv{1} }
\def\xdotz{ \xdotv{0} }

\def\xdote{ \xdot_{\edg} }

\def\dudvv#1#2{ d_{#2;#1} }

\def\dudvvpA{ \dudvv{\vert,\vertp}{A} }
\def\dudvvpB{ \dudvv{\vert,\vertp}{B} }

\def\fun#1{ \mathcal{#1} }
\def\fF{ \fun{F} }
\def\fR{ \fun{R} }
\def\fFtot{ \fF^{[\mathrm{tot}]} }
\def\fFtt{ \fF^{\mathrm{tot}} }
\def\fFtotD{ \fFtot_{\mcat{D} } }
\def\fFttD{ \fF^{\mathrm{tot}}_{\mcat{D} } }

\def\xCbv#1 { \sbsC^{#1} }

\def\xCbm{ \xCbv{\blbn} }
\def\xCbmp{ \xCbm }

\def\svbrcn{ \Prcnxv{\cgnstvert} }

\def\ydavv#1#2{ d_{#1,#2} }
\def\ydvvp{ \ydavv{\vert}{\vertp} }

\def\mtset#1{ \mathfrak{#1} }

\def\sbsG{ \mtset{G} }
\def\sbsF{ \mtset{F} }
\def\sbsX{ \mtset{X} }
\def\sbsC{ \mtset{C} }

\def\sbsGv#1{ \sbsG_{#1} }
\def\sbsFv#1{ \sbsF_{#1} }
\def\sbsXv#1{ \sbsX_{#1} }
\def\sbsGi{ \sbsGv{i} }
\def\sbsFi{ \sbsFv{i} }
\def\sbsXA{ \sbsXv{A} }
\def\sbsXAi{ \sbsXv{A;i} }

\def\cnst{constituent}
\def\cnstc{\cnst\ complex}
\def\tdmn{domain}

\def\zdvv#1#2{ d^{#1}_{#2} }
\def\zdotv#1{ \zdvv{#1}{12} }
\def\zdAot{ \zdotv{A} }
\def\zdBot{ \zdotv{B} }

\def\Acnv#1{ A_{#1} }
\def\Avr{\Acnv{\vert} }

\def\Bcnv#1{ B_{#1} }
\def\Bvr{\Bcnv{\vert} }

\def\imx{i_{\mathrm{max}} }
\def\imn{i_{\mathrm{min}} }

\def\xhen{ h }

\def\xhenv#1{ \xhen_{(#1)} }
\def\xheni{ \xhenv{i} }

\def\xhensv#1{ \xhen_{#1} }

\def\EnChA{ \End_{\ChFnA}}

\def\gnC{ \ccnbr }

\def\frF{ F }
\def\frFv#1{ \frF_{#1} }
\def\frFi{ \frFv{i} }

\def\yhm{ h }
\def\tdv{ \tdgv{\vert} }
\def\tdvp{ \tdgv{\vertp} }

\def\dgtt{ \deg_{\mathrm{tot} } }

\def\xpsi{ \psi }
\def\xpsiv#1{ \xpsi_{#1} }
\def\xpsim{ \xpsiv{-} }
\def\xpsip{ \xpsiv{+} }
\def\xxi{ \xi }
\def\xei{ \eta }
\def\xxit{ \tilde{\xxi}}
\def\xeit{ \tilde{\xei}}

\def\xkap{ \kappa }
\def\xkapt{ \tilde{\xkap} }

\def\prg{ p_{\circ}}
\def\xfxi{f_{\times\shortmid} }
\def\xfAas{ f_{A\ast} }
\def\xfasA{ f_{\ast A} }

\def\ftv#1{ f_{\mathrm{#1} } }
\def\flr{ \ftv{lr} }
\def\frl{ \ftv{rl} }

\def\fls{ \ftv{ls} }
\def\fsl{ \ftv{sl} }

\def\frs{ \ftv{rs} }
\def\fsr{ \ftv{sr} }

\def\cpbl{ [ }
\def\cpbr{ ] }
\def\Bcpbl{ \Bigl\cpbl }
\def\Bcpbr{ \Bigr\cpbr }
\def\bgcpbl{ \biggl\cpbl }
\def\bgcpbr{ \biggr\cpbr }

\def\dfbxyv#1{ p_#1(\bfy)-p_#1(\bfx) }

\def\qiso{quasi-isomorphism}
\def\qisc{quasi-isomorphic}

\def\xcpCs{ C_{\ast} }

\def\zzf{ f }
\def\zzg{ g }
\def\zzh{ h }

\def\zzfz{ \zeta }
\def\zzfu{ \zzfz_{\mathrm{u}} }
\def\zzfd{ \zzfz_{\mathrm{d}} }

\def\otdr{ \stackrel{\mathrm{L}}{\otimes} }

\def\otlgt{ \otdr_{\Qv{x_2,y_2}}\ctxtLbbxyt }

\def\dgMv#1{ \tensor*[_{y_{#1}}]{M}{^{\Delta}_{x_{#1}}} }
\def\dgMo{ \dgMv{1} }
\def\dgMbv#1#2 { \tensor*[_{#1}]{M}{_{#2}} }
\def\dgMbxy{ \dgMbv{\bfy}{\bfx} }

\def\zQz{ z\Qz }
\def\Qzz{\Qz/(z)}
\def\Qzzbr{ \bigl(\Qzz \bigr)  }
\def\Qclcz{ \saqt\Qz\oplus\Qz }
\def\Qclczbr{ \bigl(\Qclcz\bigr) }

\def\tinn{inner}
\def\tout{outer}

\def\xars{\ar}
\def\xari{\ar}

\def\ysgmm{\sigma^{-1}}
\def\ysgmmn{\sigma^{-n}}
\def\ysgmmtn{\sigma^{-2n}}
\def\ysgmmtno{\sigma^{-(2n+1)}}
\def\ysgmmtnmo{\sigma^{-2n+1}}

\def\ctsgmtn{ \xctv{\ysgmmtn} }
\def\ctsgmtno{ \xctv{\ysgmmtno} }

\def\xthet{ \theta }

\def\xdmz{ z }
\def\xdmw{ w }

\def\tdmmy{dummy} 

\def\shar{0.25em}
\def\lncf{0.4}

\def\lntx{0.3}
\def\lnty{0.3}
\def\lnte{-9}

\def\zvirtv{
\begin{tikzpicture} \draw + (1,0) -- ++(0,1);
\draw (0,0) -- ++(1,1);
\end{tikzpicture}
 }

\def\zvirtv#1#2{
 \begin{tikzpicture}[scale=#1,baseline=11*#1-#2*#1]
 \path[use as bounding box] (-0.1,-0.1) rectangle (1.1,1.1);
 \draw (1,0) to [crvbr]  (0,1);
 \draw (0,0) to [crvbr] (1,1);
\end{tikzpicture}
}

\def\zbrdpv#1#2{
\begin{tikzpicture}[scale=#1,baseline=11*#1-#2*#1]
\path[use as bounding box] (-0.1,-0.1) rectangle (1.1,1.1);
\draw (1,0) to [crvbr] (0,1);
\draw[line width=6pt, draw=white] (0,0) to [crvbr] (1,1);
\draw (0,0) to [crvbr] (1,1);
\end{tikzpicture}
}

\def\pcrcl{\draw (0,1) to [out=90,in=0] (-0.25,1.25) to [out=180, in=90] (-0.5,1)
to [out=270, in =90] (-0.5,0) to [out = -90, in = 180] (-0.25,-0.25) to [out=0,in=270] (0,0)}

\def\zbrdpcv#1#2{
\begin{tikzpicture}[scale=#1,baseline=11*#1-#2*#1]
\path[use as bounding box] (-0.75,-0.5) rectangle (1.1,1.5);
\draw (1,0) to [crvbr] (0,1);
\draw[line width=6pt, draw=white] (0,0) to [crvbr] (1,1);
\draw (0,0) to [crvbr] (1,1);
\pcrcl;
\end{tikzpicture}
}

\def\zbrdmcv#1#2{
\begin{tikzpicture}[scale=#1,baseline=11*#1-#2*#1]
\path[use as bounding box] (-0.75,-0.5) rectangle (1.1,1.5);
\draw (0,0) to [crvbr] (1,1);
\draw[line width=6pt, draw=white] (1,0) to [crvbr] (0,1);
\draw (1,0) to [crvbr] (0,1);
\pcrcl;
\end{tikzpicture}
}

\def\zdgrclv#1#2{
\begin{tikzpicture}[scale=#1,baseline=11*#1-#2*#1]
 \path[use as bounding box] (-0.75,-0.5) rectangle (1.1,1.5);
 \draw (1,0) to [crvbr]  (0,1);
 \draw (0,0) to [crvbr] (1,1);
 \draw [flwt] (0.5,0.5) circle [cirdgf];
 \pcrcl;
\end{tikzpicture}
}

\def\zparlclv#1#2{
 \begin{tikzpicture}[scale=#1,baseline=11*#1-#2*#1]
\path[use as bounding box] (-0.75,-0.5) rectangle (1.1,1.5);
\draw (1,0) to [crvbr] (0.8,0.5) to [crvbr] (1,1);
\draw (0,0) to [crvbr] (0.2,0.5) to [crvbr] (0,1);
\pcrcl;
\end{tikzpicture}
}

\def\zvirtclv#1#2{
 \begin{tikzpicture}[scale=#1,baseline=11*#1-#2*#1]
 \path[use as bounding box] (-0.75,-0.5) rectangle (1.1,1.5);
 \draw (1,0) to [crvbr]  (0,1);
 \draw (0,0) to [crvbr] (1,1);
 \pcrcl;
\end{tikzpicture}
}

\def\zblbbclv#1#2{
 \begin{tikzpicture}[scale=#1,baseline=11*#1-#2*#1]
 \path[use as bounding box] (-0.75,-0.5) rectangle (1.1,1.5);
 \draw (1,0) to [crvbr]  (0,1);
 \draw (0,0) to [crvbr] (1,1);
 \draw [fill] (0.5,0.5) circle [cirrad];
 \pcrcl;
\end{tikzpicture}
\end{tikzpicture}
}

\def\zbrdmv#1#2{
\begin{tikzpicture}[scale=#1,baseline=11*#1-#2*#1]
\path[use as bounding box] (-0.1,-0.1) rectangle (1.1,1.1);
\draw (0,0) to [crvbr] (1,1);
\draw[line width=6pt, draw=white] (1,0) to [crvbr] (0,1);
\draw (1,0) to [crvbr] (0,1);
\end{tikzpicture}
}

\def\zdgrfv#1#2{
 \begin{tikzpicture}[scale=#1,baseline=11*#1-#2*#1]
 \path[use as bounding box] (-0.1,-0.1) rectangle (1.1,1.1);
 \draw (1,0) to [crvbr]  (0,1);
 \draw (0,0) to [crvbr] (1,1);
 \draw [flwt] (0.5,0.5) circle [cirdgf];
\end{tikzpicture}
}

\def\zblbbv#1#2{
 \begin{tikzpicture}[scale=#1,baseline=11*#1-#2*#1]
 \path[use as bounding box] (-0.1,-0.1) rectangle (1.1,1.1);
 \draw (1,0) to [crvbr]  (0,1);
 \draw (0,0) to [crvbr] (1,1);
 \draw [fill] (0.5,0.5) circle [cirrad];
\end{tikzpicture}
}

\def\zblbgv#1#2{
 \begin{tikzpicture}[scale=#1,baseline=11*#1-#2*#1]
 \path[use as bounding box] (-0.1,-0.1) rectangle (1.1,1.1);
 \draw (1,0) to [crvbr]  (0,1);
 \draw (0,0) to [crvbr] (1,1);
 \draw [flgr] (0.5,0.5) circle [cirrad];
\end{tikzpicture}
}

\def\zblbgpv#1#2{
 \begin{tikzpicture}[scale=#1,baseline=11*#1-#2*#1]
 \path[use as bounding box] (-0.1,-0.1) rectangle (1.6,1.1);
 \draw (1,0) to [crvbr]  (0,1);
 \draw (0,0) to [crvbr] (1,1);
 \draw (1.5,0) to (1.5,1);
 \draw [flgr] (0.5,0.5) circle [cirrad];
\end{tikzpicture}
}

\def\zblbbpv#1#2{
 \begin{tikzpicture}[scale=#1,baseline=11*#1-#2*#1]
 \path[use as bounding box] (-0.1,-0.1) rectangle (1.6,1.1);
 \draw (1,0) to [crvbr]  (0,1);
 \draw (0,0) to [crvbr] (1,1);
 \draw (1.5,0) to (1.5, 1);
 \draw [fill] (0.5,0.5) circle [cirrad];
\end{tikzpicture}
}

\def\zlblbgpv#1#2#3{
\begin{tikzpicture}[xscale=#1,yscale=#2,baseline=11*#2-#3*#1]
\path[use as bounding box] (-0.2,-0.1) rectangle (2.2,2.1);
\draw (0,0) to [crvbr] (1,2);
\draw (1,0) to [crvbr] (0,2);
\draw (2,0) to [crvbr] (2,2);
\draw [flgr] (0.5,1) circle [cirrad];
\end{tikzpicture}
}

\def\zparlv#1#2{
 \begin{tikzpicture}[scale=#1,baseline=11*#1-#2*#1]
\path[use as bounding box] (-0.1,-0.1) rectangle (1.1,1.1);
\draw (1,0) to [crvbr] (0.8,0.5) to [crvbr] (1,1);
\draw (0,0) to [crvbr] (0.2,0.5) to [crvbr] (0,1);
\end{tikzpicture}
}

\def\zparlpv#1#2{
 \begin{tikzpicture}[scale=#1,baseline=11*#1-#2*#1]
\path[use as bounding box] (-0.1,-0.1) rectangle (1.6,1.1);
\draw (1,0) to [crvbr] (0.8,0.5) to [crvbr] (1,1);
\draw (0,0) to [crvbr] (0.2,0.5) to [crvbr] (0,1);
\draw (1.5,0) to (1.5,1);
\end{tikzpicture}
}

\def\zparlpv#1#2{
 \begin{tikzpicture}[scale=#1,baseline=11*#1-#2*#1]
\path[use as bounding box] (-0.1,-0.1) rectangle (1.6,1.1);
\draw (1,0) to [crvbr] (0.8,0.5) to [crvbr] (1,1);
\draw (0,0) to [crvbr] (0.2,0.5) to [crvbr] (0,1);
\draw (1.5,0) to (1.5,1);
\end{tikzpicture}
}

\def\ztparl#1#2#3{
\begin{tikzpicture}[xscale=#1,yscale=#2,baseline=11*#2-#3*#1]
\path[use as bounding box] (-0.2,-0.1) rectangle (2.2,2.1);
\draw (0,0) to [crvbr] (0,2);
\draw (1,0) to [crvbr] (1,2);
\draw (2,0) to [crvbr] (2,2);
\end{tikzpicture}
}

\def\ztvirp#1#2#3{
\begin{tikzpicture}[xscale=#1,yscale=#2,baseline=11*#2-#3*#1]
\path[use as bounding box] (-0.2,-0.1) rectangle (2.2,2.1);
\draw (0,0) to [crvbr] (1,2);
\draw (1,0) to [crvbr] (0,2);
\draw (2,0) to [crvbr] (2,2);
\end{tikzpicture}
}

\def\ztpvir#1#2#3{
\begin{tikzpicture}[xscale=#1,yscale=#2,baseline=11*#2-#3*#1]
\path[use as bounding box] (-0.2,-0.1) rectangle (2.2,2.1);
\draw (0,0) to [crvbr] (0,2);
\draw (1,0) to [crvbr] (2,2);
\draw (2,0) to [crvbr] (1,2);
\end{tikzpicture}
}

\def\ztrilo#1#2#3{
\begin{tikzpicture}[xscale=#1,yscale=#2,baseline=11*#2-#3*#1]
\path[use as bounding box] (-0.2,-0.1) rectangle (2.2,2.1);
\draw (0,0) to [crvbr] (1,2);
\draw (1,0) to [crvbr] (2,2);
\draw (2,0) to [crvbr] (0,2);
\end{tikzpicture}
}

\def\ztlelo#1#2#3{
\begin{tikzpicture}[xscale=#1,yscale=#2,baseline=11*#2-#3*#1]
\path[use as bounding box] (-0.2,-0.1) rectangle (2.2,2.1);
\draw (0,0) to [crvbr] (2,2);
\draw (1,0) to [crvbr] (0,2);
\draw (2,0) to [crvbr] (1,2);
\end{tikzpicture}
}

\def\ztvrtr#1#2#3{
\begin{tikzpicture}[xscale=#1,yscale=#2,baseline=11*#2-#3*#1]
\path[use as bounding box] (-0.2,-0.1) rectangle (2.2,2.1);
\draw (0,0) to [crvbr] (2,2);
\draw (1,0) to [crvbr] (1,2);
\draw (2,0) to [crvbr] (0,2);
\end{tikzpicture}
}

\def\ztvrtb#1#2#3{
\begin{tikzpicture}[xscale=#1,yscale=#2,baseline=11*#2-#3*#1]
\path[use as bounding box] (-0.2,-0.1) rectangle (2.2,2.1);
\draw (0,0) to [crvbr] (2,2);
\draw (1,0) to [crvbr] (1,2);
\draw (2,0) to [crvbr] (0,2);
\draw [fill] (1,1) circle [cirrad];
\end{tikzpicture}
}

\def\ztvrbb#1#2#3{
\begin{tikzpicture}[xscale=#1,yscale=#2,baseline=11*#2-#3*#1]
\path[use as bounding box] (-0.2,-0.1) rectangle (2.2,2.1);
\draw [name path=A] (0,0) to [out=25, in=245] (2,2);
\draw [name path=B] (1,0) to [out=140,in=-140] (1,2);
\draw [name path=C] (2,0) to [out=115, in=-25] (0,2);
\path [name intersections={of=A and B,by={AB}}];
\path [name intersections={of=B and C,by={BC}}];
\path [name intersections={of=A and C,by={AC}}];
\draw [fill] (AB) circle [cirrad];
\draw [fill] (AC) circle [cirrad];
\draw [fill] (BC) circle [cirrad];
\end{tikzpicture}
}

\def\ztvrbu#1#2#3{
\begin{tikzpicture}[xscale=#1,yscale=#2,baseline=11*#2-#3*#1]
\path[use as bounding box] (-0.2,-0.1) rectangle (2.2,2.1);
\draw [name path=A] (0,0) to [out=25, in=245] (2,2);
\draw [name path=B] (1,0) to [out=140,in=-140] (1,2);
\draw [name path=C] (2,0) to [out=115, in=-25] (0,2);
\path [name intersections={of=A and B,by={AB}}];
\path [name intersections={of=B and C,by={BC}}];
\path [name intersections={of=A and C,by={AC}}];
\draw [fill] (BC) circle [cirrad];
\end{tikzpicture}
}

\def\ztvrbd#1#2#3{
\begin{tikzpicture}[xscale=#1,yscale=#2,baseline=11*#2-#3*#1]
\path[use as bounding box] (-0.2,-0.1) rectangle (2.2,2.1);
\draw [name path=A] (0,0) to [out=25, in=245] (2,2);
\draw [name path=B] (1,0) to [out=140,in=-140] (1,2);
\draw [name path=C] (2,0) to [out=115, in=-25] (0,2);
\path [name intersections={of=A and B,by={AB}}];
\path [name intersections={of=B and C,by={BC}}];
\path [name intersections={of=A and C,by={AC}}];
\draw [fill] (AB) circle [cirrad];
\end{tikzpicture}
}

\def\ztlelob#1#2#3{
\begin{tikzpicture}[xscale=#1,yscale=#2,baseline=11*#2-#3*#1]
\path[use as bounding box] (-0.2,-0.1) rectangle (2.2,2.1);
\draw [name path=A] (0,0) to [crvbr] (2,2);
\draw [name path=B] (1,0) to [crvbr] (0,2);
\draw [name path=C] (2,0) to [crvbr] (1,2);
\path [name intersections={of=A and B, by={AB}}];
\path [name intersections={of=A and C, by={AC}}];
\draw [fill] (AB) circle [cirrad];
\draw [fill] (AC) circle [cirrad];
\end{tikzpicture}
}

\def\ztrilob#1#2#3{
\begin{tikzpicture}[xscale=#1,yscale=#2,baseline=11*#2-#3*#1]
\path[use as bounding box] (-0.2,-0.1) rectangle (2.2,2.1);
\draw [name path=B] (0,0) to [crvbr] (1,2);
\draw [name path=C] (1,0) to [crvbr] (2,2);
\draw [name path=A] (2,0) to [crvbr] (0,2);
\path [name intersections={of=A and B, by={AB}}];
\path [name intersections={of=A and C, by={AC}}];
\draw [fill] (AB) circle [cirrad];
\draw [fill] (AC) circle [cirrad];
\end{tikzpicture}
}

\def\zlblbbp#1#2#3{
\begin{tikzpicture}[xscale=#1,yscale=#2,baseline=11*#2-#3*#1]
\path[use as bounding box] (-0.2,-0.1) rectangle (2.2,2.1);
\draw (0,0) to [crvbr] (1,2);
\draw (1,0) to [crvbr] (0,2);
\draw (2,0) to [crvbr] (2,2);
\draw [fill] (0.5,1) circle [cirrad];
\end{tikzpicture}
}

\def\zlpblbb#1#2#3{
\begin{tikzpicture}[xscale=#1,yscale=#2,baseline=11*#2-#3*#1]
\path[use as bounding box] (-0.2,-0.1) rectangle (2.2,2.1);
\draw (0,0) to [crvbr] (0,2);
\draw (1,0) to [crvbr] (2,2);
\draw (2,0) to [crvbr] (1,2);
\draw [fill] (1.5,1) circle [cirrad];
\end{tikzpicture}
}

\def\zdblbbp#1#2#3{
\begin{tikzpicture}[xscale=#1,yscale=#2,baseline=11*#2-#3*#1]
\path[use as bounding box] (-0.2,-0.1) rectangle (2.2,2.1);
\draw [name path=A] (0,0) to [crvbr] (1,1);
\draw [name path=B] (1,0) to [crvbr] (0,1);
\draw [name path=C] (1,1) to [crvbr] (0,2);
\draw [name path=D] (0,1) to [crvbr] (1,2);
\draw (2,0) to (2,2);
\path [name intersections={of=A and B, by={AB}}];
\path [name intersections={of=D and C, by={DC}}];
\draw [fill] (AB) circle [cirrad];
\draw [fill] (DC) circle [cirrad];
\end{tikzpicture}
}

\def\sftparl{ \ztparl{\lntx}{\lnty}{\lnte} }

\def\sftpvir{ \ztpvir{\lntx}{\lnty}{\lnte} }

\def\sftrilo{ \ztprilo{\lntx}{\lnty}{\lnte} }

\def\sfcntparl{ \zectcnv{\ztparl{\lntx}{\lnty}{\lnte} } }
\def\sfcntvirp{ \zectcnv{\ztvirp{\lntx}{\lnty}{\lnte} } }
\def\sfcntpvir{ \zectcnv{\ztpvir{\lntx}{\lnty}{\lnte} } }
\def\sfcntlelo{ \zectcnv{\ztlelo{\lntx}{\lnty}{\lnte} } }
\def\sfcntrilo{ \zectcnv{\ztrilo{\lntx}{\lnty}{\lnte} } }
\def\sfcntvrtr{ \zectcnv{\ztvrtr{\lntx}{\lnty}{\lnte} } }
\def\sfcntvrtb{ \zectcnv{\ztvrtb{\lntx}{\lnty}{\lnte} } }

\def\sfcnlblbgp{ \zectcnv{\zlblbgpv{\lntx}{\lnty}{\lnte} } }

\def\sftparl{ \zectv{\ztparl{\lntx}{\lnty}{\lnte} } }

\def\sftpvir{ \zectv{\ztpvir{\lntx}{\lnty}{\lnte} } }

\def\sftrilo{ \zectv{\ztrilo{\lntx}{\lnty}{\lnte} } }

\def\sftvrtb{ \zectv{\ztvrtb{\lntx}{\lnty}{\lnte} } }

\def\sftlelob{ \zectv{\ztlelob{\lntx}{\lnty}{\lnte} } }
\def\sftrilob{ \zectv{\ztrilob{\lntx}{\lnty}{\lnte} } }
\def\sfdblbbp{ \zectv{\zdblbbp{\lntx}{\lnty}{\lnte} } }
\def\sflblbbp{ \zectv{\zlblbbp{\lntx}{\lnty}{\lnte} } }
\def\sflblbgp{ \zectv{\zlblbgpv{\lntx}{\lnty}{\lnte} } }
\def\sflpblbb{ \zectv{\zlpblbb{\lntx}{\lnty}{\lnte} } }

\def\sftvrbb{ \zectv{\ztvrbb{\lntx}{\lnty}{\lnte} } }
\def\sftvrbu{ \zectv{\ztvrbu{\lntx}{\lnty}{\lnte} } }
\def\sftvrbd{ \zectv{\ztvrbd{\lntx}{\lnty}{\lnte} } }

\def\dparl
{
\draw (1,0) to [crvbr] (0.8,0.5) to [crvbr] (1,1);
\draw (0,0) to [crvbr] (0.2,0.5) to [crvbr] (0,1);
}

\def\dvirt
{
\draw (1,0) to [crvbr]  (0,1);
\draw (0,0) to [crvbr] (1,1);
}

\def\dbrdp
{
\draw (1,0) to [crvbr] (0,1);
\draw[line width=6pt, draw=white] (0,0) to [crvbr] (1,1);
\draw (0,0) to [crvbr] (1,1);
}

\def\dbrdm
{
\draw (0,0) to [crvbr] (1,1);
\draw[line width=6pt, draw=white] (1,0) to [crvbr] (0,1);
\draw (1,0) to [crvbr] (0,1);
}

\def\dblbb
{
 \draw (1,0) to [crvbr]  (0,1);
 \draw (0,0) to [crvbr] (1,1);
 \draw [fill] (0.5,0.5) circle [cirrad];
}

\def\dblbg
{
 \draw (1,0) to [crvbr]  (0,1);
 \draw (0,0) to [crvbr] (1,1);
 \draw [flgr] (0.5,0.5) circle [cirrad];
}

\def\zvvv#1#2#3#4
{
 \begin{tikzpicture}[scale=#1,baseline=11*#1-#2*#1]
\path[use as bounding box] (-0.1,-0.1) rectangle (1.1,2.1);
#3
\begin{scope}[shift={(0,1)}]
#4
\end{scope}
\end{tikzpicture}
}

\def\zpvvv#1#2#3#4
{
 \begin{tikzpicture}[scale=#1,baseline=11*#1-#2*#1]
\path[use as bounding box] (-0.1,-0.1) rectangle (1.6,2.1);
#3
\begin{scope}[shift={(0,1)}]
#4
\end{scope}
\draw (1.5,0) to (1.5,2);
\end{tikzpicture}
}

\def\zzvv#1#2#3#4#5
{
 \begin{tikzpicture}[scale=#1,baseline=11*#1-#2*#1]
\path[use as bounding box] (-0.1,-0.1) rectangle (2.1,3.1);
#3
\begin{scope}[shift={(1,1)}]
#4
\end{scope}
\begin{scope}[shift={(0,2)}]
#5
\end{scope}
\draw (2,0) to (2,1);
\draw (0,1) to (0,2);
\draw (2,2) to (2,3);
\end{tikzpicture}
}

\def\zzmvv#1#2#3#4#5
{
 \begin{tikzpicture}[scale=#1,baseline=11*#1-#2*#1]
\path[use as bounding box] (-0.1,-0.1) rectangle (2.1,3.1);
\begin{scope}[shift={(1,0)}]
#3
\end{scope}
\begin{scope}[shift={(0,1)}]
#4
\end{scope}
\begin{scope}[shift={(1,2)}]
#4
\end{scope}
\draw (0,0) to (0,1);
\draw (2,1) to (2,2);
\draw (0,2) to (0,3);
\end{tikzpicture}
}

\def\ztwv#1#2#3#4#5
{
 \begin{tikzpicture}[scale=#1,baseline=11*#1-#2*#1]
\path[use as bounding box] (-0.1,-0.1) rectangle (1.1,3.1);
#3
\begin{scope}[shift={(0,1)}]
#4
\end{scope}
\begin{scope}[shift={(0,2)}]
#5
\end{scope}
\end{tikzpicture}
}

\def\zbrpbrm#1#2{ \zvvv{#1}{#2}{\dbrdp}{\dbrdm} }
\def\zbrmbrp#1#2{ \zvvv{#1}{#2}{\dbrdm}{\dbrdp} }
\def\zbrmbrm#1#2{ \zvvv{#1}{#2}{\dbrdm}{\dbrdm} }

\def\zbrmbrmp#1#2{ \zpvvv{#1}{#2}{\dbrdm}{\dbrdm} }

\def\zblbblg#1#2{ \zvvv{#1}{#2}{\dblbb}{\dblbg} }
\def\zblgblb#1#2{ \zvvv{#1}{#2}{\dblbg}{\dblbb} }
\def\zblbblb#1#2{ \zvvv{#1}{#2}{\dblbb}{\dblbb} }

\def\zparblb#1#2{ \zvvv{#1}{#2}{\dparl}{\dblbb} }
\def\zparblg#1#2{ \zvvv{#1}{#2}{\dparl}{\dblbg} }
\def\zblgpar#1#2{ \zvvv{#1}{#2}{\dblbg}{\dparl} }
\def\zblbpar#1#2{ \zvvv{#1}{#2}{\dblbb}{\dparl} }

\def\zbrmblb#1#2{ \zvvv{#1}{#2}{\dbrdm}{\dblbb} }
\def\zbrmblg#1#2{ \zvvv{#1}{#2}{\dbrdm}{\dblbg} }
\def\zbrmpar#1#2{ \zvvv{#1}{#2}{\dbrdm}{\dparl} }

\def\zparpar#1#2{ \zvvv{#1}{#2}{\dparl}{\dparl} }

\def\zblbvir#1#2{ \zvvv{#1}{#2}{\dblbb}{\dvirt} }
\def\zvirblb#1#2{ \zvvv{#1}{#2}{\dvirt}{\dblbb} }

\def\zblgvir#1#2{ \zvvv{#1}{#2}{\dblbg}{\dvirt} }
\def\zvirblg#1#2{ \zvvv{#1}{#2}{\dvirt}{\dblbg} }

\def\zvirbrpvir#1#2{ \ztwv{#1}{#2}{\dvirt}{\dbrdp}{\dvirt} }
\def\zvirbrmvir#1#2{ \ztwv{#1}{#2}{\dvirt}{\dbrdm}{\dvirt} }

\def\zverlv#1#2{
\begin{tikzpicture}[scale=#1,baseline=11*#1-#2*#1]
 \path[use as bounding box] (-0.1,-0.1) rectangle (0.2,1.1);
 \draw (0.,0) -- ++(0,1);
\end{tikzpicture}
}

\def\cblth{1.3pt}
\def\ljwp{1pt}

\tikzset{ptzer/.style={line width=\ljwp} }
\tikzset{ptzert/.style={line width=\ljwp,fill=white}}
\tikzset{ptone/.style={line width=\ljwp,fill=gray!30}}
\tikzset{pttwo/.style={line width=\ljwp,fill=white,pattern=horizontal lines} }
\tikzset{ptthr/.style={line width=\ljwp,pattern=north east lines} }
\tikzset{ptfour/.style={line width=\ljwp,pattern=crosshatch} }
\tikzset{vthln/.style={line width=\ljwp} }

\tikzset{smmt/.style={matrix of math nodes, row sep=2em,
column sep=1.5em, text height=1.5ex, text depth=0.25ex} }

\tikzset{stmt/.style={matrix of math nodes, row sep=3em,
column sep=1.5em, text height=1.5ex, text depth=0.25ex} }

\tikzset{ssmt/.style={matrix of math nodes, row sep=3em,
column sep=3em, text height=1.5ex, text depth=0.25ex} }

\tikzset{ssmto/.style={matrix of math nodes, row sep=3em,
column sep=1.5em, text height=1.5ex, text depth=0.25ex} }

\tikzset{ssmtt/.style={matrix of math nodes, row sep=2em,
column sep=3.5em, text height=1.5ex, text depth=0.25ex} }

\tikzset{ssmth/.style={matrix of math nodes, row sep=3em,
column sep=4.5em, text height=2.5ex, text depth=1.25ex} }

\tikzset{ssmtf/.style={matrix of math nodes, row sep=3em,
column sep=3em, text height=2.5ex, text depth=1.25ex} }

\tikzset{ssmts/.style={matrix of math nodes, row sep=2em,
column sep=4.5em, text height=2.5ex, text depth=1.25ex} }

\tikzset{sxmt/.style={matrix of math nodes, row sep=4em,
column sep=3em, text height=3.5ex, text depth=2.25ex} }

\tikzset{ampre/.style={ampersand replacement=\&}}

\tikzset{menvone/.style={scale=0.5,baseline=-1.5} }
\tikzset{menvtwo/.style={scale=0.70,baseline=-1.5} }
\tikzset{menvthree/.style={scale=0.35, baseline=-1.5} }
\tikzset{menvfour/.style={scale=0.25, baseline=-1.5} }
\tikzset{menvzer/.style={scale=1,baseline=-1.5} }

\tikzset{menvrtone/.style={scale=0.45,rotate=-45,baseline=-3.5} }
\tikzset{menvrtsone/.style={scale=0.65,rotate=-45,baseline=-3.5} }
\tikzset{menvrtonep/.style={scale=0.5,baseline=-3.5} }

\tikzset{lnovr/.style={line width=6pt,color=white}}
\tikzset{lnovrtw/.style={line width=4pt,color=white}}
\tikzset{thkln/.style={line width=\cblth} }
\tikzset{thnln/.style={} }
\tikzset{thkc/.style={line width=0.6pt} }
\tikzset{bcrc/.style={line width=0.2cm,color=gray!40,opacity=0.5} }
\tikzset{bcrct/.style={line width=0.25cm,color=gray!20} }

\tikzset{crvbr/.style={out=90,in=270} }
\tikzset{cirrad/.style={radius=0.2}}
\tikzset{cirdgf/.style={radius=0.35}}

\tikzset{flgr/.style={fill=gray!30}}
\tikzset{flwt/.style={fill=white}}

\pgfdeclaredecoration{ignore}{final}
{
\state{final}{}
}

\pgfdeclaremetadecoration{middle}{initial}{
    \state{initial}[
        width={(\pgfmetadecoratedpathlength - \the\pgfdecorationsegmentlength)/2},
        next state=middle
    ]
    {\decoration{moveto}}

    \state{middle}[
        width={\the\pgfdecorationsegmentlength},
        next state=final
    ]
    {\decoration{curveto}}

    \state{final}
    {\decoration{ignore}}
}

\tikzset{middle segment/.style={decoration={middle},decorate, segment length=#1}}

\def\sfverlbxyo{ \xctvbyxo{ \zverlv{\lncf}{5}}}

\def\sfbrdcpyxt{\yectvyxt{\zbrdpcv{\lncf}{5}} }
\def\sfbrdcmyxt{\yectvyxt{\zbrdmcv{\lncf}{5}} }

\def\sfbrdmclxyt{\yectvxyt{\zbrdmcv{\lncf}{5} }}
\def\sfdgrclxyt{\yectvxyt{\zdgrclv{\lncf}{5}} }
\def\sfdgrfxy{\xctvbxy{\zdgrfv{\lncf}{5}} }
\def\sfcdgrfxy{\xctcnvbyx{\zdgrfv{\lncf}{5}} }
\def\sverlxyt{ \xctvbyxt{ \zverlv{\lncf}{5}}}

\def\sfvirtclxyt{ \yectvxyt{\zvirtclv{\lncf}{5}} }
\def\sfparlclxyt{ \yectvxyt{\zparlclv{\lncf}{5}} }

\def\sfvirtcl{ \yectv{\zvirtclv{\lncf}{5}} }
\def\sfparlcl{ \yectv{\zparlclv{\lncf}{5}} }

\def\sfbrdpcl{ \yectv{\zbrdpcv{\lncf}{5}} }
\def\sfbrdmcl{ \yectv{\zbrdmcv{\lncf}{5}} }

\def\sfblbb{\xctv{\zblbbv{\lncf}{5}} }
\def\sfblbg{\xctv{\zblbgv{\lncf}{5}} }
\def\sfvirt{\xctv{\zvirtv{\lncf}{5}} }
\def\sfparl{\xctv{\zparlv{\lncf}{5}} }
\def\sfverl{\xctv{\zverlv{\lncf}{5}} }
\def\sfbrdp{\xctv{\zbrdpv{\lncf}{5}} }

\def\sfbrdm{\xctv{\zbrdmv{\lncf}{5}} }

\def\sfblbgp{\xctv{\zblbgpv{\lncf}{5}} }
\def\sfblbbp{\xctv{\zblbbpv{\lncf}{5}} }
\def\sfparlp{\xctv{\zparlpv{\lncf}{5}} }

\def\sflcnblbgp{\zectcnv{\zlblbgpv{\lntx}{\lnty}{\lnte}} }

\def\sfmparl{\xctmv{\zparlv{\lncf}{5}} }
\def\sfmvirt{\xctmv{\zvirtv{\lncf}{5}} }
\def\sfmdgrf{\xctmv{\zdgrfv{\lncf}{5}} }
\def\sfmblbb{\xctmv{\zblbbv{\lncf}{5}} }
\def\sfmblbg{\xctmv{\zblbgv{\lncf}{5}} }

\def\sfnblbb{\xctcnv{\zblbbv{\lncf}{5}} }
\def\sfnblbg{\xctcnv{\zblbgv{\lncf}{5}} }
\def\sfnvirt{\xctcnv{\zvirtv{\lncf}{5}} }
\def\sfnparl{\xctcnv{\zparlv{\lncf}{5}} }

\def\sfndgrf{\xctcnv{\zdgrfv{\lncf}{5}} }

\def\sfbrpbrm{\yectv{\zbrpbrm{\lncf}{-11}} }
\def\sfbrmbrp{\yectv{\zbrmbrp{\lncf}{-11}} }
\def\sfbrmbrm{\yectv{\zbrmbrm{\lncf}{-11}} }
\def\sfblbblg{\yectv{\zblbblg{\lncf}{-11}} }
\def\sfblgblb{\yectv{\zblgblb{\lncf}{-11}} }
\def\sfblbblb{\yectv{\zblbblb{\lncf}{-11}} }

\def\sfparblb{\yectv{\zparblb{\lncf}{-11}} }
\def\sfparblg{\yectv{\zparblg{\lncf}{-11}} }
\def\sfblgpar{\yectv{\zblgpar{\lncf}{-11}} }
\def\sfblbpar{\yectv{\zblbpar{\lncf}{-11}} }

\def\sfbrmblb{\yectv{\zbrmblb{\lncf}{-11}} }
\def\sfbrmblg{\yectv{\zbrmblg{\lncf}{-11}} }
\def\sfbrmpar{\yectv{\zbrmpar{\lncf}{-11}} }

\def\sfparpar{\yectv{\zparpar{\lncf}{-11}} }

\def\sfcnbrpbrm{\yectcnv{\zbrpbrm{\lncf}{-11}} }
\def\sfcnblgblb{\yectcnv{\zblgblb{\lncf}{-11}} }
\def\sfcnparblb{\yectcnv{\zparblb{\lncf}{-11}} }
\def\sfcnparpar{\yectcnv{\zparpar{\lncf}{-11}} }
\def\sfcnblgpar{\yectcnv{\zblgpar{\lncf}{-11}} }

\def\sfcnblbblb{\yectcnv{\zblbblb{\lncf}{-11}} }

\def\sfbrmbrmp{\yectv{\zbrmbrmp{\lncf}{-11}} }

\def\sfvervv#1#2{ \xctvv{\zverlv{\lncf}{5}}{#1}{#2} }
\def\sfveryx{ \sfvervv{y}{x} }

\def\sthmmmo{ \yectv{\zzvv{\lncf}{-25}{\dbrdm}{\dbrdm}{\dbrdm}} }
\def\sthbcn{ \yectcnv{\zzvv{\lncf}{-25}{\dblbb}{\dblbb}{\dblbb} } }
\def\sthbbpcn{ \yectcnv{\zzvv{\lncf}{-25}{\dblbb}{\dparl}{\dblbb} } }
\def\sthbbvcn{ \yectcnv{\zzvv{\lncf}{-25}{\dblbb}{\dvirt}{\dblbb} } }

\def\sthmmb{ \yectv{\zzvv{\lncf}{-25}{\dbrdm}{\dblbb}{\dbrdm} } }
\def\sthmmp{ \yectv{\zzvv{\lncf}{-25}{\dbrdm}{\dparl}{\dbrdm} } }

\numberwithin{equation}{section}

\title[Virtual crossings and a filtration of the triply graded homology]
{Virtual crossings and a filtration of the triply graded homology of a link diagram}

\author[M.~Abel]{Michael Abel}
\address{
M.~Abel\\
Department of Mathematics and Applied Mathematics\\
Virginia Commonwealth University\\
Harris Hall 4103\\
Richmond, VA 23284}
\email{maabel@vcu.edu}
\author[L.~Rozansky]{Lev Rozansky}
\address{
L.~Rozansky\\
Department of Mathematics\\
University of North Carolina at Chapel Hill\\
CB \# 3250, Phillips Hall\\
Chapel Hill, NC 27599
}
\email{rozansky@math.unc.edu}
\thanks{The work of L.R. was supported in part by the NSF grant DMS-1108727}

\input xy.sty
\ifx\xyloaded\undefined \input xy \fi
\xyprovide{arrow}{Arrow and Path feature}{\stripRCS$Revision: 3.5 $}%
 {Kristoffer H.~Rose}{krisrose@brics.dk}%
 {BRICS/Computer Science, University of Aarhus, Ny Munkegade, building 540,
 DK--8000 Aarhus~C}
\xydef@\PATH{\afterPATH{}}
\xylet@\afterPATH@=\empty
\xydef@\afterPATH#1{\save
 \DN@##1{\def\afterPATH@{\restore \def\afterPATH@{##1}#1}}%
 \expandafter\next@\expandafter{\afterPATH@}%
 \let\PATHfail@@=\empty
 \let\PATHbefore@@=\empty
 \let\PATHafter@@=\empty
 \let\PATHlabelsevery@@=\empty
 \let\PATHlabelsnext@@=\empty
 \let\PATHlabelslast@@=\empty
 \xy@{\afterPATH{#1}}{\def\PATHslide@@{\z@}}%
 \def\PATHcontinue@@{\xyFN@\PATH@}%
 \xyFN@\PATH@}
\xynew@{if}\ifPATHsingle@
\xydef@\PATH@{%
 \ifx \space@\next \expandafter\DN@\space{\xyFN@\PATH@}%
 \else\ifPATHsingle@ \let\next@=\PATH@single
 \else \let\next@=\PATH@normal \fi\fi \next@}
\xydef@\PATH@normal{%
 \ifx ~\next \DN@ ~{\xyFN@\PATHsetting@}%
 \else \addRQ@\ifx \next \addRQ@\DN@{\xy@'{}\xyFN@\PATHstraight@}%
 \else \addLQ@\ifx \next \addLQ@\DN@{\xy@`{}\xyFN@\PATHturn@}%
 \else\ifx \PATHfail@@\PATH@x \DN@{\xyFN@\PATHfinal@}%
 \else
 \DNii@{\let\PATHfail@@=\PATH@x \xyFN@\PATH@}%
 \DN@{\expandafter\nextii@\PATHfail@@}%
 \fi\fi\fi\fi \next@}
\xylet@\PATHfail@@=\empty
\xylet@\PATHbefore@@=\empty
\xylet@\PATHafter@@=\empty
\xylet@\PATHlabelsevery@@=\empty
\xylet@\PATHlabelsnext@@=\empty
\xylet@\PATHlabelslast@@=\empty
\xydef@\PATHsetting@{%
 \ifx \space@\next \expandafter\DN@\space{\xyFN@\PATHsetting@}%
 \else\ifx \bgroup\next
 \DN@##1{\xy@{~{##1}}{}\def\PATHfail@@{##1}\xyFN@\PATH@}%
 \else \addEQ@\ifx \next
 \addEQ@\DN@##1{\xy@{~={##1}}{}\def\PATHbefore@@{##1}\xyFN@\PATH@}%
 \else \addLT@\ifx \next
 \addLT@\DN@##1{\xy@{~<{##1}}{}\def\PATHlabelsnext@@{##1}\xyFN@\PATH@}%
 \else \addGT@\ifx \next
 \addGT@\DN@##1{\xy@{~>{##1}}{}\def\PATHlabelslast@@{##1}\xyFN@\PATH@}%
 \else \addPLUS@\ifx \next
 \addPLUS@\DN@##1{\xy@{~+{##1}}{}\def\PATHlabelsevery@@{##1}\xyFN@\PATH@}%
 \else \ifx /\next
 \DN@/##1{\xy@{~/{##1}}{}\def\PATHafter@@{##1}\xyFN@\PATH@}%
 \else \addDASH@\ifx \next
 \xywarning@{Obsolete `-' PATH <action> translated to `='.}%
 \addDASH@\DN@##1{\xy@{~-{##1}}{}\def\PATHbefore@@{##1}\xyFN@\PATH@}%
 \else
 \xyerror@{Unknown \string~ setting: \meaning\next}%
 {See the Xy-pic arrow feature documentation for help.}%
 \fi\fi\fi\fi\fi\fi\fi\fi \next@}
\xylet@\PATHinit@@=\empty
\xylet@\PATHextra@@=\empty
\xylet@\PATHpost@@=\empty
\xylet@\PATHcontinue@@=\empty
\xydef@\PATHstraight@{%
 \def\PATHinit@@{\PATHinitstraight@}%
 \let\PATHextra@@=\empty
 \let\PATHpost@@=\empty
 \let\PATHlabelsextra@@=\relax
 \xy@@\pfromc@ \PATHafterPOS{\xyFN@\PATHsegment@}}
\xydef@\PATHfinal@{%
 \def\PATHinit@@{\PATHinitstraight@}%
 \def\PATHextra@@{\let\PATHcontinue@@=\afterPATH@}%
 \let\PATHpost@@=\empty
 \let\PATHlabelsextra@@=\PATHlabelsextralast@
 \xy@@\pfromc@ \PATHafterPOS{\xyFN@\PATHsegment@}}
\xydef@\PATHlabelsextralast@{\let\PATHlabelsextra@@=\relax
 \expandafter\xyFN@\expandafter\PATHlabels@\PATHlabelslast@@}
\xydef@\PATHinitstraight@{\xy@@{\setupDirection@ \dimen@=\PATHslide@@
 \dimen@ii=-\sinDirection\dimen@
 \ifPATHomitslide@@\else\advance\X@p\dimen@ii\fi \advance\X@c\dimen@ii
 \dimen@ii= \cosDirection\dimen@
 \ifPATHomitslide@@\else\advance\Y@p\dimen@ii\fi \advance\Y@c\dimen@ii
 \PATHomitslide@@false \resetupDirection@}}
\xydef@\PATHslide@@{\z@}
\xynew@{if}\ifPATHomitslide@@
\xydef@\PATHsegment@{%
 \addLT@\ifx\next
 \addGT@{\addLT@\DN@##1}{%
 \xy@{<##1>}{\dimen@=##1\relax \edef\PATHslide@@{\the\dimen@}}%
 \xyFN@\PATHsegment@@}%
 \else \let\next@=\PATHsegment@@
 \fi \next@}
\xydef@\PATHsegment@@{\PATHinit@@ \addEQ@\PATHaction@\PATHbefore@@ \PATHextra@@
 \expandafter\toks@\expandafter{\PATHlabelsnext@@}\let\PATHlabelsnext@@=\empty
 \expandafter\addtotoks@\expandafter{\PATHlabelsevery@@}%
 \expandafter\def\expandafter\PATHlabels@@\expandafter{\the\toks@}%
 \toks@={}\expandafter\xyFN@\expandafter\PATHlabels@\PATHlabels@@}
\xylet@\PATHlabelit@@=\empty
\xydef@\PATHlabels@{%
 \ifx \space@\next \expandafter\DN@\space{\xyFN@\PATHlabels@}%
 \else \ifx ^\next 
 \DN@##1{\xy@^{}\let\PATHlabelit@@=\PATHlabelabove@@
 \DNii@{}\xyFN@\PATHanchor@}%
 \else \ifx _\next
 \DN@##1{\xy@_{}\let\PATHlabelit@@=\PATHlabelbelow@@
 \DNii@{}\xyFN@\PATHanchor@}%
 \else \ifx |\next
 \DN@##1{\xy@|{}\let\PATHlabelit@@=\PATHlabelbreak@
 \DNii@{}\xyFN@\PATHanchor@}%
 \else \let\next@=\PATHfinishsegments@ \fi\fi\fi\fi \next@}
\xydef@\PATHanchor@{\begingroup \toks@={}\PATHanchor@i}
\xydef@\PATHanchor@i{%
 \ifx \space@\next \expandafter\DN@\space{\xyFN@\PATHanchor@i}%
 \else\addDASH@\ifx \next
 \addDASH@\DN@{\expandafter\addtotoks@\expandafter{\PATHanchor@toks}%
 \xyFN@\PATHanchor@i}%
 \else
 \DNii@##1{\endgroup\afterPLACE{\xyFN@\PATHit@}##1}%
 \DN@{\expandafter\nextii@\expandafter{\the\toks@}}%
 \fi\fi \next@}
{\xyuncatcodes \gdef\next{<>(.5)}}
\xylet@\PATHanchor@toks=\next
\xylet@\PATHitshapes@@=\empty
\xydef@\PATHit@{\let\PATHitshape@@=\empty \xyFN@\PATHit@i}
\xydef@\PATHit@i{%
 \ifx \space@\next \expandafter\DN@\space{\xyFN@\PATHit@i}%
 \else \ifx *\next \DN@*##1##{\PATHlabelit@@{!C##1}}%
 \else \addAT@\ifx\next \addAT@\DN@##1##{\PATHlabelit@@{\dir##1}}%
 \else \ifx [\next \DN@[##1]{%
 \expandafter\def\expandafter\PATHitshape@@\expandafter{\PATHitshape@@[##1]}%
 \xyFN@\PATHit@i}%
 \else \DN@{\PATHlabelit@@{}}%
 \fi\fi\fi\fi \next@}
\xydef@\PATHfinishsegments@{%
 \ifx\PATHlabelsextra@@\relax \expandafter\PATHfinishsegments@i
 \else
 \expandafter\PATHlabelsextra@@ \fi}
\xydef@\PATHfinishsegments@i{%
 \xy@@{\Clast@@}\the\toks@ \toks@={}%
 \xy@@\setupDirection@
 \PATHpost@@ \PATHaction@/\PATHafter@@
 \PATHcontinue@@}
\xydef@\PATHaction@default#1#2{\xy@{PATHaction#1{#2}}{}\POS#2\relax}
\xylet@\PATHaction=\PATHaction@default
\xydef@\PATHaction@#1#2{\expandafter\PATHaction\expandafter#1\expandafter{#2}}
\xylet@\PATHafterPOS@default=\afterPOS
\xylet@\PATHafterPOS=\PATHafterPOS@default
\xydef@\PATHlastout@@{3}
\xydef@\PATHturn@{\afterCIRorDIAG\PATHturn@cir\PATHturn@diag}
\xydef@\PATHturn@cir{\toks@={\xy@@{%
 \count@=\CIRin@@ \ifnum\count@<4\else\advance\count@-4\fi
 \count@@=\CIRout@@ \ifnum\count@@<4\else\advance\count@@-4\fi
 \ifnum\count@=\count@@ \xyerror@{<turn> cannot be half or full}{%
You asked for a <turn>ed segment with parallel in- and out-direction.^^J%
This is not allowed because it is not possible to position it uniquely.}\fi}}%
 \edef\next@{{\CIRin@@}{\expandafter\noexpand\CIRorient@@}{\CIRout@@}}%
 \expandafter\PATHturn@i\next@}
\xydef@\PATHturn@diag{\toks@={\xy@@{%
 \setupDirection@ \count@=\CIRin@@ \dimen@=\xydashl@ \ABfromdiag@
 \ifdim \sinDirection\A@>\cosDirection\B@ \def\CIRorient@@{\CIRacw@}%
 \advance\count@\ifnum\count@<6 \tw@ \else -6\fi
 \else \def\CIRorient@@{\CIRcw@}%
 \advance\count@\ifnum\count@>\@ne -\tw@ \else 6\fi
 \fi
 \edef\CIRout@@{\the\count@}}}%
 \edef\next@{{\CIRin@@}{}{}}%
 \expandafter\PATHturn@i\next@}
\xylet@\PATHpostpos@@=\empty
\xydef@\PATHturn@i#1#2#3{%
 \DN@##1{\def\PATHinit@@{\xy@@{%
 \def\CIRin@@{#1}\def\CIRorient@@{#2}\def\CIRout@@{#3}%
 \ifnum\CIRin@@=8 \let\CIRin@@=\PATHlastout@@ \fi
 \R@=\turnradius@
 ##1\relax}%
 \xy@@{\count@=\CIRin@@
 \ifPATHomitslide@@ \dimen@=-\PATHslide@@
 \ABfromdiag@ \advance\X@p-\B@ \advance\Y@p\A@ \fi
 \enter@{\pfromthep@ \basefromthebase@}}%
 \xy@@{\dimen@=\expandafter\ifx\CIRorient@@\CIRcw@-\fi\R@
 \ABfromdiag@ \advance\X@p-\B@ \advance\Y@p\A@
 \X@origin=\X@p \Y@origin=\Y@p}%
 \xy@@{\dimen@=\xydashl@ \ABfromdiag@ \R@c=\A@ \U@c=\B@
 \count@=\CIRout@@
 \dimen@=\expandafter\ifx\CIRorient@@\CIRcw@-\fi\R@
 \ABfromdiag@ \advance\X@c-\B@ \advance\Y@c\A@
 \dimen@=\xydashl@ \X@p=\X@c \Y@p=\Y@c
 \ABfromdiag@ \advance\X@p-\A@ \advance\Y@p-\B@
 \intersect@
 \count@=\CIRin@@ \dimen@=\p@ \ABfromdiag@}%
 \xy@@{\edef\next@{\expandafter\removePT@\the\A@}%
 \edef\nextii@{\expandafter\removePT@\the\B@}%
 \A@=\X@c \advance\A@-\X@origin \B@=\Y@c \advance\B@-\Y@origin
 \ifdim \next@\A@<-\nextii@\B@
 \dontleave@ \cfromp@
 \count@=\CIRin@@ \advance\count@\ifnum\count@<4 +4\else -4\fi
 \dimen@=\xydashl@ \Directionfromdiag@
 \the\Edge@c\z@
 \count@=\CIRin@@ \dimen@=\expandafter\ifx\CIRorient@@\CIRcw@-\fi\R@
 \ABfromdiag@ \advance\X@c-\B@ \advance\Y@c\A@
 \dimen@=5sp \ABfromdiag@ \advance\X@c\A@ \advance\Y@c\B@
 \fi
 \dimen@ii=\expandafter\ifx\CIRorient@@\CIRacw@-\fi\PATHslide@@
 \advance\R@\dimen@ii
 \drop@\literal@{\hbox\bgroup\cir@i}}%
 \xy@@{\X@p=\X@c \Y@p=\Y@c
 \count@=\CIRout@@
 \dimen@=\expandafter\ifx\CIRorient@@\CIRacw@-\fi\R@c
 \ABfromdiag@ \advance\X@p-\B@ \advance\Y@p\A@
 \edef\PATHpostpos@@{\X@c=\the\X@p \Y@c=\the\Y@p \noexpand\czeroEdge@
 \noexpand\PATHomitslide@@true}}%
 \xy@@{\count@=\CIRin@@
 \dimen@=\expandafter\ifx\CIRorient@@\CIRacw@-\fi\R@c
 \ABfromdiag@ \advance\X@c-\B@ \advance\Y@c\A@ \czeroEdge@}%
 \xy@@{\leave@
 \count@=\CIRin@@ \dimen@=\PATHslide@@
 \ABfromdiag@ \advance\X@p-\B@ \advance\Y@p\A@
 \edef\PATHlastout@@{\CIRout@@}%
 \count@=\CIRout@@ \dimen@=\xydashl@ \Directionfromdiag@}}}%
 \expandafter\next@\expandafter{\the\toks@}\toks@={}%
 \let\PATHextra@@=\empty
 \def\PATHpost@@{\xy@@\PATHpostpos@@}%
 \let\PATHlabelsextra@@=\relax
 \xyFN@\PATHturn@ii}
\xydef@\PATHturn@ii{%
 \ifx /\next
 \DN@ /{\afterassignment\nextii@\dimen@=}%
 \DNii@{%
 \edef\next@{\noexpand\xy@@{\edef\noexpand\turnradius@{\the\dimen@}}}\next@
 \xy@@\pfromc@ \PATHafterPOS{\xyFN@\PATHsegment@}}%
 \else \DN@{\xy@@\pfromc@ \PATHafterPOS{\xyFN@\PATHsegment@}}\fi
 \next@}
\xydef@\ABfromdiag@{\ifcase\count@\relax
 \A@=-.7071\dimen@ \B@=-.7071\dimen@ \or \A@=\z@ \B@=-\dimen@
 \or \A@=+.7071\dimen@ \B@=-.7071\dimen@ \or \A@=\dimen@ \B@=\z@
 \or \A@=+.7071\dimen@ \B@=+.7071\dimen@ \or \A@=\z@ \B@=+\dimen@
 \or \A@=-.7071\dimen@ \B@=+.7071\dimen@ \or \A@=-\dimen@ \B@=\z@
 \else\xybug@{impossible <diag>?}\fi}
\xydef@\turnradius@{10pt}
\xydef@\turnradius{\afterADDOP{\Addop@@\turnradius@}}
\xydef@\PATHlabelabove@#1#2{\droplabel@\belowDirection@{#1}{#2}%
 \let\afteraliases@@=\empty \xyFN@\PATHlabelalias@}
\xydef@\PATHlabelbelow@#1#2{\droplabel@\aboveDirection@{#1}{#2}%
 \let\afteraliases@@=\empty \xyFN@\PATHlabelalias@}
\xylet@\PATHlabelabove@@=\PATHlabelabove@
\xylet@\PATHlabelbelow@@=\PATHlabelbelow@
\xydef@\droplabel@#1#2#3{\xy@@{\enter@\DirectionfromtheDirection@}%
 \DN@{#2}\ifx\next@\empty
 \expandafter\xy@@ix@\expandafter{\PATHitshape@@\labelbox{#3}}%
 \else \expandafter\xy@@ix@\expandafter{\PATHitshape@@#2{#3}}\fi
 \xy@@{\setbox\z@=\expandafter\object\the\toks9
 \advance\L@c\labelmargin@ \advance\R@c\labelmargin@
 \advance\D@c\labelmargin@ \advance\U@c\labelmargin@ 
 \setboxz@h{\kern\labelmargin@\boxz@\kern\labelmargin@}%
 \A@=\X@c \B@=\Y@c #1\xydashl@ \the\Edge@c\thr@@ 
 \advance\A@-\X@c \advance\B@-\Y@c \advance\X@c 2\A@ \advance\Y@c 2\B@
 \ht\z@=\U@c \dp\z@=\D@c \dimen@=\L@c \advance\dimen@\R@c \wd\z@=\dimen@ 
 \setbox\lastobjectbox@=\hbox{\X@c=\L@c \Y@c=\z@ \Drop@@}%
 \drop@{\box\lastobjectbox@}{}%
 \leave@}}
\xydef@\PATHlabelbreak@#1#2{%
 \DN@{#1}\ifx\next@\empty \drop\labelbox{#2}\else \drop#1{#2}\fi
 \def\afteraliases@@{\xy@@\Cbreak@@}\xyFN@\PATHlabelalias@}
\xylet@\labelmargin@=\objectmargin@
\xydef@\labelmargin{\afterADDOP{\Addop@@\labelmargin@}}
\xydef@\hole{\hbox{\dimen@=\objectmargin@ \kern2\dimen@
 \vrule height\dimen@ depth\dimen@ width\z@}}
\xylet@\labelstyle=\scriptstyle
\xydef@\labelbox#1{\hbox{$\m@th\labelstyle{#1}$}}
\xydef@\PATHlabelalias@{%
 \ifx \space@\next \expandafter\DN@\space{\xyFN@\PATHlabelalias@}%
 \else \addEQ@\ifx \next
 \addEQ@\DN@"##1"{\savealias@{##1}\xyFN@\PATHlabelalias@}%
 \else \DN@{\afteraliases@@\PATHlabels@}\fi\fi \next@}
\xylet@\afteraliases@@=\empty
\xydef@\savealias@#1{\xy@@{\enter@\cfromthec@
 \advance\X@c-\L@c \advance\L@c\R@c \L@c=.5\L@c \R@c=\L@c \advance\X@c\L@c 
 \advance\Y@c-\D@c \advance\D@c\U@c \D@c=.5\D@c \U@c=\D@c \advance\Y@c\D@c 
 \idfromc@{#1}\leave@}}
\xylet@\arvariant@@=\empty
\xylet@\arstemprefix@@=\empty
\xylet@\artail@@=\empty
\xylet@\arstem@@=\empty
\xylet@\arhead@@=\empty
\xylet@\armodifiers@@=\empty
\xylet@\arlabels@@=\empty
\xylet@\arafterPOS@@=\empty
\xylet@\arinit@@=\empty
\xylet@\arexit@@=\empty
\xylet@\arcomponent@@=\relax
\xylet@\arcomponenttype@@=\relax
\xylet@\afterar@@=\relax
\xydef@\ar{\relax\arSAFE}
\xydef@\arSAFE{%
 \let\arvariant@@=\empty
 \def\arstemprefix@@{\dir}%
 \edef\artail@@{\arvariant@@{}}%
 \edef\arstem@@{\arvariant@@{-}}%
 \edef\arhead@@{\arvariant@@{>}}%
 \def\armodifiers@@{}%
 \def\arafterPOS@@{}%
 \def\arlabels@@{}%
 \def\arinit@@{}%
 \def\arexit@@{}%
 \let\PATHlabelabove@@=\PATHlabelabove@
 \let\PATHlabelbelow@@=\PATHlabelbelow@
 \xyFN@\ar@}
\xydef@\ar@{%
 \ifx \space@\next \expandafter\DN@\space{\xyFN@\ar@}%
 \else \addAT@\ifx\next \addAT@\DN@{\xyFN@\ar@form}%
 \else\ifx |\next
 \DN@ |{\ar@anchor|}%
 \else\ifx ^\next
 \DN@ ^{\ar@anchor^}%
 \else\ifx _\next
 \DN@ _{\ar@anchor_}%
 \else \let\next@=\ar@x \fi\fi\fi\fi\fi \next@}
\xydef@\addtoarinit@#1{%
 \expandafter\def\expandafter\arinit@@\expandafter{\arinit@@ #1}}
\xydef@\addtoarlabels@#1{%
 \expandafter\def\expandafter\arlabels@@\expandafter{\arlabels@@ #1}}
\xydef@\ar@anchor#1{\begingroup
 \def\PATHlabelit@@##1##2{%
 \DN@{##1}\ifx\next@\empty
 \DN@####1{\expandafter\endgroup\expandafter\addtoarlabels@
 \expandafter{\the\toks@####1{##2}}\xyFN@\ar@}%
 \else
 \DN@####1{\expandafter\endgroup\expandafter\addtoarlabels@
 \expandafter{\the\toks@*!C####1##1{##2}}\xyFN@\ar@}\fi
 \expandafter\next@\expandafter{\PATHitshape@@}}%
 \toks@={#1}%
 \def\xy@##1##2{\addtotoks@{##1}}\change@oxy@\xy@ \let\xy@@ix@=\eat@
 \DNii@{}\xyFN@\PATHanchor@}
\xydef@\ar@form{%
 \ifx \space@\next \expandafter\DN@\space{\xyFN@\ar@form}%
 \else\ifx ^\next \DN@ ^{\xyFN@\ar@style}\edef\arvariant@@{\string^}%
 \else\ifx _\next \DN@ _{\xyFN@\ar@style}\edef\arvariant@@{\string_}%
 \else\ifx 0\next \DN@ 0{\xyFN@\ar@style}\def\arvariant@@{0}%
 \else\ifx 1\next \DN@ 1{\xyFN@\ar@style}\def\arvariant@@{1}%
 \else\ifx 2\next \DN@ 2{\xyFN@\ar@style}\def\arvariant@@{2}%
 \else\ifx 3\next \DN@ 3{\xyFN@\ar@style}\def\arvariant@@{3}%
 \else\ifx \bgroup\next \let\next@=\ar@style
 \else\ifx [\next \DN@[##1]{\ar@modifiers{[##1]}}%
 \else\ifx *\next \DN@ *{\ar@modifiers}%
 \else\addLT@\ifx\next \let\next@=\ar@slide
 \else\ifx /\next \let\next@=\ar@curveslash
 \else\ifx (\next \let\next@=\ar@curveinout
 \else\addRQ@\ifx\next \addRQ@\DN@{\ar@curve@}%
 \else\addLQ@\ifx\next \addLQ@\DN@{\xyFN@\ar@curve}%
 \else\addDASH@\ifx\next \addDASH@\DN@{\defarstem@-\xyFN@\ar@}%
 \else\addEQ@\ifx\next \addEQ@\DN@{\def\arvariant@@{2}\defarstem@-\xyFN@\ar@}%
 \else\addDOT@\ifx\next \addDOT@\DN@{\defarstem@.\xyFN@\ar@}%
 \else\ifx :\next \DN@:{\def\arvariant@@{2}\defarstem@.\xyFN@\ar@}%
 \else\ifx ~\next \DN@~{\defarstem@~\xyFN@\ar@}%
 \else\ifx !\next \DN@!{\dasharstem@\xyFN@\ar@}%
 \else\ifx ?\next \DN@?{\ar@upsidedown\xyFN@\ar@}%
 \else \let\next@=\ar@error
 \fi\fi\fi\fi\fi\fi\fi\fi\fi\fi\fi\fi\fi\fi\fi\fi\fi\fi\fi\fi\fi\fi \next@}
\xydef@\defarstem@#1{\edef\arstem@@{\arvariant@@{\string#1}}%
 \DNii@##1##{\next@}%
 \edef\next@##1{\def\noexpand\artail@@{\arvariant@@{##1}}}%
 \expandafter\nextii@\artail@@
 \edef\next@##1{\def\noexpand\arhead@@{\arvariant@@{##1}}}%
 \expandafter\nextii@\arhead@@}
\xydef@\dasharstem@{%
 \DN@##1##{\nextii@{##1}}\DNii@##1##2{\def\arstem@@{##1{##2##2}}}%
 \expandafter\next@\arstem@@}
\xydef@\ar@error#1{\xyerror@{Illegal <form>ation (\meaning\next)}{}%
 \xyFN@\ar@}
\xydef@\ar@style{%
 \ifx \bgroup\next \def\artail@@{{}}%
 \edef\arstem@@{\arvariant@@{-}}\edef\arhead@@{\arvariant@@{>}}%
 \expandafter\ar@i
 \else \resetvariant@\artail@@ \resetvariant@\arstem@@ \resetvariant@\arhead@@
 \expandafter\xyFN@\expandafter\ar@ \fi}
\xydef@\resetvariant@#1{%
 \DN@##1##{\DN@{##1}\ifx\next@\empty
 \DN@{\expandafter\nextii@\expandafter{\arvariant@@}}%
 \else \DN@{\nextii@{##1}}\fi \next@}%
 \DNii@##1##2{\def#1{##1{##2}}}%
 \expandafter\next@#1}
\xydef@\ar@i#1{\DN@{#1}%
 \ifx\next@\empty \edef\arstem@@{\arvariant@@{}}\edef\arhead@@{\arvariant@@{}}%
 \DN@{\xyFN@\ar@}%
 \else
 \let\arcomponent@@=\ar@ii \let\arcomponenttype@@=\artip@
 \DN@{\xyFN@\arcomponent@#1$}%
 \fi \next@}
\xydef@\ar@ii{\ifx $\next \let\next@=\ar@iv
 \else \expandafter\def\expandafter\artail@@\expandafter{\the\toks@}%
 \let\arcomponent@@=\ar@iii \let\arcomponenttype@@=\arconn@
 \DN@{\xyFN@\arcomponent@}\fi \next@}
\xydef@\ar@iii{%
 \expandafter\def\expandafter\arstem@@\expandafter{\the\toks@}%
 \resetvariant@\artail@@
 \let\arcomponent@@=\ar@iv \let\arcomponenttype@@=\artip@
 \xyFN@\arcomponent@}
\xydef@\ar@iv{%
 \expandafter\def\expandafter\arhead@@\expandafter{\the\toks@}%
 \ifx $\next \DN@ ${\xyFN@\ar@}%
 \else \xyerror@{illegal <arrow>: \meaning\next\space not valid}{}\fi \next@}
\xydef@\ar@x{%
 \let\arsavedPATHafterPOS@@=\PATHafterPOS \let\PATHafterPOS=\arafterPOS@
 \toks@={\ar@PATH}%
 \expandafter\addtotoks@\expandafter{\expandafter{\artail@@}}%
 \expandafter\addtotoks@\expandafter{\expandafter{\arstem@@}}%
 \expandafter\addtotoks@\expandafter{\expandafter{\arstemprefix@@}}%
 \expandafter\addtotoks@\expandafter{\expandafter{\arhead@@}}%
 \expandafter\addtotoks@\expandafter{\expandafter{\armodifiers@@}}%
 \expandafter\addtotoks@\expandafter{\expandafter{\arinit@@}}%
 \expandafter\addtotoks@\expandafter{\expandafter{\arexit@@}}%
 \expandafter\addtotoks@\expandafter{\expandafter{\arlabels@@}}%
 \addtotoks@{\afterar@@}%
 \expandafter\DNii@\expandafter{\the\toks@}\toks@={}%
 \nextii@}
{\xyuncatcodes \gdef\next#1#2#3#4#5#6#7#8#9{%
 \def\next{%
 \afterPATH{#9}%
 ~={#6\preconnect#5#3#2}%
 ~/{#7}%
 ~<{|<*h#5\dir#1}%
 ~>{|>*h#5\dir#4}%
 ~+{#8}%
 }%
 \next}}
\xylet@\ar@PATH=\next
\xydef@\arafterPOS@#1{%
 \arsavedPATHafterPOS@@{\let\PATHafterPOS=\arsavedPATHafterPOS@@
 \DN@{#1}\expandafter\next@\arafterPOS@@}}%
\xylet@\arsavedPATHafterPOS@@=\relax
\xydef@\arcomponent@{%
 \ifx ^\next \toks@ii={^}\DN@ ^{\xyFN@\arcomponent@i}%
 \else\ifx _\next \toks@ii={_}\DN@ _{\xyFN@\arcomponent@i}%
 \else\ifx 1\next \toks@ii={1}\DN@ 1{\xyFN@\arcomponent@i}%
 \else\ifx 2\next \toks@ii={2}\DN@ 2{\xyFN@\arcomponent@i}%
 \else\ifx 3\next \toks@ii={3}\DN@ 3{\xyFN@\arcomponent@i}%
 \else\ifx \bgroup\next \expandafter\toks@ii\expandafter{\arvariant@@}%
 \let\next@=\arcomponent@i
 \else\ifx *\next \DN@*##1##{\arcomponent@ii{##1}}%
 \else \expandafter\toks@ii\expandafter{\arvariant@@}\toks@={}%
 \let\next@=\arcomponenttype@@
 \fi\fi\fi\fi\fi\fi\fi
 \next@}
\xydef@\arcomponent@i#1{\toks@={#1}\arcomponent@x}
\xydef@\arcomponent@ii#1#2{\toks@={*#1{#2}}\xyFN@\arcomponent@@}
\xydef@\artip@{%
 \addGT@\ifx\next \addGT@\addtotoks@ \addGT@\DN@{\xyFN@\artip@}%
 \else\addLT@\ifx\next \addLT@\addtotoks@ \addLT@\DN@{\xyFN@\artip@}%
 \else\ifx (\next \addtotoks@(\DN@({\xyFN@\artip@}%
 \else\ifx )\next \addtotoks@)\DN@){\xyFN@\artip@}%
 \else\ifx |\next \addtotoks@|\DN@|{\xyFN@\artip@}%
 \else\addLQ@\ifx\next \addLQ@\addtotoks@ \addLQ@\DN@{\xyFN@\artip@}%
 \else\addRQ@\ifx\next \addRQ@\addtotoks@ \addRQ@\DN@{\xyFN@\artip@}%
 \else\addPLUS@\ifx \next \addPLUS@\addtotoks@ \addPLUS@\DN@{\xyFN@\artip@}%
 \else\ifx /\next \addtotoks@/\DN@/{\xyFN@\artip@}%
 \else\ifcat A\noexpand\next \DN@##1{\addtotoks@{##1}\xyFN@\artip@}%
 \else\ifx\space@\next \addtotoks@{ }\expandafter\DN@\space{\xyFN@\artip@}%
 \else \let\next@=\arcomponent@x
 \fi\fi\fi\fi\fi\fi\fi\fi\fi\fi\fi \next@}
\xydef@\arconn@{%
 \addDASH@\ifx\next \addDASH@\addtotoks@ \addDASH@\DN@{\xyFN@\arconn@}%
 \else\addEQ@\ifx\next \addEQ@\addtotoks@ \addEQ@\DN@{\xyFN@\arconn@}%
 \ifx\arvariant@@\empty \def\arvariant@@{2}\fi
 \else\addDOT@\ifx\next \addDOT@\addtotoks@ \addDOT@\DN@{\xyFN@\arconn@}%
 \else\ifx :\next \addtotoks@:\DN@:{\xyFN@\arconn@}%
 \ifx\arvariant@@\empty \def\arvariant@@{2}\fi
 \else\ifx ~\next \addtotoks@~\DN@~{\xyFN@\arconn@}%
 \else \let\next@=\arcomponent@x
 \fi\fi\fi\fi\fi \next@}
\xydef@\arcomponent@x{%
 \DN@##1{\toks@=\expandafter{\the\toks@ii{##1}}}%
 \expandafter\next@\expandafter{\the\toks@}%
 \xyFN@\arcomponent@@}
\xydef@\ar@curveslash/#1/{\expandafter\ar@curve@\ar@slashing{#1}}
\xydef@\ar@slashing#1{{\xy@{**{} ?+/#1/+/#1/ @+c}{\setupDirection@
 \vfromslide@{#1}%
 \X@c=2\X@c \advance\X@c\X@p \advance\X@c.5\d@X
 \Y@c=2\Y@c \advance\Y@c\Y@p \advance\Y@c.5\d@Y
 \czeroEdge@
 \spushc@}}}
{\xyuncatcodes \catcode`\#=6
 \gdef\next(#1,#2){{+/#1 3pc/,p+/#2 3pc/}}}
\xylet@\ar@curveinout@=\next
\xydef@\ar@curveinout{\expandafter\ar@curve@load\ar@curveinout@}
\xydef@\ar@curve{%
 \ifx \space@\next \expandafter\DN@\space{\xyFN@\ar@curve}%
 \else\ifx \bgroup\next \let\next@=\ar@curve@load
 \else\ifx "\next \DN@"##1"{\ar@curve@{"##1"}}%
 \else \xyerror@{@= <form> must be followed by \string"<id>\string" or
 {<control point list>}}{}%
 \fi\fi\fi \next@}
{\xyuncatcodes \gdef\next#1#2{\def#1##1{#2{;@={##1}}}}}
\next\ar@curve@load\ar@curve@
\xydef@\ar@curve@#1{\curve@check
 \setcurvearinit@{#1}\setcurvearexit@\arexit@@
 \def\arstemprefix@@{\crvi}\xyFN@\ar@}
{\xyuncatcodes \gdef\next#1{\save @(,#1\restore}}
\xylet@\setcurvearinit@i=\next
\xydef@\setcurvearinit@#1{%
 \expandafter\addtoarinit@\expandafter{\setcurvearinit@i{#1}}}
{\xyuncatcodes \gdef\next#1{\def#1{\POS @i @) }}}
\xylet@\setcurvearexit@=\next
\xydef@\curve@check{%
 \xyerror@{Forms @/.../, @(...), and @`{...}, only available when curve
 extension loaded}{}}
\xywithoption{curve}{\let\curve@check=\relax}
\xydef@\ar@modifiers#1{%
 \expandafter\def\expandafter\armodifiers@@\expandafter{\armodifiers@@#1}%
 \xyFN@\ar@}
{\xyuncatcodes \catcode`\#=6 \catcode`\@=11
 \gdef\next<#1>{\def\arafterPOS@@{<#1>}\xyFN@\ar@}}
\xylet@\ar@slide=\next
\xydef@\ar@upsidedown{\let\next=\PATHlabelabove@@
 \let\PATHlabelabove@@=\PATHlabelbelow@@ \let\PATHlabelbelow@@=\next}
\xyendinput

\usetikzlibrary{decorations.pathmorphing}

\begin{document}

\maketitle
\begin{abstract}

A filtration of \Sbmdl s by virtual crossing \bmdl s extends to \Rq's complexes associated with braid words. We show that these complexes are invariant up to filtered homotopy with respect to the second Reidemeister move, and the filtration of the triply graded link diagram homology, constructed by Khovanov through the application of the Hochschild homology, is invariant under Markov moves. We also prove that the homotopy equivalence of the complexes of braid words related by the third Reidemeister move violates filtration by at most two units.
%
%

\end{abstract}

\begin{spacing}{0.65}
\tableofcontents
\end{spacing}

\section{Introduction}

In~\cite{KRO} the \vrcr\ categorification was introduced as a necessary tool in categorifying the \tSOtN\ Kauffman polynomial. In the same paper the categorification of \tSUN\ and \tglkhm\ were also extended to virtual links. The results of that paper inspired an attempt by Emmanuel Wagner~\cite{Wag} to introduce the third grading to the \tSUN\ link homology. Unfortunately, his approach contained substantial errors and the paper~\cite{Wag} was subsequently withdrawn.

In this paper we attempt, in the spirit of Wagner's ideas, to use virtual crossing complexes of~\cite{KRO} in order to introduce a \emph{filtration} in the \tglkhm. A similar filtration can be introduced in the \tSUN\ link homology.


Recall that two chain complexes $\cwdA$ and $\cwdB$ of objects of an additive category $\ctC$ are \het\ if there exists a diagram
\begin{equation}
\label{eq:dgheq}
\xymatrix@C=1.7cm{
\ccA\ar@/^/[r]^{\fAB}
\ar@(ul,dl)[]_-{\hmA}
&
\ccB \ar@/^/[l]^{\fBA}
\ar@(ur,dr)[]^-{\hmB}
}
\end{equation}
whose morphisms satisfy the relaitons
\begin{gather}
\label{eq:dghrel}
\dcB\,\fAB - \fAB\,\dcA = 0,\qquad
\dcA\,\fBA - \fBA\,\dcB = 0,
\\
\fBA\,\fAB - \xIA = [\dcA,\hmA],\qquad
\fAB\,\fBA - \xIB = [\dcB,\hmB].
\end{gather}
We call $\fAB$, $\fAB$ \emph{\hec s} and we call $\hmA$, $\hmB$ \emph{homotopies}.

Suppose that the complexes $\cwdA$ and $\cwdB$ are filtered, that is, the chain objects of $\ccA$, $\ccB$ are filtered and the differentials $\dcA$, $\dcB$ are filtered morphisms. We say that $\cwdA$ and $\cwdB$ are \emph{\fhet} if all morphisms of the diagram~\eqref{eq:dgheq} are filtered.

Let $\shs$ be the filtration shift functor: for a filtered object $A$ we have $\xFlmi (\shs A) = \xFlmimo A$. We say that a morphism between two filtered objects $f\colon A\rightarrow B$ \emph{violates} filtration by $k$ units if
its shifted version $f\colon M\rightarrow \shs^k N$ is filtered.

The definition of the \tghm\ requires a presentation of a link $\xL$ as a circular closure of a \brwd\ $\brwb$, which by definition is a finite sequence of positive and negative elementary braids $\sgi$ and $\sgni$, $1\leq i \leq \bstr-1$, where $\bstr$ is the number of strands in $\brwb$. Following Soergel~\cites{Sgl1,Sgl2} and Rouquier~\cite{Rou}, to elementary braids $\sgi$ and $\sgni$ we associate complexes $\ctsgi$ and $\ctsgni$ of \Sbmdl s. However this time each \Sbmdl\ in these complexes is assigned a descending filtration and the complexes are also filtered, that is, the differentials are filtered homomorphisms. Remarkably, this filtration is related to virtual crossings: an associated graded of a \Sbmdl\ is a sum of \bmdl s corresponding to purely virtual braids.

For an $\bstr$-strand \brwd\ $\brwb$ we introduce two sets of $\bstr$ variables: $\bfx = x_1,\ldots,x_{\bstr}$ and similarly $\bfy$. A filtered \Sbmdl\ $\dgMbxy$ is an object in the derived category $\DQfbxy$ of \Zgrdd\ (\qgrdg), filtered $\Qbxy$-modules. Then a complex of bimodules $\xctvbyx{\brwb}$ associated with $\brwb$ is an object in the category $\ChDQfbxy$ of chain complexes over an additive category $\DQfbxy$. As usual, the complex $\xctvbyx{\brwb}$ is determined by the choice of complexes $\ctsgi$ and $\ctsgni$ and by the composition rule: a complex associated to the product of \brwd s is the tensor product of their complexes over the intermediate variables: $\xctvbzx{\brwb_2\brwb_1} = \xctvbzy{\brwb_2}\otQby \xctvbyx{\brwb_1}$.

Note that an object of $\ChDQfbxy$ is a complex of complexes and, as such, it has two differentials: an \tinn\ differential corresponding to $\Drv(\xdmm)$ and an \tout\ one corresponding to $\ctChv{\xdmm}$. We denote the corresponding homologies as $\Hmint$ and $\Hmext$ respectively.

%
%

Let $\xLbb$ be the link diagram which is a circular closure of a \brwd\ $\brwb$. We define its associated complex $\ctxLbb$ as the result of replacing the \bmdl s of $\ctbrwbbxy$ by their derived tensor products with the `diagonal' \bmdl\ corresponding to the identity braid $\xIdbstr$ (see \ex{eq:cmplb}).
The complex $\ctxLbb$ is an object in the category of complexes over the homotopy category of complexes of filtered, graded vector spaces: $\ChKQfb$. The homology of $\ctxLbb$ with respect to the \tinn\ differential is the \Hchh\ of $\ctbrwbbxy$. The subsequent application of the \tout\ homology produces the filtered version of the \tgrd\ homology of $\xLbb$: $\Htg(\ctxLbb) = \Hmext\bigl(\Hmint(\ctxLbb)\bigr)$.


A braid is determined by a \brwd\ up to the second and third Reidemeister moves:
\[
\sgi \sgni = \xIdbrn\;\; \text{(\Rdta)},\qquad\sgni\sgi=\xIdbrn\;\;\text{(\Rdtb)},\qquad
\sgni\sgnio\sgni = \sgnio \sgni \sgnio\;\;\text{(\Rdth)},
\]
where $\xIdbrn$ is the unit braid with $\bstr$ strands. We will show that if two \brwd s
are related by a \sRm, then their complexes are \fhet\ (Theorem~\ref{th:R2}).
If \brwd s $\brwbo$ and $\brwbt$ are related by a \tRm, then we can \emph{not} prove the \fhec\ of their complexes. However, we prove that the \hec s established by \Rq\
\[
\xymatrix@C=1.7cm{
\ctbrwbo\ar@/^/[r]^{f_{12}}
&
\ctbrwbt \ar@/^/[l]^{f_{21}}
}
\]
violate filtration by at most two units (Theorem~\ref{th:trrx}).


An \orfr\ link is determined by a braid up to the first and second Markov moves:
\[
\xLv{\brdbo\brdbt} \lkeq \xLv{\brdbt\brdbo}\;\;\text{(\Mo)},\qquad
\xLv{(\sgo\sqbr\xIdbrn)(\xIdbro\sqbr\brdb)} = \xfLbo\;\;\text{(\Mta)},\qquad
\xLv{(\sgomo\sqbr\xIdbrn)(\xIdbro\sqbr\brdb)} = \xfLbmo\;\;\text{(\Mtb)},
\]
We will show that if two \brwd s $\brwbo$ and $\brwbt$ are related by \Mo\ or \Mta, then the complexes of their closures are \fhet: $\ctxLbbo\sim\ctxLbbt$, 
whereas if they are related by \Mtb, then the relation between their complexes is weaker: their \tinn\ homologies are \fhet\ (Remark~\ref{rm:mrmvo} and Theorem~\ref{th:mrkmvt}). As a consequence, if $\brwbo$ and $\brwbt$ are related by any sequence of Markov moves, then their \tinn\ homologies are \fhet: $\Hmint(\ctxLbbo)\sim\Hmint(\ctxLbbt)$ and their
\tgrd\ filtered homologies are isomorphic: $\Htg(\ctxLbbo)\cong\Htg(\ctxLbbt)$.
%


In Section~\ref{s:fhtbr} we find a simple presentation of the filtered complex of a two-strand \brwd\ (Theorem~\ref{th:fhtbr}) and compute the filtered homology of two-strand torus knots and links presented as their closures (Theorem~\ref{th:fhmtbr}).

\subsection{Acknowledgements}
M.A. and L.R. are thankful to Mikhail Khovanov for numerous illuminating discussions. L.R. is thankful to Eugene Gorsky for sharing the results of his unfinished research and presenting arguments in favor of the existence of an extra grading in the triply graded link homology.

The work of L.R. was supported in part by the NSF grant DMS-1108727.

\section{Filtered \Rqc\ of a \brwd}

\subsection{Overview}

A convenient description of algebraic constructions related to all versions of the \tglkhm\ uses various versions of a monoidal 2-category of points $\ctpt$. An object in this category is a finite number of ordered points: $\ptn$ is an object of $\bstr$ points. A monoidal product of two objects is their ordered disjoint union $\ptno\ptpr\ptnt$. In this paper we assume that the only morphisms between objects are endomorphisms. An endomorphism of an object is a 1-dimensional cobordism of some sort: a braid, a braid word, a graph-braid, a virtual braid, \etc. The endomorphisms are composed as cobordisms and their monoidal product is again a disjoint union. The set of endomorphisms $\End\ptn$ usually has a category structure: a morphism between two 1-dimensional cobordisms is a 2-dimensional cobordism.

We use a shortcut notation $\bfx=x_1,\ldots,x_{\bstr}$ for a list of indexed variables, with $\nmv{\bfx}=\bstr$ denoting the number of variables in a list.

A link homology construction involves a \tctfn\ $\ctdm$ from the 2-category $\ctpt$ to a monoidal algebraic 2-category.
To an object $\ptn$ the functor associates an algebra $\Qbx$, $\nmv{\bfx}=\bstr$, the variables $\bfx=x_1,\ldots,x_{\bstr}$ being associated with the point of $\ptn$.
Then the functor $\ctdm$ maps the category of endomorphisms $\Eptn$ to an appropriate category of (graded and filtered) bimodules over $\Qbx$, that is, to a derived category of $\Qbxy$-modules or to a category of chain complexes over it. We refer to these categories as \emph{\tgcts}.


To a composition of endomorphisms in $\Eptn$ the functor $\ctdm$ associates the composition of bimodules defined as the (derived) tensor product over the intermediate variables, whose action is then forgotten (we call such variables \emph{\tdmmy}), while to the disjoint union of endomorphisms one associates the tensor product of bimodules over $\IQ$. Finally, to a morphism between two endomorphisms in $\Eptn$ the functor $\ctdm$ associates a homomorphism of corresponding bimodules or a chain map between their complexes.

%
%

The language of the category $\ctpt$ allows a simple definition of the endomorphisms $\Eptn$: they are generated by the identity endomorphism $\Idpto\in\End\pto$, which we denote graphically as
$\Idpto=\zverlv{\lncf}{5}$,
and by a finite number of `simple' endomorphisms $\ene_i\in\End\ptt$ through composition and monoidal product modulo certain relations such as the symmetry group relations or the braid group relations. The functor $\ctdm$ is defined by its action on the generators: $\xctv{\Idpto}$ and $\xctv{\ene_i}$, and the main challenge is to verify that, thus defined, $\ctdm$ respects the relations between them.

In this paper we consider three types of endomorphisms in $\Eptn$. The endomorphisms of the first type form the group of virtual braids or, equivalently, the symmetric group $\Smgn$. Apart from the unit element $\zverlv{\lncf}{5}$, virtual braids are generated by the elementary virtual braid (\ie permutation) $\gnstot\in\Eptt$, presented graphically as $\gnstot=\zvirtv{\lncf}{5}$, modulo the relations of $\Smgn$ (\ie the purely virtual second and third Reidemeister moves). A generator $\gnsti$ of $\Smgn$ has a graphical presentation as a virtual braid:
\[
\gnsti = \cmbai{\zvirtv{\lncf}{5}}.
\]

The second type of endomorphisms in $\Eptn$ are \emph{\tbrgr s}. These endomorphisms are generated freely by $\zverlv{\lncf}{5}$ and by two \blgn s $\zblbbv{\lncf}{5},\zblbgv{\lncf}{5}\in\Eptt$. More precisely, the monoid of $\bstr$-strand \tbrgr s is freely generated by two elements:
\begin{equation}
\label{eq:brdgpm}
\fblbi = \cmbai{\zblbgv{\lncf}{5}},\qquad
\fblgi = \cmbai{\zblbbv{\lncf}{5}}.
\end{equation}


One can mix \tbrgr s with virtual braids modulo the relations
\begin{equation}
\label{eq:virblb}
\zblbvir{\lncf}{-11} \cong \zvirblb{\lncf}{-11} \cong \zblbgv{\lncf}{5},
\end{equation}
which, in view of the second virtual Reidemeister move, imply
\begin{equation}
\label{eq:virblg}
\zblgvir{\lncf}{-11} \cong \zvirblg{\lncf}{-11} \cong \zblbbv{\lncf}{5}.
\end{equation}

The third type of endomorphisms in $\Eptn$ are \brwd s which form a monoid. These monoids are generated freely by $\zverlv{\lncf}{5}$ and by two elementary braids $\zbrdpv{\lncf}{5},\zbrdmv{\lncf}{5}\in\Eptt$. In other words, the monoid of $\bstr$-strand \brwd s is generated by the elements
\[
\sgmp_i = \cmbai{\zbrdpv{\lncf}{5}},\qquad
\sgmm_i = \cmbai{\zbrdmv{\lncf}{5}}.
\]

One could pass from the \brwd\ monoid to the braid group by imposing the second and third Reidemeister move relations, but we do not do it, because in our filtered context we can not prove the invariance of the functor $\ctdm$ under the third Reidemeister move.

As usual, braids can be mixed with virtual braids and \tbrgr s, but we do not study possible relations, except the following two:
\begin{equation}
\label{eq:virbr}
\zvirbrpvir{\lncf}{-23} \cong \zbrdpv{\lncf}{5},\qquad
\zvirbrmvir{\lncf}{-23} \cong \zbrdmv{\lncf}{5}
\end{equation}

\subsection{Target categories}
Our target categories are the usual ones with an extra filtration. The first target category is the derived category of filtered, \Zgrdd\ $\Qbxy$-modules: $\DQfbxy$. The $\ZZ$-grading and its associated \tqdgr\ comes from assigning degrees to variables:
$\xdgq x_i = \xdgq y_i = 2$; this grading is always present and we drop it from the category notations. Apart from filtration and \qgrdg, the category $\DQfbxy$ has a homological \Zgrdg\ with \tadgr\ denoted as $\xdga$. It determines all sign factors associated with homological algebra. The category $\DQfbxy$ is endowed with three degree shift endofunctors related to the \qgrdg, \agrdg\ and filtration: $\shq$, $\sha$ and $\shs$.

A filtered $\Qbxy$-module is called \emph{\tfsf} if it is split-free as a $\Qbx$-module and as a $\Qby$-module. All bimodules in this paper are \tfsf, therefore one can use an ordinary tensor product over the intermediate variables rather than a derived one when composing \bmdl s by taking their tensor product over the intermediate variables.

The category $\DQfbxy$ is additive, and our second target category is the category of bounded chain complexes over it: $\ChDQfbxy$. An object of $\ChDQfbxy$ is a `complex of complexes'. This category has yet another shift endofunctor $\sht$ whose associated grading reflects the homological degree within the complexes of $\ctCh(\xdmm)$. However $\dgt$ is not a homological degree, the latter being $\dga$, and, in fact, $\dgt$ may take semi-integer values. In order to simplify our formulas, we introduce combined shift endofunctors:
\[
\saqt = \sha\shq^2,\qquad\ssqt=\shs\shq^2,\qquad\stat=\sht\sha,\qquad\stashf = (\sht\shs)^{1/2}.
\]

In order to distinguish the \tinn\ and \tout\ complexes within a complex of complexes, we put the \tinn\ complexes in boxes, while the \tout\ complexes are in square brackets. Here are two examples of this notation:
%
\[
\left[\;
\stat\,
\begin{tikzpicture}[baseline=-0.25em]
\matrix (m) [ssmtf,column sep=2em]
{\sha\,M_1 & M_2
&& \sha\,N_1 & N_2\\};
\node (pz) at (m-1-2) [xshift=1.25em] {};
\node (po) at (m-1-4) [xshift=-1.5em] {};
\path[->,font=\scriptsize]
(m-1-1) edge node[auto] {$g_M$} (m-1-2)
(m-1-4) edge node[auto] {$g_N$} (m-1-5)
(pz) edge node[auto] {$f$} (po);
\node [draw,fit=(m-1-1) (m-1-2)] {};
\node [draw,fit=(m-1-4) (m-1-5)] {};
\end{tikzpicture}
\;\right]
=
\left[\;
\stat\,
\begin{tikzpicture}[baseline=-0.25em]
\matrix (m) [ssmtf,column sep=5em, row sep=2em]
{\sha\,M_1 & \sha\,N_1 \\ M_2 & N_2\\};
\node (pz) at (m-1-1) [xshift=1.70em] {};
\node (po) at (m-1-2) [xshift=-1.65em] {};
\node (qz) at (m-2-1) [xshift=1.70em] {};
\node (qo) at (m-2-2) [xshift=-1.65em] {};
\path[->,font=\scriptsize]
(m-1-1) edge node[auto] {$g_M$} (m-2-1)
(m-1-2) edge node[auto] {$g_N$} (m-2-2)
(pz) edge node[auto] {$f_1$} (po)
(qz) edge node[auto] {$f_2$} (qo);
\node [draw,fit=(m-1-1) (m-2-1) ] {};
\node [draw,fit=(m-1-2) (m-2-2)] {};
\end{tikzpicture}
\;\right],
\]
where $M_1$, $M_2$, $N_1$ and $N_2$ are $\Qbxy$-modules, while $g_M$, $g_N$, $f_1$ and $f_2$ are homomorphisms between them. Note that the \rhs presents the chain morphism $f$ explicitly in terms of its components $f_1$ and $f_2$.


\subsection{Filtered modules and their resolutions}
Let $M = M_0\supset \cdots \supset M_m$
be a filtered $\Qbxy$-module representing an object in the derived category $\DQfbxy$. Denote $N_i = M_i/M_{i+1}$ the associated graded submodules. Then a filtered projective resolution $\Prrv{M}$ of $M$ can be constructed as a multiple cone, that is, a sum of projective resolutions $\Prrv{N_i}$ with added differentials $f_{ij}\colon N_i\rightarrow N_j$, $i<j$, which are not necessarily closed, except when $j=i+1$.
For example, if the filtration depth $m$ is three, then
\[
\Prrv{M}\sim
\boxed{
\xymatrix{
\Prrv{N_0}
\ar[r]^{f_{01}}
\ar@/^2pc/[rr]^-{f_{02}}
\ar@/^4pc/[rrr]^-{f_{03}}
&
\shs \Prrv{N_1}
\ar[r]^{f_{12}}
\ar@/_2pc/[rr]_-{f_{13}}
&
\shs^2 \Prrv{N_2}
\ar[r]^{f_{23}}
&
\shs^3 \Prrv{N_3}
}
}.
\]
We call this picture a \fdg\ of $M$.

Let the $\Qbxy$-module $M$ have a filtration of depth $m=1$: $M = M_0\supset M_1$, then $N_0 = M_0/M_1$ and these modules form an exact sequence
\begin{equation}
\label{eq:exsqo}
0\rightarrow M_1 \rightarrow M \rightarrow N_0\rightarrow 0.
\end{equation}
The latter turns into an exact triangle of filtered resolutions:
\[
\begin{tikzpicture}[baseline=-0.25em]
\matrix (m) [ssmts]
{& \Prrv{N_0} & \Prrv{N_0} & \Prrv{M_1}
\\
\shs\Prrv{M_1} & \shs \Prrv{M_1}
\\};
\node (xup) at (m-1-2) [xshift=+2.75em]{};
\node (xdn) at (m-2-2) [xshift=-2.75em]{};
\path[->,font=\scriptsize]
(m-2-1) edge node[auto] {$\xId$} (xdn)
(xup) edge node[auto] {$\xId$} (m-1-3)
(m-1-2) edge node[auto] {$f$} (m-2-2)
(m-1-3) edge node[auto] {$f$} (m-1-4);
\node [draw, fit=(m-1-2) (m-2-2) ] {};
\end{tikzpicture}
\]
where the chain map $f$ represents the first extension determined by the exact sequence~\eqref{eq:exsqo}.

\subsection{Symmetric group and special bimodules}

The simplest choice of $\Eptn$ is the symmetry group: $\Eptn = \Smgn$. From the topological point of view $\Smgn$ is the group of purely virtual braids (a categorification of the whole virtual braid group is discussed in~\cite{Th}).
From the $\ctpt$ point of view, the groups $\Smgn$ are generated by the identity 1-strand braid $\Idpto\in\Epto$ and by the elementary permutation $\gnstot\in\Eptt$ which generates $\Smgt$. These generators have a graphic presentation: $\Idpto=\zverlv{\lncf}{5}$,  $\gnstot = \zvirtv{\lncf}{5}$. Note that we do not put a circle around a virtual crossing.

The \tgct\ for $\Eptn=\Smgn$ is just $\DQfbxy$ and we set
\[
\sfverl=
\Qxy/(y-x),\qquad
\xctv{
 \zvirtv{\lncf}{5}
} = \Qbxy/(y_2 - x_1,y_1 - x_2), \quad\nmv{\bfx}=\nmv{\bfy} = 2.
\]
This choice dictates the bimodule to be associated with a general permutation $\gnst\in\Smgn$:
\begin{equation}
\label{eq:spbmdq}
\xctv{\gnst} = \Qbxy/(y_{\gnst(1)}-x_1,\ldots y_{\gnst(\bstr)} - x_\bstr).
\end{equation}
All bimodules $\xctv{\gnst}$ are endowed with the trivial filtration: $\xFlz\xctv{\gnst} = \xctv{\gnst}$ and $\xFli\xctv{\gnst} = 0$ if $i>0$. It is easy to see that the bimodules~\eqref{eq:spbmdq} satisfy $\Smgn$ relations, that it,
\begin{equation}
\label{eq:tpcmp}
\xctvbzy{\gnst'}\otQby\xctvbyx{\gnst}\cong\xctvbzx{\gnst'\gnst}
\end{equation}
for any $\gnst,\gnst'\in\Smgn$.
\begin{remark}
\label{rm:tpsb}
The tensor product functors $\xctvbyx{\gnst}\otQbx\dmmy$ and $\dmmy\otQby\xctvbyx{\gnst}$ act by permuting the corresponding variables.
\end{remark}


The bimodule~\eqref{eq:spbmdq} can be presented as a tensor product of $\sfverl$-like bimodules over $\IQ$:
$\xctv{\gnst} \cong \bigotimes_{i=1}^{\bstr} \xctvv{\zverlv{\lncf}{5}}{y_{\gnst(i)}}{x_i}$.
Denote the \emph{canonical} \Kr\ of the `diagonal' bimodule $\sfverl$ as $\dgbyx$:
\begin{equation}
\label{eq:stkzrdo}
\Prcnxv{\sfveryx}=\dgbyx =
\boxed{
\xymatrix@C=1.2cm{
\saqt\,\Qxy\ar[r]^-{y - x}
&
\Qxy
}
}.
\end{equation}
The index `$\mcan$'
means `canonical'.
The \cKr\ of the bimodule~\eqref{eq:spbmdq} is defined as the tensor product of resolutions~\eqref{eq:stkzrdo}:
\begin{equation}
\label{eq:cnrsvK}
\Prcnxv{\xctv{\gnst}} = \bigotimes_{i=1}^\bstr \dgbv{y_{\gnst(i)}}{x_i}.
\end{equation}

Denote
\begin{equation}
\label{eq:dfdfv}
\xp = x_1 + x_2,\qquad \xm = x_2 - x_1,\qquad \yp=y_1+y_2,\qquad \ym=y_2-y_1.
\end{equation}
Then the canonical \Kszl\ resolutions of $\sfparl$ and $\sfvirt$ split into tensor products
of complexes of $\Qv{\xp,\yp}$-modules and complexes of $\Qv{\xm,\ym}$-modules:
\begin{equation}
\label{eq:cmfct}
\sfnparl = \dgbyxp\otimes \sfmparl
,\qquad
\sfnvirt = \dgbyxp\otimes\sfmvirt
,
\end{equation}
where
\begin{equation}
\label{eq:dfvrtm}
\sfmparl=\dgbv{\ym}{\xm},\qquad\sfmvirt=\dgbv{\ym}{-\xm}.
\end{equation}
Note that $\dgbyxp$ is a common factor in these products.
We introduce special morphisms $\vsdp\in\Ext^1\bigl(\sfparl,\sfvirt\bigr)$ and
$\vsdm\in\Ext^1\bigl(\sfvirt,\sfparl\bigr)$: they act as identity on the common factor $\dgbv{\yp}{\xp}$, while acting on the second factors as follows:
\begin{equation}
\label{eq:mcns}
\begin{tikzpicture}[baseline=-0.25em]
\matrix (m) [ssmth, column sep=1.5em, row sep = 5em]
{\sfmparl
&
\saqt\,\Qmxy
&&
\Qmxy
\\
\sfmvirt
&
\saqt\,\Qmxy
&&
\Qmxy
\\};
\node (P12) at (m-1-2) [yshift=-1.75em] {};
\node (P24) at (m-2-4) [yshift=1.75em] {};
\path[->,font=\tiny]
(m-1-1) edge node[auto] {$\vsdp$} (m-2-1)
(m-1-2) edge node[auto] {$\ym-\xm$} (m-1-4)
(m-2-2) edge node[auto] {$\ym+\xm$} (m-2-4)
(P12) edge node[auto] {$1$} (P24);
\node[draw,fit=(m-1-2) (m-1-4) ] {};
\node[draw,fit=(m-2-2) (m-2-4) ] {};
\path (m-1-1) node [xshift=1.95em] {$\cong$};
\path (m-2-1) node [xshift=1.95em] {$\cong$};
\end{tikzpicture}
\qquad
\begin{tikzpicture}[baseline=-0.25em]
\matrix (m) [ssmth, column sep=1.5em, row sep = 5em]
{\sfmvirt
&
\saqt\,\Qmxy
&&
\Qmxy
\\
\sfmparl
&
\saqt\,\Qmxy
&&
\Qmxy
\\};
\node (P12) at (m-1-2) [yshift=-1.75em] {};
\node (P24) at (m-2-4) [yshift=1.75em] {};
\path[->,font=\tiny]
(m-1-1) edge node[auto] {$\vsdm$} (m-2-1)
(m-1-2) edge node[auto] {$\ym+\xm$} (m-1-4)
(m-2-2) edge node[auto] {$\ym-\xm$} (m-2-4)
(P12) edge node[auto] {$1$} (P24);
\node[draw,fit=(m-1-2) (m-1-4) ] {};
\node[draw,fit=(m-2-2) (m-2-4) ] {};
\path (m-1-1) node [xshift=1.95em] {$\cong$};
\path (m-2-1) node [xshift=1.95em] {$\cong$};

\end{tikzpicture}
\end{equation}
Obviously, $\dga\vsdp = \dga\vsdm = -1$, $\dgq\vsdp = \dgq\vsdm = -2$.

We refer to $\vsdp$ and $\vsdm$ as \tvsd\ morphisms, because they may be considered as the result of applying the \tctfn\ to the \tvsd\ cobordism connecting the virtual braids $\zparlv{\lncf}{5}$ and $\zvirtv{\lncf}{5}$.

For any permutation $\gnst\in\Smgn$ and any transposition $\gnstij$ we define a \tvsd\ morphism \begin{equation}
\label{eq:tvsd}
\vsdijs\in\Ext^1\left(\xctv{\gnst},\xctv{\gnst\gnstij}\right),
\end{equation}
where $\gnst,\gnstij\in\Smgn$ and $\gnstij$ is a transposition. The bimodules $\xctv{\gnst}$ and $\xctv{\gnst\gnstij}$ have a common factor $\xMcm=\bigotimes_{k\neq i,j}\sfvervv{y_{\gnst(k)}}{x_k}$:
\[
\xctvbyx{\gnst} = \xMcm \otimes \xectvv{ \zparlv{\lncf}{5}}
{y_{\gnst(i)},y_{\gnst(j)}}{x_i,x_j},
\qquad
\xctvbyx{\gnst\gnstij} = \xMcm\otimes
\xectvv{\zvirtv{\lncf}{5} }{y_{\gnst(i)},y_{\gnst(j)}}{x_i,x_j}
\]
and $\vsdijs$ acts as identity on $\xMcm$ and as $\vsdp$ on the remaining 2-strand factors.

 %
In view of Remark~\ref{rm:tpsb}, the tensor product with other permutation bimodules transforms the \tvsd\ morphism by relabeling:
\begin{equation}
\label{eq:cmdgs}
\xymatrix@C=2cm@R=1.5cm{
\xctvbwz{\gnst''}\otQbz\xctvbzy{\gnst}\otQby\xctvbyx{\gnst'}
\ar[d]^-{\cong}
\ar[r]^{\xId\otimes\vsdijs\otimes\xId}
&
\xctvbwz{\gnst''}\otQbz\xctvbzy{\gnst\gnst_{ij}}\otQby\xctvbyx{\gnst'}
\ar[d]^-{\cong}
\\
\xctvbwx{\gnst''\gnst\gnst'}
\ar[r]^-{\vsdvv{(\gnst')^{-1}(i),(\gnst')^{-1}(j)}{\gnst''\gnst\gnst''}}
&
\xctvbwx{\gnst''\gnst\gnst_{ij}\gnst'}
}
\end{equation}

%


\subsection{Filtered \Sbmdl s}
\subsubsection{Filtration}
The second step in building up $\Eptn$ is the addition of two \blgn s $\zblbbv{\lncf}{5},\zblbgv{\lncf}{5}\in\Eptt$. The corresponding filtered bimodules are defined by endowing the standard \Sgl\ blob bimodule
\[
\mMblb = \Qbxy/\bigl( (y_1+y_2) - (x_1+x_2),y_1y_2-x_1x_2\bigr) = \Qbxy/\bigl(\yp-\xp,(\ym+\xm)(\ym-\xm)\bigr)
\] with a filtration related to the following exact sequences:
\begin{equation}
\label{eq:exsq}
\xctv{\zblbbv{\lncf}{5}}\colon
\xymatrix@C=1.3cm{
\shq^2\xctv{\zvirtv{\lncf}{5} }
\ar[r]^-{\ym-\xm}
&
\mMblb
\ar[r]^-{1}
&
\xctv{\zparlv{\lncf}{5}},
}
\qquad
\xctv{\zblbgv{\lncf}{5}}\colon
\xymatrix@C=1.3cm{
\shq^2\xctv{\zparlv{\lncf}{5} }
\ar[r]^-{\ym+\xm}
&
\mMblb
\ar[r]^-{1}
&
\xctv{\zvirtv{\lncf}{5}}.
}
\end{equation}
In other words,
\[
\xFlz\sfblbb = \sfblbb,\qquad
\xFlo\sfblbb = \shq^2\sfvirt,\qquad
\xFlz\sfblbg = \sfblbg,\qquad
\xFlo\sfblbg = \shq^2\sfparl.
\]
It is easy to see that the exact sequences~\eqref{eq:exsq} correspond to the extensions $\vsdp$ and $\vsdm$ defined by the diagrams~\eqref{eq:mcns}. Canonical filtered resolutions
of the bimodules $\sfblbb$ and $\sfblbg$ are defined as cones:
\begin{equation}
\label{eq:fdgel}
\xctcnv{\zblbbv{\lncf}{5}} =
\boxed{
\xymatrix{
\xctcnv{\zparlv{\lncf}{5}}
\ar[r]^-{\vsdp}
&
\ssqt\,\xctcnv{\zvirtv{\lncf}{5}}
}
}\, ,
\qquad
\xctcnv{\zblbgv{\lncf}{5}} \simeq
\boxed{
\xymatrix{
\xctcnv{\zvirtv{\lncf}{5}}
\ar[r]^-{\vsdm}
&
\ssqt\,\xctcnv{\zparlv{\lncf}{5}}
}
}\, ,
\end{equation}
boxed diagrams being the \fdg s of the \blbm s.

The disjoint union with \varc s on the left and on the right
turns \ex{eq:fdgel} into \fdg s for the bimodules associated with diagrams~\eqref{eq:brdgpm}:
\begin{equation}
\label{eq:cnrsml}
\ctcnblbi \simeq
\boxed{
\xymatrix{
\ctcnIdptn
\ar[r]^-{\vsdi}
&
\ssqt\,\ctcngnsti
}
}\,,
\qquad
\ctcnblgi \simeq
\boxed{
\xymatrix{
\ctcngnsti
\ar[r]^-{\vsdi}
&
\ssqt\,\ctcnIdptn
}
}\,.
\end{equation}

\subsubsection{Filtered homomorphisms}

The bimodules $\sfblbb$ and $\sfblbg$ are related to the diagonal bimodule $\sfparl$ by two
canonical morphisms $\chp$ and $\chm$ associated with the filtration:
\[
\begin{tikzpicture}[baseline=-0.25em]
\matrix (m) [ssmth, column sep=2.5em]
{
\ssqt\sfparl
&
&
\ssqt\sfparl
\\
\sfblbg
&
\sfvirt
&
\ssqt\sfparl
\\
};
\node  (PH) at (m-2-3) [yshift=1.5em] {};
\path[->,font=\scriptsize]
(m-1-1) edge node[auto] {$\chp$} (m-2-1)
(m-1-3) edge node[auto] {$\xId$} (PH)
(m-2-2) edge node[auto] {$\vsdm$} (m-2-3);
\node[draw,fit=(m-2-2) (m-2-3) ] {};
\node at (m-2-1) [xshift=2.5em] {$\simeq$};
\node at (m-1-1) [xshift=2.5em] {$=$};
\end{tikzpicture}
\quad,
\qquad\qquad\quad
\begin{tikzpicture}[baseline=-0.25em]
\matrix (m) [ssmth, column sep=2.5em]
{
\sfblbb
&
\sfparl
&
\ssqt\sfvirt
\\
\sfparl
&
\sfparl
&
\\
};
\node  (PH) at (m-1-2) [yshift=-1.5em] {};
\path[->,font=\scriptsize]
(m-1-1) edge node[auto] {$\chm$} (m-2-1)
(PH) edge node[auto] {$\xId$} (m-2-2)
(m-1-2) edge node[auto] {$\vsdp$} (m-1-3);
\node[draw,fit=(m-1-2) (m-1-3) ] {};
\node at (m-2-1) [xshift=2.5em] {$=$};
\node at (m-1-1) [xshift=2.5em] {$\simeq$};
\end{tikzpicture}
\]
These morphisms appear in \Rq\ complexes associated with braid group generators.
Two other filtered homomorphisms relating \blbm s appear later in our constructions:
\[
\xpsim\colon \ssqt\,\sfblbb \xrightarrow{\;\ym-\xm\;}\sfblbg,\qquad
\xpsip\colon\ssqt\,\sfblbg\xrightarrow{\;\ym+\xm\;} \sfblbb.
\]
Their action on filtered resolutions is depicted by the diagrams
\[
\begin{tikzpicture}[baseline=-0.25em]
\matrix (m) [ssmth, row sep=4em, column sep=1.5em]
{
\ssqt\sfblbb
&
&
\sfparl
&
\ssqt\sfvirt
\\
\sfblbg
&
\sfvirt
&
\ssqt\sfparl
\\
};
\node  (P1) at (m-2-3) [yshift=1.5em] {};
\node (P2) at (m-1-3) [yshift=-1.5em] {};
\path[->,font=\scriptsize]
(m-1-1) edge node[auto] {$\xpsim$} (m-2-1)
(P2) edge node[auto] {$\xId$} (P1)
(m-2-2) edge node[auto] {$\vsdm$} (m-2-3)
(m-1-3) edge node[auto] {$\vsdp$} (m-1-4);
\node[draw,fit=(m-2-2) (m-2-3) ] {};
\node[draw,fit=(m-1-3) (m-1-4) ] {};
\node at (m-2-1) [xshift=2em] {$\simeq$};
\node at (m-1-1) [xshift=2em] {$\simeq$};
\node at (m-1-3) [xshift=-2.5em] {$\ssqt$};
\end{tikzpicture}
\quad,
\qquad
\begin{tikzpicture}[baseline=-0.25em]
\matrix (m) [ssmth, row sep=4em, column sep=1.5em]
{
\ssqt\sfblbg
&&
\sfvirt
&
\ssqt\sfparl
\\
\sfblbb
&
\sfparl
&
\sfvirt
&
\\
};
\node  (P1) at (m-1-3) [yshift=-1.5em] {};
\node (P2) at (m-2-3) [yshift=1.5em]{};
\path[->,font=\scriptsize]
(m-1-1) edge node[auto] {$\xpsip$} (m-2-1)
(P1) edge node[auto] {$\xId$} (P2)
(m-1-3) edge node[auto] {$\vsdm$} (m-1-4)
(m-2-2) edge node[auto] {$\vsdp$} (m-2-3);
\node[draw,fit=(m-1-3) (m-1-4) ] {};
\node[draw,fit=(m-2-2) (m-2-3) ] {};
\node at (m-2-1) [xshift=2em] {$\simeq$};
\node at (m-1-1) [xshift=2em] {$\simeq$};
\node at (m-1-3) [xshift=-2.5em] {$\ssqt$};
\end{tikzpicture}
\]

\subsection{Filtered \Rq\ complexes}
Now we can add \brwd s to $\Eptn$. We add $\sfbrdp,\sfbrdm\in\Eptt$ and these braids generate by composition and disjoint union the \brwdm\ whose elements, by definition, are finite sequences of elementary positive and negative braids $\sgi$ and $\sgi^{-1}$.

To elementary 2-strand braids we associate the complexes which are the cones of \tflcnm s within the category $\ChDQfbxy$:
\begin{align}
\label{eq:rqco}
\sfbrdp
&\;= \;
\stasaqt
\Bcpbl
\xymatrix{
\sfparl
\xari[r]^-{\chp}
&
\stati\ssqti\sfblbg
}
\Bcpbr
\;\cong\;
\stasaqt
\bgcpbl
\xymatrix{
\sfparl \xari[r]^-{\chp}
&
\stati
\boxed{\ssqti\sfvirt\xrightarrow{\vsdm}\sfparl}
}
\bgcpbr,
\\
\label{eq:rqct}
\sfbrdm
&\; = \;
\stasaqti
\Bcpbl
\xymatrix{
\stat\sfblbb
\xars[r]^-{\chm}
&
\sfparl
}
\Bcpbr
\;\cong\;
\stasaqti
\bgcpbl
\xymatrix{
\stat
\boxed{\sfparl\xrightarrow{\vsdp}\ssqt\sfvirt}
\xars[r]^-{\chm}
&
\sfparl
}
\bgcpbr.
\end{align}
Recall that the boxes around the complexes indicate that the latter are `inner' complexes, that is, they represent the objects of $\DQfbxy$, while the complexes in square brackets are `outer', that is, they are  the complexes within the category of chain complexes $\ChDQfbxy$.
%
%

\subsection{\HMFp}
In our conventions, in order to obtain the \HMFp\ of a link diagram as a graded Euler characteristic of its \tghm\ we introduce a factor
\[
(-1)^{\dga} a^{\dga-\dgt} q^{\dgq}
\]
multiplying the graded dimensions of the homology. As a result, the skein relation is
\[
a^{1/2}q \left\langle \zbrdmv{\lncf}{5}\right\rangle
-
a^{-1/2}q^{-1} \left\langle \zbrdpv{\lncf}{5}\right\rangle
= (q-q^{-1}) \left\langle \zparlv{\lncf}{5}\right\rangle
\]
and the \HMFp\ of the unknot is $(a^{1/2} q - a^{-1/2}q^{-1})/(q-q^{-1})$.

\subsection{Virtual relations}



The virtual relations~\eqref{eq:virblb},~\eqref{eq:virblg} and~\eqref{eq:virbr} are satisfied after the application of the functor $\xctv{\xdmm}$.
\begin{theorem}
There are isomorphisms of modules
\begin{equation}
\label{eq:cvirbl}
\yectv{\zblbvir{\lncf}{-11}} \cong \yectv{\zvirblb{\lncf}{-11}} \cong \sfblbg,\qquad
\yectv{\zblgvir{\lncf}{-11}} \cong \yectv{\zvirblg{\lncf}{-11}} \cong \sfblbb,
\end{equation}
and isomorphisms of complexes
\begin{equation}
\label{eq:cvircr}
\yectv{\zvirbrpvir{\lncf}{-23}} \cong \sfbrdp,\qquad
\yectv{\zvirbrmvir{\lncf}{-23}} \cong \sfbrdm.
\end{equation}
\end{theorem}
\begin{proof}
Consider the presentation~\eqref{eq:fdgel} of $\sfblbb$ as a cone. In view of Remark~\ref{rm:tpsb} and a commutative diagraim~\eqref{eq:cmdgs}, a tensor product of this presentation with $\sfvirt$ over the intermediate variables turns it into the similar presentation for $\sfblbg$. This proves the left isomorphisms of~\eqref{eq:cvirbl}. The right isomorphisms are proved similarly.

In order to prove the first isomorphism of~\eqref{eq:cvircr}, we sandwich the first presentation of $\sfbrdp$ in~\eqref{eq:rqco} by two bimodules $\sfvirt$. Both bimodules of the cone remain the same and the morphism $\chp$ transforms into itself, because it is a canonical morphism associated with the filtration. The second isomorphism of~\eqref{eq:cvircr} is proved similarly.
\end{proof}

\subsection{Morphism abbreviations within large diagrams}
In order to avoid overloading of large diagrams appearing in sections~\ref{s.srm},~\ref{s:thrm} and~\ref{s:mmvs}, we will omit labels on arrows connecting \bmdl s and their resolutions which correspond to two types of standard morphisms.

First, an unlabelled arrow connecting two \bmdl s (or their resolutions) of virtual braids, for example,
\[\begin{tikzcd} \sftpvir \rar & \sftrilo\end{tikzcd}\]
appearing in the diagram~\eqref{eq:lngsmth},
corresponds to the \tvsd~\eqref{eq:tvsd}.

Second, an unlabelled arrow connecting two \bmdl s (or their resolutions) of braid-graphs composed of black blobs, for example
\begin{equation}
\label{eq:ttvsd}
\begin{tikzcd}
\sftrilob \rar&
\sflpblbb
\end{tikzcd}
\end{equation}
appearing in the diagram~\eqref{eq:bgthrc} corresponds to the homomorphism $\chm$ which maps one of the blobs into two parallel arcs:
\begin{equation}
\label{eq:tpchm}
\begin{tikzcd}
\xxectvv{\zlblbbp{\lntx}{\lnty}{\lnte}}{\bfz}{\bfy}
\otQby
\xxectvv{\zlpblbb{\lntx}{\lnty}{\lnte}}{\bfy}{\bfx}
\rar{\chm\otimes\xId}
&
\xxectvv{\ztparl{\lntx}{\lnty}{\lnte}}{\bfz}{\bfy}
\otQby
\xxectvv{\zlpblbb{\lntx}{\lnty}{\lnte}}{\bfy}{\bfx}.
\end{tikzcd}
\end{equation}

\section{Multi-filtered structure of \Sgl\ \bmdl s and \Rq\ complexes}

\subsection{A category of \mfspl\ chain complexes}
Consider a lattice $\Zbstr$. We denote its vertices as $\vert = (\vert_1,\ldots,\vert_\blbn)$. The vertices are partially ordered: $\vertp\geq \vert$ if $\vertp_i\geq\vert_i$ for all $i$. We define the \ttdg\ of a vertex as $\tdgv{\vert} = \vert_1+\cdots+\vert_{\blbn}$.

For an additive category $\ctA$ we define an associated \nfspl\ category $\FnA$.
Its objects are $\blbn$-graded objects of $\ctA$, that is, an object of $\FnA$ is a \emph{finite} direct sum of objects of $\ctA$: $A = \bigoplus_{\vert\in\Zbstr} A_{\vert}$. The gradings define $\blbn$ filtrations $\xFlmi$, $1\leq i \leq \blbn$ of the object $A$: $\xFlmij A = \bigoplus_{\substack{\vert\in\Zbstr \\ \vert_i\geq j}} A_{\vert}.$ A morphism $A\xrightarrow{f} B$ between two objects of $\FnA$ is a morphism between them as objects of $\ctA$ which respects all filtrations, that is,  it has a nontrivial component $A_{\vert}\xrightarrow{f_{\vert,\vertp}}B_{\vertp}$ only when $\vertp\geq \vert$.

The category $\FnA$ is again additive, and we consider the category $\ChumA$
of bounded chain complexes over it.
An object of this category is a bounded chain complex
$\xcmpAd=\bigl(\cdots\xrightarrow{d_{i-1}} A_i \xrightarrow{d_i} A_{i+1}\xrightarrow{d_{i+1}}\cdots\bigr)$ whose chain objects $A_i$ are finitely $\blbn$-graded objects of $\ctA$ and whose differentials are \nfltd\ morphisms.
For $\vert\in\Zbstr$ the complex $A$ determines an associated graded complex
\[\xcmpAdv = \bigl( \cdots\rightarrow A_{i;\vert} \xrightarrow{d_{i;\vert} } A_{i+1;\vert}\rightarrow\cdots\bigr),\] which is an object in the category $\ChA$ of chain complexes over $\ctA$. We refer to $\xcmpAdv$ as \emph{\cnstc es} of $\xcmpAd$. We often drop the differential from the notation of the complex writing $\cpA$ instead of $\xcmpAd$.

\subsection{An associated homotopy category}
By the standard definition, if three morphisms $f,g,h\in\Hom_{\FnA}(\cpA,\cpB)$ satisfy a relation $f-g = [\cpd,h]$, then $f$ and $g$ are called homotopic. This defines the \mfspl\ homotopy category $\KumA$ over the additive category $\ctA$.

A filtration $\xFlmi$ determines two endofunctors $\xFlmij$ and $\xQlmij$ of the category $\ChFnA$: $\xFlmij\cpA$ is a subcomplex of $\cpA$, while $\xQlmij\cpA$ is the corresponding quotient complex. These two complexes are related to $\cpA$ by \emph{canonical} morphisms
\begin{equation}
\label{eq:mfcnmr}
\xymatrix{
\xFlmij \cpA
\xari[r]
&
\cpA
\xars[r]
&
\xQlmij\cpA
}.
\end{equation}
\begin{theorem}
\label{th:cnhmt}
The functors $\xFlmij$ and $\xQlmij$ as well as the canonical morphisms~\eqref{eq:mfcnmr} descend to the homotopy category $\KumA$.
\end{theorem}
The proof is standard (triangulated structure of homotopy category) and we leave it to the reader.

\subsection{Multi-cone structure of \mfspl\ chain complexes}

Since the complex $\cpA$ is bounded, there exists a finite \emph{\tdmn} $\sbsXA\subset\Zbstr$ such that the complexes $\Avr$ are trivial unless $\vert\in\sbsXA$. Note that a \tdmn\ $\sbsXA$ is not determined by the complex $\cpA$ uniquely, because one can add to $\sbsXA$ more vertices $\vert$ whose complexes $\Avr$ are trivial.

Define subsets of $\Zbstr$:
\[
\sbsGi = \{\vert\in\Zbstr\colon\tdgv{\vert}=i\},\qquad
\sbsFi = \{\vert\in\Zbstr\colon\tdgv{\vert}\geq i \} = \bigcup_{j\geq i} \sbsGv{j}.
\]

The total grading of an object of $\FnA$  is defined as the sum of individual $\blbn$ gradings:
$\xGlmtoti \cpA = \bigoplus_{\vert\in\sbsGi} A_{\vert}$. Then there is a forgetful functor
$\fFtot\colon(\FnA)\rightarrow(\FA)$ from $\FnA$ to the filtered \smspl\ category $\FA=\FoA$, which remembers only the total grading and filtration. This functor extends to the categories of complexes:
$\fFtot\colon\ChumA\rightarrow\ChuoA$.
%
%

The total complex $A$ is a \tmcn\ of associated graded complexes $A_{\vert}$: $A$ is the direct sum of complexes $\bigoplus_{\vert\in\Zbstr}A_{\vert}$ deformed by adding new differentials $d_{\vert,\vertp} = \sum_i d_{i;\vert,\vertp}$ for every pair $\vertp>\vert$. In particular, $\cpA$ is a multicone with respect to the \ttdg. Namely, define
\[
\xFlmtoti\cpA = \cpfbvv{ \bigoplus_{\vert\in\sbsFi} A_\vert }{\sum_{\vert,\vertp\in\sbsFi}\ydvvp}
\]
Then there is a recursive cone relation
\begin{equation}
\label{eq:cnrel}
\xFlmtoti\cpA \cong
\xCsv{
\xymatrix@C=3.5cm{
\xGlmtoti\cpA
\ar[r]^-
{\sum_{\vert\in\sbsGi,\;\vertp\in\sbsFv{i+1}}
\ydvvp
}
&
\xFlmtotio\cpA
}
}.
\end{equation}

\begin{theorem}
\label{th:gncntr}
If the complexes $\cpA$ and $\xcpAp$ are \het\ in $\ChFnA$, then
\begin{enumerate}
\item
their associated graded complexes $\cpA_{\vert}$ and $\xcpAp_{\vert}$ are homotopy equivalent in $\ChA$;
\item
length-one differentials $\cpA_{\vert}\xrightarrow{\dudvvpA}\cpA_{\vertp}$ and
$\xcpAp_{\vert}\xrightarrow{\dudvvpB}\xcpAp_{\vertp}$, where $\tdgv{\vertp-\vert}=1$, are homotopic as morphisms in $\ChA$.
\end{enumerate}
\end{theorem}
The proof of this theorem is obvious and we leave it to the reader.

\begin{theorem}
\label{th:hmeq}
Let $A$ be a complex in $\ChFnA$ with a \tdmn\ $\sbsXA$. For a collection of complexes $\Bvr$, $\vert\in\sbsXA$ such that $\Bvr\sim\Avr$ in $\ChA$ there exists a complex $\cpB\sim\cpA$ in $\ChFnA$, whose \cnstc es are $\Bvr$.
%
\end{theorem}

The proof of this theorem is based on a well-known lemma used to show that homotopy categories are triangulated.
\begin{lemma}
\label{lm:lmcn}
For two pairs of \hec\ complexes $\cpA_i\sim\cpB_i$, $i=1,2$ in $\ChA$ and for a morphism 
$\zdAot\colon\cpA_1\rightarrow\cpA_2$ there exists a morphism $\zdBot\colon\cpB_1\rightarrow\cpB_2$ such that their cones are \het:
\[
\bigl[\xymatrix{\cpA_1\ar[r]^-{\zdAot} & \cpA_2}\bigr] \sim
\bigl[\xymatrix{\cpB_1\ar[r]^-{\zdBot} & \cpB_2}\bigr].
\]
\end{lemma}
\begin{proof}
Denote $\cpA = \cpA_1\oplus \cpA_2$, $\xdA = \xdAo + \xdAt$ and similarly for $\cpB$. By the assumption of the lemma there is a \hec
\[
\begin{tikzpicture}
\matrix (m) [smmt, column sep=2.75em]
{(\cpA,\xdA)&(\cpB,\xdB)\\};
\path[->,font=\scriptsize]
(m-1-1) edge [bend left=12] node[auto]{$\fAB$} (m-1-2)
(m-1-1) edge [loop left] node [auto]{$\hmA$} (m-1-1)
(m-1-2) edge [bend left=12] node[auto]{$\fBA$} (m-1-1)
(m-1-2) edge [loop right] node[auto]{$\hmB$} (m-1-2);
\end{tikzpicture}
\]
Choose $\zdBot = \fAB\, \zdAot\, \fBA$, then a straightforward calculation shows that the \hec
\[
\begin{tikzpicture}
\matrix (m) [smmt, column sep=2.75em]
{(\cpA,\xdA+\zdAot)&(\cpB,\xdB+\zdBot)\\};
\path[->,font=\scriptsize]
(m-1-1) edge [bend left=12] node[auto]{$\FAB$} (m-1-2)
(m-1-1) edge [loop left] node [auto]{$\HmA$} (m-1-1)
(m-1-2) edge [bend left=12] node[auto]{$\FBA$} (m-1-1)
(m-1-2) edge [loop right] node[auto]{$\HmB$} (m-1-2);
\end{tikzpicture}
\]
is established by the morphisms
\begin{equation}
\label{eq:heqcn}
\begin{aligned}
\FAB & = \fAB + \fAB \,\zdAot\,\hmA + (\hmB\,\fAB - \fAB\,\hmA)\, \zdAot,
\\
\FBA & = \fBA + \hmA\,\zdAot\, \fBA + \zdAot\,(\fBA\,\hmB - \hmA\,\fBA),
\\
\HmA & = \hmA + \fBA\,(\hmB\,\fAB - \fAB\,\hmA) + \hmA\, \zdAot\,\hmA - \zdAot\,(\hmA)^2,
\\
\HmB & = \hmB.
\end{aligned}
\end{equation}
\end{proof}

\begin{proof}[Proof of Theorem~\ref{th:hmeq}]
Let $\imn = \min\{ \tdgv{\vert},\;\vert\in\sbsXA\}$ and $\imx = \max\{ { \tdgv{\vert},\;\vert\in\sbsXA}\}$. We will prove the theorem for all subcomplexes $\xFlmtoti\cpA$ with \tdmn s $\sbsXAi = \sbsXA\cap\sbsFi$
by induction from $i = \imx + 1$ to $i = \imn$ by going over $i$ `backwards'.

If $i = \imx +1$,
then $\sbsXAi = \varnothing$,
the complex $\xFlmtoti\cpA$ is trivial and the claim of the theorem is obvious.
If $i = \imn$, then $\sbsXAi = \sbsXA$, $\xFlmtoti\cpA = \cpA$ and the veracity of the theorem for $\xFlmtoti\cpA$ implies that it holds for $\cpA$. Now assume that for some value of $i$ the theorem holds for $\xFlmtotio\cpA$. Its extension to $\xFlmtoti\cpA$ is performed with the help of \ex{eq:cnrel} and Lemma~\ref{lm:lmcn} by observing that the differentials, homotopy equivalences and homotopies defined explicitly by the formulas~\eqref{eq:heqcn} are compositions of \mfltd\ morphisms, so they are \mfltd\ themselves and define \hec\ in the category $\ChFnA$.
\end{proof}

From now on till the end of this subsection we make a distinction between the complex $\xcmpAd$, which is an object of $\ChFnA$ and the sum of its chain objects $A$, which is an object of $\FnA$ endowed with an extra (homological) grading. Any endomorphism $\xhen$ of $A$ can be split according to the total degree:
\[
\xhen = \sum_{i = 0}^{\infty} \xheni,\qquad
\xheni =
 \sum_{\substack{\vert,\vertp\in\sbsXA \\
\tdgv{\vertp-\vert}=i}} \xhensv{\vert,\vertp}.
\]
In particular, we split the differential $\cpd$ of the complex $\xcmpAd$: $\cpd = \sum_{i=0}^{\infty} \xbdi$.
%
%
Obviously, $\xbdz = \sum_{\vert\in\sbsXA}\yydvrt$ and $\xcmpAdz = \bigoplus_{\vert\in\sbsXA} \xcmpAdv$.
The condition $\cpd^2=0$ implies relations
\begin{equation}
\label{eq:dsqc}
\xbdi^2 = - \sum_{j=0}^{i-1}[\xbdj, \xbdv{2i-j}],\qquad
\sum_{j=0}^{i}[\xbdj,\xbdv{2i+1-j}]=0.
\end{equation}
These relations at $i=1$ say $[\xbdz,\xbdo]=0$, $\xbdo^2 = -[\xbdz,\xbdt]$, so $\xbdo$
is an element of $\EnChA\xcmpAdz$
with homological degree one and with the property that it is closed and its square is homotopic to zero.

\subsection{\Sbmdl s and \mfspl\ chain complexes}
For the category $\left. \Qbxyfr\right.$ of free \qgrdd\ $\Qbxy$-modules, consider an associated category of \mfspl\ chain complexes $\ChFnQbxyfr$. There is a functor $\fFtotD$
%
%
\begin{equation}
\label{eq:fftot}
\xymatrix{
\ChFnQbxyfr
\ar@/^5pc/[drr]^-{\fFtotD}
\ar[d]
\ar[r]^-{\fFtot}
&
\ChFoQbxyfr
\ar[d]
\ar[dr]^-{\fR}
\\
\KFnQbxyfr
\ar[r]^-{\fFtot}
&
\KFoQbxyfr
\ar[r]
&
\DQfbxy
}
\end{equation}
which factors through the homotopy categories.
The functor $\fR$  is surjective and if $\fR(A) = B$, then $A$ is a free resolution of the object $B$ in the derived category of filtered modules.

Recall that a $\Qbxy$-module $\cgnst$ of \ex{eq:spbmdq} corresponding to a permutation $\gnst\in\Smgn$, has a canonical Koszul resolution $\Prcnxv{\xctv{\gnst}}$ of \ex{eq:cnrsvK}, which is an object in $\ctCh\bigl(\Qbxyfr\bigr)$. To a \tbrgr\ $\brgr=\brgr_{\Ngr}\cdots\brgr_1$
presented as a product of $\Ngr$  elementary \tbrgr s of the type~\eqref{eq:brdgpm}
we will associate a canonical complex $\ccnbr$ from $\ChFnQbxyfr$ such that $\fFtotD\ccnbr \cong \cbrgr$, that is, $\fFtot\ccnbr$ is a free filtered resolution of $\cbrgr$.

As a first step towards $\ccnbr$ we define a different complex $\cfrbr$ from $\ChFnQbxyfr$ with the same property
$\fFtot\cfrbr\cong\cbrgr$.
%
%
If $\brgr$ is an elementary \tbrgr~\eqref{eq:brdgpm}, then $\blbn=1$ and we define $\cfrbr=\ccnbr$,
where $\ccnbr$ is defined by \ex{eq:cnrsml}
and the single filtration of $ \ctChuv{1}$ is the filtration of $\ccnbr$. For a general $\brgr=\brgr_{\Ngr}\cdots\brgr_1$ we define $\cfrbr$ as the product of complexes $\ccnbri$ over the intermediate variables, the $i$-th filtration of $\Chum$ being the filtration of the $i$-th factor $\ccnbri$.

Here is a simple example of $\cfrbr$ for $\Ngr=2$:
\def\lngtp#1#2
{
\Prcnxv{\xctvbzy{#1}}
\otQby
\Prcnxv{\xctvbyx{#2}}
}
\begin{equation}
\label{eq:exlngr}
\yectfrbzx{\zblgblb{\lncf}{-11}} =
\boxed{
\vcenter{
\xymatrix@C=1.5cm@R=1.5cm{
\lngtp{\zparlv{\lncf}{5}}{\zvirtv{\lncf}{5} }
\ar[r]^-{\vsdp\otimes\xId}
\ar[d]^-{\xId\otimes\vsdm}
&
\ssqt\;\lngtp{\zvirtv{\lncf}{5}}{\zvirtv{\lncf}{5} }
\ar[d]^-{\xId\otimes\vsdm}
\\
\ssqt\;\lngtp{\zparlv{\lncf}{5}}{\zparlv{\lncf}{5} }
\ar[r]^-{-\vsdp\otimes\xId}
&
\ssqt^2\;\lngtp{\zvirtv{\lncf}{5}}{\zparlv{\lncf}{5} }
}
}
}
\end{equation}

The constituent complexes $\cfrbrvr$ of $\cfrbr$ sit at the vertices of the unit $\blbn$-dimensional cube $\xCbm\subset\Zbstr$ and they are tensor products of canonical resolutions appearing in the cones~\eqref{eq:cnrsml} shifted by $\ssqt^{\tdv}$. The coordinates of a vertex $\vert$ of the cube determine which of the complexes $\ctcnIdptn$ or $\ctcngnsti$ appear at each position in this tensor product.
Let $\xdot = \xdotz+\xdoto$ be the differential of the complex $\cfrbr$ expanded according to the total degree. $\xdotz$ is the sum of differentials of the \cnstc es $\cfrbrvr$, while $\xdoto=\sum_{\edg}\xdote$ is the sum of components corresponding to the edges of the cube $\xCbm$ having the form $\xdote = \pm\xId^{\otimes i}\otimes \vsd\otimes \xId^{\otimes (\blbn - i -1)}$, where $\vsd$ is a \tvsd\ morphism $\vsdp$ or $\vsdm$.

Let $\gnstvert$ be the product of permutations $\Idptn$ or $\gnsti$ as they appear in the tensor product of complexes $\ctcnIdptn$ and $\ctcngnsti$ at $\vert$. Then there is a \hec\ $ \cfrbrvr\sim \svbrcn$ since, in view of \ex{eq:tpcmp}, both complexes are resolutions of the bimodule $\cgnstvert$. Hence, according to Theorem~\ref{th:hmeq}, the complex $\cfrbr$ is \het\ within the category $\ChFnQbxyfr$ to a complex $\ccnbr$, whose \cnstc es are $\ssqt^{\tdv}\svbrcn$ and whose differentials of positive total degree are determined by the choice of contraction homotopy. In particular, the complex~\eqref{eq:exlngr} transforms into the complex~\eqref{eq:sqdiag}.

Let $\cpd = \sum_{i=0}^{\blbn} \xbdi$ be the differential of the complex $\ccnbr$ expanded according to the total degree. The component $\xbdz$ is the sum of the differentials of the canonical complexes $\svbrcn$, hence it is predetermined. The component $\xbdo$ is the sum of components corresponding to the edges of the cube $\xCbm$: $\xbdo = \sum_{\edg} \xbde$.
If two vertices $\vert,\vertp\in\xCbm$ are connected by an edge $\edg$, then there exists a transposition $\gnstedg$ such that $\gnstvertp = \gnstvert\gnstedg$ and
in view of the commutative diagram~\eqref{eq:cmdgs} and Theorem~\ref{th:gncntr},
after the contraction of constituent complexes the differential of the edge $\edg$ becomes (up to a sign) homotopic to the \tvsd\ morphism $\vsde$  between
the resolutions: $\xbde \sim \pm \vsde$.
The contraction might also generate longer differentials $\xbdi$, $i\geq 2$ in $\ccnbr$. We show that the total degree one differentials of $\ccnbr$ are (up to a sign) precisely the \tvsd\ morphisms $\vsde$: $\xbde = \pm\vsde$ and differentials of higher total degree are absent.
\begin{theorem}
\label{th:cncp}
The differential $\cpd$ of the complex $\gnC$ has the following canonical form:
$\xbde=\pm \vsde$ and $\xbdi=0$ for $i\geq 2$.%
\end{theorem}
For example, according to this theorem,
\begin{equation}
\label{eq:sqdiag}
\yectcnv{\zblgblb{\lncf}{-11}} =
\boxed{
\vcenter{
\xymatrix{
\Prcnxv{\sfvirt}
\ar[r]^-{\vsdm}
\ar[d]_-{\vsdp}
&
\ssqt\,\Prcnxv{\sfparl}
\ar[d]^-{\vsdp}
\\
\ssqt\,
\Prcnxv{\sfparl}
\ar[r]_-{-\vsdm}
&
\ssqt^2\,
\Prcnxv{\sfvirt}
}
}
}
\end{equation}
The proof of Theorem~\ref{th:cncp} is based on the following lemma.
\begin{lemma}
If $\yhm$ is a $\Qbxy$-endomorphism of the complex
$\gnC$
and $\yhm$ is homogeneous with respect to \tadgr, \tqdgr\ and total degree, then
\begin{equation}
\label{eq:qbnd}
\dgq \yhm \geq 2 (\dgtot \yhm + \dga \yhm ).
\end{equation}
\end{lemma}
\begin{proof}
It is easy to see that the \Kszl\ resolution $\Prcnxv{\xctv{\gnst}}$ of \ex{eq:cnrsvK} has the form
\[
\Prcnxv{\xctvbyx{\gnst}} = \bigl(\saqt^{\blbn}\frFv{\blbn}\rightarrow\cdots\rightarrow
\saqt^i\frFi \rightarrow\cdot\rightarrow \frFv{0} \bigr),
\qquad\text{where}\quad
\frFi = \underbrace{\Qbxy\oplus\cdots\oplus\Qbxy}_{{\blbn \choose i}\;\text{times}}.
\]
Therefore, for any homogeneous homomorphism $\yhm\in\Hom_{\Qbxy} (\Prcnxv{\xctvbyx{\gnst}},
\Prcnxv{\xctvbyx{\gnstp}})$, there is a bound $\dgq h \geq 2\dga h$.
Hence for any homogeneous homomorphism
\begin{equation}
\label{eq:hmvvp}
\yhm\in\Hom_{\Qbxy}\bigl(\ssqt^{\tdv}\Prcnxv{\cgnstvert},
\ssqt^{\tdvp}\Prcnxv{\cgnstvertp}\bigr)
\end{equation}
between two \cnstc es of $\gnC$ there is a bound
\begin{equation}
\label{eq:smbnd}
\dgq h \geq 2(\tdgv{\vertp-\vert} + \dga h),
\end{equation}
 and it implies the bound~\eqref{eq:qbnd}.
\end{proof}
\begin{proof}[Proof of Theorem~\ref{th:cncp}]
Since $\xbde\sim\pm \vsde$, there should exist a homomorphism~\eqref{eq:hmvvp}, $\vert$ and $\vertp$ being the endpoints of $\edg$, such that
\[
\xbde = \pm\vsde + \yydvrtp\, \yhm + \yhm\,\yydvrt,\qquad
\dgq \yhm = 0,\quad\dga\yhm = 0,
\]
Since $\tdgv{\vertp-\vert} = 1$, the degree requirements on $\yhm$ are incompatible with the bound~\eqref{eq:smbnd}, hence $\yhm=0$ and $\xbde = \pm \vsde.$

A component $\xbdi$ of the differential $\cpd$ of the complex $\gnC$ is a homogeneous endomorphism of $\gnC$ with degrees
\[
\dgq \xbdi = 0,\quad\dga\xbdi = -1,\quad\dgtot \xbdi = i,
\]
and if $i\geq 2$, then they are incompatible with the bound~\eqref{eq:qbnd}, which means that $\xbdi = 0$.
\end{proof}

\subsection{Multi-filtered \Rq\ complexes}


For an additive category $\ctA$ and for a positive integer $\blbn$ we define a category of \mcpl es $\ChbumA$. Its objects are pairs $\xcmqAd$, where $\cpA$  is an object of the $\blbn$-graded category over $\ctA$: $\cpA = \bigoplus_{\vert\in\sbsXA} A_{\vert}$, where $\sbsXA\subset\Zbstr$ is a finite subset, $\dgtt=\deg_1+\cdots +\deg_{\blbn}$ being the homological degree, while $\cqd\in\End_{\ctA}(\cpA)$, $\cqd = \cqdo + \cdots \cqdm$, $\deg_i \cqdj = \delta_{ij}$, where $\delta_{ij}$ is the Kronecker symbol and $[\cqdi,\cqdj]=0$. For our present purposes we don't have to define morphisms in $\ChbumA$. There is an obvious forgetful functor $\fFtt\colon\ChbumA \rightarrow \ChA$. A combination of this functor and the forgetful functor $\fFtotD$ of \ex{eq:fftot} yields a forgetful functor
\[
\xymatrix@C=2cm{
\ChbmFnQbxyfr
\ar[d]^-{\fFtt}
\ar[dr]^-{\fFttD}
\\
\ChboFnQbxyfr
\ar[r]^-{\fFtotD}
&
\ChDQfbxy
}
\]
which factors through the category $\ChboFnQbxyfr$.

For a \brwd\ $\brwb$ we define a canonical object $\ccnbw$ of the category $\ChbmFnQbxyfr$ such that $\fFttD\ccnbw\cong \cbrwb$. First, we define the object $\cfrbw$ in the same category by taking the tensor product of \Rqc es~\eqref{eq:rqco} and~\eqref{eq:rqct} over the algebras of intermediate variables selected in accordance with the content of $\brwb$. The degree of each complex is considered as independent, so the result (up to an overall shift in \tmtdg) is the complex supported in the unit cube $\xCbmp\subset\Zbstr$.

The coordinates $0$ and $1$ in an $i$-th coordinate direction correspond to a replacement of the $i$-th elementary factor in the \brwd\ $\brwb$ by either the identity braid $\Idptn$ or by a \blgn\ $\fblbj$ or $\fblgj$, where $j$ is the index of the left strand participating in the $i$-th elementary crossing.
To each vertex $\wert\in\xCbmp$ corresponds a \tbrgr\ $\brgrw$, which is the composition of identity braids and \blgn s selected according to its coordinates. A \cnst\ object at $\wert$ is the complex $\cfrbraw$.
At an edge $\etdg$ connecting the vertices $\brgrw$ and $\brgrwp$ sits a component of the differential $\cnqdte\colon\cfrbraw \rightarrow\cfrbrawp$, which is a canonical morphism corresponding to the $i$-th filtration.

In order to obtain the canonical complexes $\ccnbw$ we replace the \cnst\ objects $\cfrbraw$ with \het\ objects $\ccnbrw$. According to Theorem~\ref{th:cnhmt}, the definition of a canonical morphism extends to the homotopy category, so the differential $\cqde$ at the edge $\etdg$ is still
the canonical morphism associated with $i$-th filtration.

The following example illustrates the structure of the canonical complex $\ccnbw$:
\begin{eqnarray}
\sfcnbrpbrm  & = &
\left[
\begin{tikzcd}[column sep=2.25em]
\sfcnparblb
\rar
\dar
&
\sfcnparpar
\dar
\\
\sfcnblgblb
\rar
&
\sfcnblgpar
\end{tikzcd}
\right]
=
\left[
\hspace{-0.4cm}
\vcenter{
\begin{tikzpicture}
\node (P11) at (3.75,-3) 
{
\begin{tikzpicture}[baseline=(current bounding box.base)]
\matrix (m) [smmt]
{ \sfnvirt \\
\sfnparl \\};
\path[->,font=\scriptsize]
(m-1-1.250) edge node[auto] {} (m-2-1.110);
\node[draw,fit=(m-1-1) (m-2-1) ] {};
\end{tikzpicture}
};
\node (P01) at (0,-3) 
{
\begin{tikzpicture}[baseline=(current bounding box.base)]
\matrix (m) [smmt]
{\sfnvirt & \sfnparl \\
\sfnparl & \sfnvirt \\};
\path[->,font=\scriptsize]
(m-1-1.250) edge node[auto] {} (m-2-1.110)
(m-1-1) edge node[auto] {} (m-1-2)
(m-1-2.250) edge node[auto] {} (m-2-2.110)
(m-2-1) edge node[auto] {} (m-2-2);
\node[draw,fit=(m-1-1) (m-2-2),use as bounding box ] {};
\end{tikzpicture}
};
\node (P10) at (3.75,0) 
{
\begin{tikzcd}
\sfnparl
\end{tikzcd}
};
\node (P00) at (0,0) 
{
\begin{tikzpicture}[commutative diagrams/every diagram,baseline=(current bounding box.base)]
\matrix (m) [smmt]
{ \sfnparl
& \sfnvirt \\};
\path[->,font=\scriptsize]
(m-1-1) edge node[auto] {} (m-1-2);
\node[draw,fit=(m-1-1) (m-1-2) ] {};
\end{tikzpicture}
};
\path[commutative diagrams/.cd, every arrow, every label]
(P00) edge node {} (P10)
(P00) edge node {} (P01)
(P10) edge node {} (P11)
(P01) edge node {} (P11);
\end{tikzpicture}
}
\hspace{-3.7in}
\right]
\end{eqnarray}
We omitted all shift functors in order to emphasize the structure of \cnst\ complexes $\ccnbrw$.

\section{Second Reidemeister move}
\label{s.srm}
\begin{theorem}
\label{th:R2}
The following complexes are \het\ in $\ChDQfbxy$:
\begin{equation}
\label{eq:hcrtw}
\sfbrpbrm\sim\sfparl,\qquad
\sfbrmbrp\sim\sfparl.
\end{equation}
\end{theorem}
The proof is based on the lemma:
\begin{lemma}
\label{lm:splto}
The `double \blb' bimodules split;
\[
\sfblgblb \cong\sfblbg \oplus\ssqt\, \sfblbb,\qquad
\sfblbblg \cong \sfblbg \oplus\ssqt\, \sfblbb,
\]
and the following diagrams are commutative:
\begin{equation}
\label{eq:twcmdgs}
\xymatrix{
\sfparblb
\xari[r]^-{\chp\otimes\xId}
\ar[d]^-{\cong}
&
\sfblgblb
\xars[r]^-{\xId\otimes\chm}
\ar[d]^-{\cong}
&
\sfblgpar
\ar[d]^-{\cong}
\\
\sfblbb
\ar[r]^-{{\xsmmt{0 \\ \xId}}}
&
\sfblbg\oplus\ssqt\sfblbb
\ar[r]^-{{\xsmmt{\xId & \xpsim}}}
&
\sfblbg
}
\qquad
\xymatrix{
\sfblbpar
\xari[r]^-{\xId\otimes\chp}
\ar[d]^-{\cong}
&
\sfblbblg
\xars[r]^-{\chm\otimes\xId}
\ar[d]^-{\cong}
&
\sfparblg
\ar[d]^-{\cong}
\\
\sfblbb
\ar[r]^-{{\xsmmt{0 \\ \xId}}}
&
\sfblbg\oplus\ssqt\sfblbb
\ar[r]^-{{\xsmmt{\xId & \xpsim}}}
&
\sfblbg
}
\end{equation}
\end{lemma}
\begin{proof}
We will prove only the first splitting, the second one is proved similarly. We represent the \lhs bimodule by its canonical resolution~\eqref{eq:sqdiag} and then replace it by the isomorphic complex which splits:
\begin{equation}
\label{eq:dcmpo}
\begin{tikzpicture}[baseline=-0.25em]
\matrix (m) [ssmt,row sep=6.5em,column sep=1.5cm,ampersand replacement=\&]
{
\sfcnblgblb
\&
\sfnvirt
\&
\ssqt\bigl(\sfnparl \oplus \sfnparl\bigr)
\&
\ssqt^2\, \sfnvirt
\\
\sfnblbg \oplus\ssqt\, \sfnblbb
\&
\sfnvirt
\&
\ssqt\bigl(\sfnparl \oplus \sfnparl\bigr)
\&
\ssqt^2\, \sfnvirt
\\
};
\def\yshu{-1.25em}
\def\yshd{2.25em}
\node (Q1) at (m-1-2) [yshift=1.25em] {};
\node (Q2) at (m-2-2) [yshift=1.25em] {};
\node  (P12) at (m-1-2) [yshift=\yshu] {};
\node (P22) at (m-2-2) [yshift=\yshd] {};
\node  (P13) at (m-1-3) [yshift=\yshu] {};
\node (P23) at (m-2-3) [yshift=\yshd] {};
\node  (P14) at (m-1-4) [yshift=\yshu] {};
\node (P24) at (m-2-4) [yshift=\yshd] {};
\node (R11) at (m-1-1) [yshift=-1.25em] {};
\node at (m-1-1) [xshift=4.75em] {$\cong$};
\node at (m-2-1) [xshift=4.75em] {$\cong$};
\path[->,font=\footnotesize]
(m-1-2) edge node[auto] {$\bigl( \begin{smallmatrix}\vsd \\ \vsd \end{smallmatrix} \bigr)$} (m-1-3)
(m-2-2) edge node[auto] {$\bigl( \begin{smallmatrix} 0  \\ \vsd \end{smallmatrix} \bigr)$} (m-2-3)
(m-1-3) edge node[auto] {$\bigl( \begin{smallmatrix}\vsd & -\vsd \end{smallmatrix} \bigr)$} (m-1-4)
(m-2-3) edge node[auto] {$\bigl( \begin{smallmatrix}\vsd & 0 \end{smallmatrix} \bigr)$} (m-2-4)
(R11) edge node[auto] {$\cong$} (m-2-1)
(P12) edge node[auto] {$\xId$} (P22)
(P13) edge node[auto] {${\bigl( \begin{smallmatrix}1 & -1 \\ 0 &\; \;\;1 \end{smallmatrix} \bigr)}$} (P23)
(P14) edge node[auto] {$\xId$} (P24);
\node[draw,fit=(Q1) (m-1-2) (m-1-4) ] {};
\node[draw,fit=(Q2) (m-2-2) (m-2-4) ] {};
\end{tikzpicture}
\end{equation}
The lower lines of the diagrams~\eqref{eq:twcmdgs} are established by a straightforward computation.
\end{proof}
\begin{proof}[Proof of Theorem~\ref{th:R2}]
We prove only the first \hec\ of~\eqref{eq:hcrtw}, the second one is proved similarly.
The proof is well-known, we just have to verify that the isomorphisms and homotopies appearing there are filtered:
\[
\sfbrpbrm  \cong
\left[
\begin{tikzcd}[row sep=3em]
\sfparblb
\arrow{r}{\xId\otimes\chm}
\arrow{d}{\chp\otimes\xId}
&
\sfparpar
\arrow{d}{\chp\otimes\xId}
\\
\shs^{-1}\sfblgblb
\arrow{r}{-\xId\otimes\chm}
&
\shs^{-1}\sfblgpar
\end{tikzcd}
\right]
\cong
\left[
\begin{tikzcd}[ampre,row sep=3em]
\sfblbb
\rar{\chm}
\dar{\bigl( \begin{smallmatrix}0 \\ \xId\end{smallmatrix} \bigr)}
\&
\sfparl
\dar{\chp}
\\
\shs^{-1}\sfblbg \oplus\sfblbb
\rar{{\xsmmt{\xId & \xpsim}}}
\&
\shs^{-1}\sfblbg
\end{tikzcd}
\right]
\sim
\sfparl.
\]
The last homotopy comes from contracting the subcomplex $\shs^{-1}\sfblbg\xrightarrow{\;\xId\;}\shs^{-1}\sfblbg$ and then contracting the quotient complex
$\sfblbb\xrightarrow{\;\xId\;}\sfblbb$. Both complexes are filtered contractible.
\end{proof}

\section{Third Reidemeister move}
\label{s:thrm}
\begin{theorem}
\label{th:trrx}
There exists a \hec\
\begin{equation}
\label{eq:thrdh}
\begin{tikzpicture}
\matrix (m) [smmt, column sep=2.75em]
{\yectv{\zzvv{\lncf}{-25}{\dbrdm}{\dbrdm}{\dbrdm}}
&
\yectv{\zzmvv{\lncf}{-25}{\dbrdm}{\dbrdm}{\dbrdm}}\\};
\path[->,font=\scriptsize]
([yshift=\shar] m-1-1.east)  edge node[auto]{$\flr$} ([yshift=\shar] m-1-2.west)
([yshift=-\shar] m-1-2.west) edge node[auto]{$\frl$} ([yshift=-\shar] m-1-1.east);
\end{tikzpicture}
\end{equation}
such that the \hec\ morphisms $\flr$ and $\frl$ violate filtration by at most two units.
\end{theorem}

\begin{proof}
The general strategy is standard. The braids of \ex{eq:thrdh} are related by the reflection symmetry with respect to the vertical axis. We construct a special reflection symmetric complex $\cpCsm$ and establish two \hec s
\begin{equation}
\label{eq:thrdc}
\begin{tikzpicture}
\matrix (m) [smmt, column sep=2.75em]
{\yectv{\zzvv{\lncf}{-25}{\dbrdm}{\dbrdm}{\dbrdm}}
&
\cpCsm
&
\yectv{\zzvv{\lncf}{-25}{\dbrdm}{\dbrdm}{\dbrdm}}\\};
\path[->,font=\scriptsize]
([yshift=\shar] m-1-1.east)  edge node[auto]{$\fls$} ([yshift=\shar] m-1-2.west)
([yshift=-\shar] m-1-2.west) edge node[auto]{$\fsl$} ([yshift=-\shar] m-1-1.east)
([yshift=\shar] m-1-2.east)  edge node[auto]{$\fsr$} ([yshift=\shar] m-1-3.west)
([yshift=-\shar] m-1-3.west) edge node[auto]{$\frs$} ([yshift=-\shar] m-1-2.east);
\end{tikzpicture},
\end{equation}
such that $\fsl$ and $\fsr$ are filtered, while $\fls$ and $\frs$ violate filtration by two units. By symmetry, it is sufficient to establish only the left \hec, and the claim of the theorem follows.

Consider the left complex in the diagram~\eqref{eq:thrdc}:
\begin{eqnarray}
\label{eq:bgthrc}
\sthmmmo
& = &
\left[
\begin{tikzpicture}[baseline=(current bounding box.base)]
\node (N) at (0,0)
{
\begin{tikzcd}[column sep=0.7cm]
&
\sftlelob
\arrow{r}
\arrow{dr}
&
\sflblbbp
\arrow{dr}
\arrow[crossing over, leftarrow]{dl}
\\
\sftvrbb
\arrow{ur}
\arrow{r}
\arrow{dr}
&
\sfdblbbp
&
\sflpblbb
\arrow{r}
&
\sftparl
\\
&
\sftrilob
\arrow{r}
\arrow{ur}
&
\sflblbbp
\arrow{ur}
\arrow[ crossing over, leftarrow]{ul}
\end{tikzcd}
};
\end{tikzpicture}
\right]
\\
\nonumber
&=&
\Biggl[
\begin{tikzcd}[column sep=2cm]
\sthmmb
\arrow{r}{\xId\otimes\chm\otimes\xId}
&
\sthmmp
\end{tikzcd}
\Biggr]
\end{eqnarray}
Here unlabelled arrows correspond to the morphisms $\chm$ of the type~\eqref{eq:tpchm}.
The cone in the last line comes from splitting the cube complex into the back face and the front face. The second complex of the cone can be simplified with the help of two lemmas. The first one is the analog of Lemma~\ref{lm:splto}:
\begin{lemma}
\label{lm:decdblbb}
The following filtered bimodule splits:
\begin{equation}
\label{eq:spldblb}
\sfblbblb \cong\sfblbb \oplus\ssqt\, \sfblbg,
\end{equation}
and the following diagram is commutative:
\[
\xymatrix{
\sfparblb
\ar[d]^-{\cong}
&
\sfblbblb
\xars[l]_-{\chm\otimes\xId}
\xars[r]^-{\xId\otimes\chm}
\ar[d]^-{\cong}
&
\sfblbpar
\ar[d]^-{\cong}
\\
\sfblbb
&
\sfblbb\oplus
\ssqt\,
\sfblbg
\ar[r]^-{{\xsmmt{\xId & 0}}}
\ar[l]_-{{\xsmmt{\xId & \xpsip}}}
&
\sfblbb
}
\]
\end{lemma}
\begin{proof}
The proof is similar to that of Lemma~\ref{lm:splto}, the analog of the diagram~\eqref{eq:dcmpo} being
\begin{equation}
\label{eq:dcmpt}
\begin{tikzpicture}[baseline=-0.25em]
\matrix (m) [ssmt,row sep=6.5em,column sep=1.5cm,ampersand replacement=\&]
{
\sfcnblbblb
\&
\sfnparl
\&
\ssqt\bigl(\sfnvirt \oplus \sfnvirt\bigr)
\&
\ssqt^2\, \sfnparl
\\
\sfnblbb \oplus\ssqt\, \sfnblbg
\&
\sfnparl
\&
\ssqt\bigl(\sfnvirt \oplus \sfnvirt\bigr)
\&
\ssqt^2\, \sfnparl
\\
};
\def\yshu{-1.25em}
\def\yshd{2.25em}
\node (Q1) at (m-1-2) [yshift=1.25em] {};
\node (Q2) at (m-2-2) [yshift=1.25em] {};
\node  (P12) at (m-1-2) [yshift=\yshu] {};
\node (P22) at (m-2-2) [yshift=\yshd] {};
\node  (P13) at (m-1-3) [yshift=\yshu] {};
\node (P23) at (m-2-3) [yshift=\yshd] {};
\node  (P14) at (m-1-4) [yshift=\yshu] {};
\node (P24) at (m-2-4) [yshift=\yshd] {};
\node (R11) at (m-1-1) [yshift=-1.25em] {};
\node at (m-1-1) [xshift=4.75em] {$\cong$};
\node at (m-2-1) [xshift=4.75em] {$\cong$};
\path[->,font=\footnotesize]
(m-1-2) edge node[auto] {$\bigl( \begin{smallmatrix}\vsd \\ \vsd \end{smallmatrix} \bigr)$} (m-1-3)
(m-2-2) edge node[auto] {$\bigl( \begin{smallmatrix} 0  \\ \vsd \end{smallmatrix} \bigr)$} (m-2-3)
(m-1-3) edge node[auto] {$\bigl( \begin{smallmatrix}\vsd & -\vsd \end{smallmatrix} \bigr)$} (m-1-4)
(m-2-3) edge node[auto] {$\bigl( \begin{smallmatrix}\vsd & 0 \end{smallmatrix} \bigr)$} (m-2-4)
(R11) edge node[auto] {$\cong$} (m-2-1)
(P12) edge node[auto] {$\xId$} (P22)
(P13) edge node[auto] {${\bigl( \begin{smallmatrix}1 & -1 \\ 0 &\; \;\;1 \end{smallmatrix} \bigr)}$} (P23)
(P14) edge node[auto] {$\xId$} (P24);
\node[draw,fit=(Q1) (m-1-2) (m-1-4) ] {};
\node[draw,fit=(Q2) (m-2-2) (m-2-4) ] {};
\end{tikzpicture}
\end{equation}
\end{proof}
\begin{lemma}
There is a \hec
\begin{equation}
\label{eq:dblbrhe}
\sfbrmbrm \sim
\left[
\begin{tikzcd}
\ssqt\,\sfblbg
\arrow{r}{\xpsip}
&
\sfblbb
\arrow{r}{\chm}
&
\sfparl
\end{tikzcd}
\right].
\end{equation}
\end{lemma}
\begin{proof}
Consider the isomorphism of complexes
\begin{equation}
\label{eq:dcmph}
\begin{tikzcd}[row sep=1cm,column sep=1cm,ampersand replacement=\&]
\&
\Biggl[\sfblbblb
\arrow
{d}{\cong}
\arrow{rr}{\bigl(\begin{smallmatrix}\chm\otimes\xId \\ \xId\otimes\chm \end{smallmatrix}\bigr)}
\&\&
\sfblbb\oplus\sfblbb
\arrow{rr}{{( \begin{smallmatrix}\chm, & - \chm \end{smallmatrix} )}}
\arrow{d}{{\bigl( \begin{smallmatrix}1 & -1 \\ 0 &\; \;\;1 \end{smallmatrix} \bigr)}}
\&\&
\sfparl\Biggr]
\arrow{d}{\xId}
\\
\&
\Bigl[\sfblbb \oplus\ssqt\, \sfblbg
\arrow{rr}{{\bigl( \begin{smallmatrix}  0 &\xpsip \\ \xId & 0  \end{smallmatrix} \bigr)}}
\&\&
\sfblbb\oplus\sfblbb
\arrow{rr}{{( \begin{smallmatrix}\chm & 0 \end{smallmatrix} )}}
\&\&
\sfparl\Bigr]
\end{tikzcd}
\end{equation}
The complex at the top line represents $\sfbrmbrm$.
The complex at the bottom line is a sum of the contractible complex $\bigl[ \sfblbb\xrightarrow{\;\xId\;}\sfblbb \bigr]$ and the complex in the \rhs of \ex{eq:dblbrhe}.
\end{proof}
 Adding an extra strand to the \hec~\eqref{eq:dblbrhe} turns it into
 \begin{equation}
\label{eq:dblbrht}
\sfbrmbrmp \sim
\left[
\begin{tikzcd}
\ssqt\,\sfblbgp
\arrow{r}{\xpsip}
&
\sfblbbp
\arrow{r}{\chm}
&
\sfparlp
\end{tikzcd}
\right].
\end{equation}
A substitution of the \rhs complex for the front face of the cube complex~\eqref{eq:bgthrc} turns the latter into the complex
\begin{eqnarray}
\label{eq:trintcn}
\sthmmmo & \sim &
\left[
\begin{tikzcd}[row sep=0cm,column sep=0.7cm]
&
\shs\sflblbgp
\arrow{ddddr}{\xpsip}
&
\\
{}
\\
{}
\\
{}
\\
&
\sftlelob
\arrow{r}
\arrow{ddr}
&
\sflblbbp
\arrow{dr}
\\
\sftvrbb
\arrow{uuuuur}{\xxi}
\arrow{ur}
\arrow{dr}
&
{}
&
{}
&
\sftparl
\\
&
\sftrilob
\arrow[crossing over]{uur}
\arrow{r}
&
\sflpblbb
\arrow{ur}
\end{tikzcd}
\right]
\\
\label{eq:trbrscn}
&\sim &
\left[
\begin{tikzcd}
\cpA \arrow{r}{\xkapt}
&
\cpB
\end{tikzcd}
\right].
\end{eqnarray}
Here $\xxi$ is a composition of morphisms
\begin{equation}
\label{eq:xxi}
\begin{tikzcd}[column sep=1.5cm]
\sftvrbb
\arrow{r}{\xId\otimes\chm\otimes\xId}
\arrow[bend left]{rr}{\xxi}
&
\sfdblbbp
\arrow{r}{\prg}
&
\shs\sflblbgp,
\end{tikzcd}
\end{equation}
while
\begin{equation}
\label{eq:cpsAB}
\cpA = \left[
\begin{tikzcd}
\sftvrbb \arrow{r}{\xxi}
&
\shs\sflblbgp
\end{tikzcd}
\right],
\qquad
\cpB =
\left[
\begin{tikzcd}[row sep=0cm,column sep=0.7cm]
\sftlelob
\arrow{r}
\arrow{ddr}
&
\sflblbbp
\arrow{dr}
\\
{}
&
{}
&
\sftparl
\\
\sftrilob
\arrow[crossing over]{uur}
\arrow{r}
&
\sflpblbb
\arrow{ur}
\end{tikzcd}
\right]
\end{equation}
and the morphism $\xkapt$ in the cone~\eqref{eq:trbrscn} is determined by the diagram~\eqref{eq:trintcn}.

It is well-known (see \eg Lemma~4.7 in~\cite{KRpw} and references therein) that \emph{if we ignore filtration}, then the bimodule appearing at the head of the complex~\eqref{eq:xxi} splits:
\begin{equation}
\label{eq:splstr}
\sftvrbb \cong \sftvrtb \oplus \sflblbgp,
\end{equation}
where
\[
\sftvrtb=\Qbxy/(\dfbxyv{1},\dfbxyv{2},\dfbxyv{3})
\]
and  $p_k(\bfx) = x_1^k + x_2^k + x_3^k$. Moreover,
\begin{equation}
\label{eq:kercng}
\sftvrtb \cong \ker\xxi
\end{equation}
and $\xxi$ acts as an isomorphism on the second component of the sum~\eqref{eq:splstr}.
In terms of the derived category, the isomorphism~\eqref{eq:kercng} says that there is a \qiso
\[
\sftvrtb\simeq\Conv{\xxi}
\]
and the splitting~\eqref{eq:splstr} says that one of the canonical maps in the corresponding distinguished triangles is trivial:
\begin{equation}
\label{eq:dstto}
\begin{tikzcd}
\sflblbgp
\arrow{r}{0}
&
\Conv{\xxi}\cong\sftvrtb
\arrow{r}
&
\sftvrbb
\arrow{r}{\xxi}
&
\sflblbgp.
\end{tikzcd}
\end{equation}

Now consider the \cres\ of the bimodule~\eqref{eq:splstr}:
%
%
%
\begin{eqnarray}
\label{eq:lngsmth}
\sthbcn & = &
\begin{tikzpicture}[baseline=(current bounding box.base)]
\node (N) at (0,0)
{
\begin{tikzcd}[column sep=0.7cm]
&
\shs\sfcntvirp
\arrow{r}
\arrow{dr}
&
\shst
\sfcntrilo
\arrow{dr}
\arrow[crossing over, leftarrow]{dl}
\\
\sfcntparl
\arrow{ur}
\arrow{r}
\arrow{dr}
&
\shs
\sfcntpvir
&
\shst\sfcntparl
\arrow{r}
&
\shsh\sfcntvrtr
\\
&
\shs\sfcntvirp
\arrow{r}
\arrow{ur}
&
\shst\sfcntlelo
\arrow{ur}
\arrow[crossing over, leftarrow]{ul}
\end{tikzcd}
};
\node[draw,fit=(N) ] {};
\end{tikzpicture}
\\
\nonumber
&=&
\begin{tikzpicture}[baseline=(current bounding box.base)]
\node (N) at (0,0)
{
\begin{tikzcd}[column sep=2cm]
\sthbbpcn
\arrow{r}{\xId\otimes\vsdp\otimes\xId}
&
\shs\sthbbvcn
\end{tikzcd}
};
\node[draw,fit=(N) ] {};
\end{tikzpicture}
\end{eqnarray}
The arrows in the large boxed diagram correspond to \tvsd s similar to~\eqref{eq:ttvsd}.
The last cone comes from the presentation~\eqref{eq:rqct} for the middle \blb, its components correspond to the back and front faces of the cube complex. Decomposing the first component of the cone in accordance with Lemma~\ref{lm:decdblbb}, we turn the cube complex~\eqref{eq:lngsmth} into the following complex
%
%
\begin{eqnarray}
\label{eq:lngsmtf}
\sthbcn & \cong &
\begin{tikzpicture}[baseline=(current bounding box.base)]
\node (N) at (0,0)
{
\begin{tikzcd}[row sep=0cm,column sep=0.7cm]
&
\shs\sfcntvirp
\arrow{r}
\arrow{ddddr}
&
\shst\sfcntparl
\arrow{dddddr}
\\
{}
\\
{}
\\
{}
\\
&
\shs\sfcntvirp
\arrow{r}
\arrow{ddr}
&
\shst\sfcntrilo
\arrow{dr}
\\
\sfcntparl
\arrow{ur}
\arrow{dr}
&
{}
&
{}
&
\shsh\sfcntvrtr
\\
&
\shs\sfcntpvir
\arrow[crossing over]{uur}
\arrow{r}
&
\shst\sfcntlelo
\arrow{ur}
\end{tikzcd}
};
\node[draw,fit=(N) ] {};
\end{tikzpicture}
\\
\label{eq:trvcn}
&=&
\begin{tikzpicture}[baseline=(current bounding box.base)]
\node (N) at (0,0)
{
\begin{tikzcd}[column sep=1cm]
\shs\, \sflcnblbgp
\arrow{r}{\xfxi}
&
\xcpCs
\end{tikzcd}
};
\node[draw,fit=(N) ] {};
\end{tikzpicture}
\end{eqnarray}
The first component of the cone~\eqref{eq:trvcn} is the complex in the first row in the big diagram, while the second component is the complex in the other two bottom rows. Since the trasformation of the diagram~\eqref{eq:lngsmth} into the diagram~\eqref{eq:lngsmtf} is performed by the splitting~\eqref{eq:spldblb}, which is a part of the morphism $\xxi$ in the diagram~\eqref{eq:xxi}, it follows that $\xxi$ is one of the canonical morphisms in the distinguished triangle of the cone~\eqref{eq:trvcn}:
\begin{equation}
\label{eq:dsttt}
\begin{tikzcd}
\shs\,\sflblbgp
\arrow{r}{\xfxi}
&
\xcpCs
\arrow{r}{\xei}
&
\sftvrbb
\arrow{r}{\xxi}
&
\shs\,\sfcnlblbgp.
\end{tikzcd}
\end{equation}
Both distinguished traingles~\eqref{eq:dsttt} and~\eqref{eq:dstto} are associated with the same morphism $\xxi$, hence they are the same. This identification means
that the complex $\xcpCs$ is \qisc\ to the bimodule $\sftvrtb$ and it represents a free resolution of the latter. Hence the filtration of $\xcpCs$ determines the filtration of the bimodule $\sftvrtb$ and we declare the former to be the \cfres\ of the latter:
\begin{equation}
\label{eq:blbtcrs}
\sfcntvrtb =
\begin{tikzpicture}[baseline=(current bounding box.base)]
\node (N) at (0,0)
{
\begin{tikzcd}[row sep=0cm]
&
\shs\sfcntvirp
\arrow{r}
\arrow{ddr}
&
\shst\sfcntrilo
\arrow{dr}
\\
\sfcntparl
\arrow{ur}
\arrow{dr}
&
{}
&
{}
&
\shsh\sfcntvrtr
\\
&
\shs\sfcntpvir
\arrow[crossing over]{uur}
\arrow{r}
&
\shst\sfcntlelo
\arrow{ur}
\end{tikzcd}
};
\node[draw,fit=(N) ] {};
\end{tikzpicture}
\end{equation}
%
%
Thus the splitting~\eqref{eq:splstr} can be presented as a pair of (unfiltered) isomorphisms
\[
\begin{tikzpicture}[ampersand replacement=\&]
\matrix (m) [smmt, column sep=2.75em]
{
\sftvrbb
\&
\left(\sftvrtb \oplus \shs\sflblbgp\right),
\\};
\path[->,font=\scriptsize]
([yshift=\shar] m-1-1.east)  edge node[auto]
{$\bigl( \begin{smallmatrix}\xeit \\ \xxi \end{smallmatrix} \bigr)$} ([yshift=\shar] m-1-2.west)
([yshift=-\shar] m-1-2.west) edge node[auto]{$( \begin{smallmatrix}\xei & \xxit \end{smallmatrix} )$} ([yshift=-\shar] m-1-1.east);
\end{tikzpicture}
\]
where the untilded homomorphisms $\xxi$ and $\xei$ are filtered and appear in the distinguished triangle~\eqref{eq:dsttt}, whereas the tilded homomorphisms $\xxit$ and $\xeit$ violate filtration. However, it is easy to see that the depth of filtration limits these violation to two units.

\begin{lemma}
There is a \hec\ (retraction) from $\cpA$ of~\eqref{eq:cpsAB} to $\sftvrtb$
\begin{equation}
\label{eq:smhmte}
\begin{tikzpicture}
\matrix (m) [smmt, column sep=1.75em]
{
\cpA = \Biggl [
\sftvrbb
&
\shs\sflblbgp\Biggr]
&&
\sftvrtb,
\\};
\path[->,font=\scriptsize]
([yshift=\shar] m-1-2.east)  edge node[auto]
{$\xfAas$} ([yshift=\shar] m-1-4.west)
([yshift=-\shar] m-1-4.west) edge node[auto]{$\xfasA$} ([yshift=-\shar] m-1-2.east)
(m-1-1)  edge node[auto]{$\xxi$} (m-1-2)
([yshift=-0.3em]m-1-1.west) edge [loop left,distance=0.75cm,
] node [auto]{$\hmA$} ([yshift=0.3em]m-1-1.west)
;
\end{tikzpicture}
\end{equation}
such that $\xfasA$ is filtered, while $\xfAas$ and $\hmA$ violate filtration by two units, and there is a relation
$\hmA\xfasA=0$.
\end{lemma}
\begin{proof}
The \hec s $\xfasA$ and $\xfAas$ are
\[
\begin{tikzcd}
\sftvrtb
\arrow{d}{\xei}
\\
\Biggl [
\sftvrbb \arrow{r}{\xxi}
\arrow{d}{\xeit}
&
\shs\sflblbgp\Biggr]
\\
\sftvrtb
\end{tikzcd}
\]
(note that $\xxi\xei=0$) while the homotopy $\hmA$ is
\[
\begin{tikzcd}
\Biggl [
\sftvrbb \arrow{r}{\xxi}
&
\shs\sflblbgp\Biggr]
\arrow[swap]{dl}{\xxit}
\\
\Biggl [
\sftvrbb \arrow{r}{\xxi}
&
\shs\sflblbgp\Biggr].
\end{tikzcd}
\]
\end{proof}

Define the complex $\cpCsm$ of \ex{eq:thrdc} as
\begin{equation}
\label{eq:dfcpsm}
\cpCsm =
\begin{tikzcd}
\Bigl[\sftvrtb
\arrow{r}{\xkap}
&
\cpB
\Bigr],
\end{tikzcd}
\end{equation}
where $\cpB$ is defined by \ex{eq:cpsAB}, while the
morphism $\xkap$ is a composition of $\xfAas$ and $\xkapt$:
\[
\begin{tikzpicture}
\matrix (m) [smmt, column sep=1.75em]
{
\sftvrtb
&&
\cpA=\Biggl [
\sftvrbb
&
\shs\sflblbgp\Biggr]
&&
\cpB
\\};
\path[->,font=\scriptsize]
(m-1-1) edge node[auto]{$\xfAas$} (m-1-3)
(m-1-3) edge node[auto]{$\xxi$} (m-1-4)
(m-1-4) edge node[auto]{$\xkapt$} (m-1-6)
(m-1-1) edge [bend left=25] node[auto]{$\xkap$} (m-1-6);
\end{tikzpicture}
\]
Since $\xfAas$ and $\xkapt$ are filtered, $\xkap$ is also filtered.
\begin{lemma}
There is a \hec\ (retraction)
\[
\begin{tikzpicture}
\matrix (m) [smmt, column sep=2em]
{
\cpCsm
=
\Bigl[\sftvrtb
&
\cpB
\Bigr]
&&
\Bigl[\cpA
&
\cpB
\Bigr]
\\};
\path[->,font=\scriptsize]
(m-1-1) edge node[auto]{$\xkap$} (m-1-2)
(m-1-4) edge node[auto]{$\xkapt$} (m-1-5)
([yshift=\shar] m-1-2.east)  edge node[auto]{$\zzf$} ([yshift=\shar] m-1-4.west)
([yshift=-\shar] m-1-4.west) edge node[auto]{$\zzg$} ([yshift=-\shar] m-1-2.east)
([yshift=0.3em]m-1-5.east) edge [loop right,distance=0.75cm,] node [auto]{$\zzh$} ([yshift=-0.3em]m-1-5.east);
\end{tikzpicture}
\]
such that $\zzf$ is filtered while $\zzg$ and $\zzh$ violate filtration by at most two units.
\end{lemma}
\begin{proof}
Since there is a relation $\hmA\xfasA=0$, then in accordance with Lemma~\eqref{lm:lmcn} one can choose \hec s $\zzf$ and $\zzg$  as
\[
\begin{tikzcd}
\Bigl[\cpA \arrow{r}{\xkapt}
\arrow{d}{\xfAas}
\arrow{dr}{\xkapt\,\hmA}
&
\cpB
\Bigr]
\arrow{d}{\xId}
\\
\Bigl[\sftvrtb
\arrow{r}{\xkap}
\arrow{d}{\xfasA}
&
\cpB
\Bigr]
\arrow{d}{\xId}
\\
\Bigl[\cpA \arrow{r}{\xkapt}
&
\cpB
\Bigr],
\end{tikzcd}
\]
while the homotopy $\zzh$ can be chosen as
\[
\begin{tikzcd}
\Bigl[\cpA
\arrow{r}{\xkapt}
\arrow{d}{\hmA}
&
\cpB
\Bigr]
\\
\Bigl[\cpA
\arrow{r}{\xkapt}
&
\cpB
\Bigr]
\end{tikzcd}
\]
Since $\xfasA$ and $\xkapt$ are filtered, while $\hmA$ violates filtration by at most two units, it follows that $\zzf$ is filtered, while $\zzg$ and $\zzh$ violate filtration by at most two units.
\end{proof}
Combining this lemma with the filtered \hec~\eqref{eq:trbrscn}, we get the following
\begin{corollary}
There is a \hec\ (retraction)
\[
\begin{tikzpicture}
\matrix (m) [smmt, column sep=2.75em]
{
\cpCsm
&
\yectv{\zzvv{\lncf}{-25}{\dbrdm}{\dbrdm}{\dbrdm}}
\\};
\path[->,font=\scriptsize]
([yshift=\shar] m-1-1.east)  edge node[auto]{$\fsr$} ([yshift=\shar] m-1-2.west)
([yshift=-\shar] m-1-2.west) edge node[auto]{$\frs$} ([yshift=-\shar] m-1-1.east);
\end{tikzpicture}
\]
such that $\fsr$ is filtered, while $\frs$ violates filtration by two units.
\end{corollary}

It remains to show that the complex $\cpCsm$ of \ex{eq:dfcpsm} is symmetric with respect to reflection of the diagrams across the vertical axis. Indeed, the complex has the form
\[
\cpCsm =
\left[
\begin{tikzcd}[row sep=0cm,column sep=0.7cm]
&
\sftlelob
\arrow{r}
\arrow{ddr}
&
\sflblbbp
\arrow{dr}
\\
\sftvrtb
\arrow{ur}[pos=0.65]{\zzfu}
\arrow{dr}[swap,pos=0.65]{\zzfd}
&
{}
&
{}
&
\sftparl
\\
&
\sftrilob
\arrow[crossing over]{uur}
\arrow{r}
&
\sflpblbb
\arrow{ur}
\end{tikzcd}
\right]
\]
where the unmarked arrows denote the morphisms $\chm$ similar to~\eqref{eq:tpchm}, while $\zzfu$ and $\zzfd$ are components of the morphism $\xkap$. Since the homotopy equivalence from the complex~\eqref{eq:bgthrc} to the complex~\eqref{eq:trintcn} involves the homotopy of only the front face of the cube~\eqref{eq:bgthrc}, the morphisms $\zzfu$ are compositions of $\xei$ and \tvsd\ morphisms:
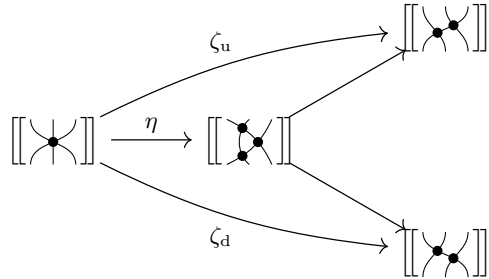
\[
\begin{tikzcd}
&&
\sftlelob
\\
\sftvrtb
\arrow{r}{\xei}
\arrow[bend left=10]{rru}{\zzfu}
\arrow[bend right=10]{rrd}[swap]{\zzfd}
&
\sftvrbb
\arrow{ru}
\arrow{rd}
\\
&&
\sftrilob
\end{tikzcd}
\]
Now it is easy to verify that $\zzfu$ and $\zzfd$ are canonical morphisms associated with two presentations of $\sftlelob$ as cones originating from the splits of the canonical complex~\eqref{eq:blbtcrs}:
\[
\sftvrtb \simeq
\begin{tikzpicture}[commutative diagrams/every diagram]
\node (N) at (0,0)
{
\begin{tikzcd}
\sftlelob
\rar
&
\shs^2\sftvrbu
\end{tikzcd}
};
\node[draw,fit=(N)] {};
\end{tikzpicture}
\simeq
\begin{tikzpicture}[commutative diagrams/every diagram]
\node (N) at (0,0)
{
\begin{tikzcd}
\sftrilob
\rar
&
\shs^2\sftvrbd
\end{tikzcd}
};
\node[draw,fit=(N)] {};
\end{tikzpicture}
\]
%
%
%
%
the first (resp. the second) cone corresponing to splitting off the upper-right (resp. upper-left) side of the hexagonal complex~\eqref{eq:blbtcrs}. Now the symmetry of the complex $\cpCsm$ becomes apparent, and this completes the proof of Theorem~\ref{th:trrx}.
\end{proof}

\section{Markov moves}
\label{s:mmvs}
\subsection{Filtered homology of a link diagram}
%
%

Let $\xLbb$ be the oriented unframed link diagram constructed by the circular closure of an $\bstr$-strand \brwd\ $\brwb$.
We define the \fcompl\ $\ctxLbb$ of a link diagram $\xLbb$ as the result of replacing the filtered $\Qbxy$-modules in the complex $\xctvbyx{\brwb}$ by their derived tensor products with the \dgbm\ corresponding to the $\bstr$-strand identity braid $\xIdbstr$
with a special grading shift:
\begin{equation}
\label{eq:cmplb}
\ctxLbb =
\stashf^{\bstr}
\,\Bigl( \xctvbyx{\brwb}\otdr_{\Qbxy}
\xctvbyx{\xIdbstr}\Bigr),
\end{equation}
where
$
\xctvbyx{\xIdbstr}=\Qbxy/(y_1 - x_1,\ldots,y_{\bstr} - x_{\bstr}).
$
A derived tensor product of two filtered complexes in the derived category $\DQfbxy$ is an object in the homotopy category of filtered complexes $\KQfb$. Hence $\ctxLbb$ is an object of the category $\ChKQfb$ of chain complexes over $\KQfb$. Note that generally a complex in the category $\KQfb$ is not \het\ to its homology because of the filtration.

\begin{remark}
\label{rm:mrmvo}
Two \brwd s $\brwb$ and $\brwbp$ yield the same link diagram by circular closure if they
are related by the first Markov move, that is, if there exist \brwd s $\brwbo$ and $\brwbt$ such that
$\brwb = \brwbt\brwbo$, $\brwbp = \brwbo\brwbt$. Obviously, the complex $\ctxLbb$ is invariant under the first Markov move, because of the isomorphisms
\begin{equation}
\label{eq:cplcls}
\ctxLbb \cong
\stashf^{\bstr}
\,
\Bigl(\xctvbyx{\brwbo} \otdr_{\Qbxy} \xctvbxy{\brwbt}\Bigr)
\cong
\ctxLbbp.
\end{equation}
Hence the complex~\eqref{eq:cmplb} is indeed determined by the link diagram $\xLbb$ rather than by the \brwd\ $\brwb$.
\end{remark}

An application of the homology with respect to the inner differential (that is, of the category $\KQfb$) and with respect to the outer differential (that is, of category $\ctChv{\dmmy}\;$) produces two more invariants of the diagram $\xLbb$:
\begin{equation}
\label{eq:innhom}
\Hmint(\ctxLbb) =
\stashf^{\bstr}
\,\HHm(\ctbrwb),
\end{equation}
($\HHm$ being the \Hchh), which is an object of the category $\ChQfb$ of filtered, graded chain complexes, and the filtered triply graded homology
\[
\Htg(\xLbb) = \Hmext\bigl(\Hmint(\ctxLbb)\bigr).
\]

\begin{remark}
\label{rm:sqimp}
Since $\Hmint(\ctxLbb)$ is derived from $\ctxLbb$ and $\Htg(\xLbb)$ is derived from $\Hmint(\ctxLbb)$, then for two \brwd s $\brwb$ and $\brwbp$ there is a sequence of implications:
\[
\ctxLbbp\sim\ctxLbb \quad\Rightarrow\quad \Hmint(\ctxLbbp)\sim \Hmint(\ctxLbb)
\quad\Rightarrow\quad\Htg(\ctxLbbp)\cong\Htg(\ctxLbb).
\]
\end{remark}


\subsection{Markov moves}

Fix an $\bstr$-strand \brwd\ $\brwb$. By definition, the braids
\begin{equation}
\label{eq:mrkmts}
\brwbp = (\sgo\sqbr\xIdbrn)(\xIdbro\sqbr\brdb),\qquad \brwbpp = (\sgomo\sqbr\xIdbrn)(\xIdbro\sqbr\brdb)
\end{equation}
are the result of applying Markov moves 2A and 2B to $\brwb$, so the circular closures $\xLbb$, $\xLbbp$ and $\xLbbpp$ represent isotopic link diagrams.

\begin{theorem}
\label{th:mrkmvt}
For any \brwd\ $\brwb$ and for the \brwd s~\eqref{eq:mrkmts}, which result from applying Markov moves 2A and 2B to $\brwb$, there are \hec s
\begin{equation}
\label{eq:smveqv}
\ctxLbbp\sim\ctxLbb,\qquad
\Hmint(\ctxLbbpp)\sim \Hmint(\ctxLbb).
\end{equation}
\end{theorem}
\begin{remark}
In view of Remark~\ref{rm:sqimp}, the first equivalence of~\eqref{eq:smveqv} is stronger than the second equivalence.
\end{remark}

\begin{proof}[Proof of Theorem~\ref{th:mrkmvt}]
Denote by $\xtLbb$ the diagram of a (1,1)-tangle resulting from the circular closure of all strands of $\brwb$ except the first strand. Define the corresponding complex $\ctxtLbb$ similar to \ex{eq:cmplb}:
\[
\ctxtLbbxyo =
\stashf^{\bstr-1}
\left(
\ctbrwbbxy \otdr_{\Qv{\bfx',\bfy'}}
\xctvbyxp{\xIdv{\bstr-1}}
\right),
\]
where $\bfxp = x_2,\ldots, x_{\bstr}$ and $\bfyp$ is defined similarly.
Also
for any 2-strand diagram $\zdgrfv{\lncf}{5}$ we define the complex of its partial (1-strand) closure as a  derived tensor product similar to~\eqref{eq:cplcls}:
\begin{equation}
\label{eq:dfprcl}
\sfdgrclxyt =
\stashf
\Bigl(
\sfdgrfxy
\otdr_{\Qv{x_1,y_1}}
\sfverlbxyo
\Bigr).
\end{equation}
Now it is easy to see that the complexes of the \brwd s~\eqref{eq:mrkmts} can be presented as derived tensor products
\begin{align}
\label{eq:brtlo}
\ctxLbbp
&\cong
\stashf
\left(
\sfbrdcpyxt \otdr_{\Qv{x_2,y_2}}\ctxtLbbxyt
\right)
,
\\
\label{eq:brtlt}
\ctxLbbpp
&\cong
\stashf
\left(
\sfbrdcmyxt
\otdr_{\Qv{x_2,y_2}}\ctxtLbbxyt
\right),
\end{align}
whereas
\begin{equation}
\label{eq:cltll}
\ctxLbb \cong \stashf
\left(
\sfverlbxyo \otdr_{\Qv{x_1,y_1}}
\ctxtLbbxyo
\right).
\end{equation}
\begin{lemma}
\label{lm:heqlm}
There are \hec s
\begin{equation}
\label{eq:heqlm}
\sfbrdpcl\sim \sfverl,\qquad
\sfbrdmcl\sim
\left[\;
\begin{tikzpicture}[baseline=-0.25em]
\matrix (m) [ssmtt]
{
\sha\,\sfverl
&
\sfverl
\\
\shs\,\sfverl
\\};
\node  (PH) at (m-1-1) [xshift=0.75cm]
{};
\path[->,font=\scriptsize]
(m-1-1) edge node[auto] {$\xId$} (m-2-1)
(PH) edge node[auto] {$\xId$} (m-1-2);
\node (bx) [draw,fit=(m-1-1) (m-2-1) ] {};
\path (bx) node [xshift=-2em] {$\sht\,$};
\end{tikzpicture}
\!\!\!
\right].
\end{equation}
\end{lemma}

Replacing the first factor in the tensor product~\eqref{eq:brtlo} with the help of the first \hec\ of~\eqref{eq:heqlm} converts the \rhs of~\eqref{eq:brtlo} into the \rhs of~\eqref{eq:cltll}, thus proving the first \hec\ of~\eqref{eq:smveqv}.

A substitution of the second \hec\ of~\eqref{eq:heqlm} into \ex{eq:brtlt} yields the presentation of $\ctxLbbpp$ as a cone of two tensor products:
%
\begin{equation}
\label{eq:hmeqcn}
\ctxLbbpp
\sim
\stashf
\left[
\sht\left(
A
\otlgt
\right)
\xrightarrow{\;\;\;\;f\otimes\xId\;\;\;\;}
\left(\sverlxyt
\otlgt
\right)
\right]
\end{equation}
where
\[
A =
\begin{tikzpicture}[baseline=-0.25em]
\matrix (m) [ssmto]
{
\left(\sha\;\sverlxyt\right)
&
\left(\shs\;\sverlxyt\right)
\\};
\path[->,font=\scriptsize]
(m-1-1) edge node[auto] {$\xId$} (m-1-2);
\node[draw,fit=(m-1-1) (m-1-2) ] {};
\end{tikzpicture}.
\]
If we forget filtration, then the cone $A$ is contractible, hence
\[
\Hmint\left(A
\otlgt
\right) = 0
\]
and the application of $\Hmint$ to both sides of the \hec~\eqref{eq:hmeqcn} produces the second \hec\ of~\eqref{eq:smveqv}.
\end{proof}

\begin{proof}[Proof of Lemma~\ref{lm:heqlm}]
We compute the derived tensor product in~\eqref{eq:dfprcl} for the \rhs of this presentation by using the \cfres\ for 2-strand graphs and presenting $\sfverlbxyo$ as the diagonal bimodule $\dgMo=\Qxyo/(y_1-x_1)$:
\[
\sfdgrclxyt \cong
\stashf
\Bigl(
\sfcdgrfxy
\otimes_{\Qv{x_1,y_1}}
\dgMo
\Bigr).
\]
Taking the tensor product with $\dgMo$ amounts to setting $x_1=y_1$ and forgetting the structure of a module over $\Qxyo$. After we replace the dummy variables $x_1=y_1$ with another dummy variable  $z=(x_2 + y_2) - (x_1 + y_1)$, the complexes of 1-strand closures of the diagrams $\zparlv{\lncf}{5}$ and $\zvirtv{\lncf}{5}$ take the form (\cf \ex{eq:stkzrdo})
\begin{eqnarray}
\label{eq:qusoo}
\sfparlclxyt
& \simeq &
\stashf
\left(\dgbxyt \otimes
\begin{tikzpicture}[baseline=-0.25em]
\matrix (m) [ssmt]
{
\saqt\Qz
&
\Qz
\\};
\path[->,font=\scriptsize]
(m-1-1) edge node[auto] {$0$} (m-1-2);
\node[draw,fit=(m-1-1) (m-1-2) ] {};
\end{tikzpicture}
\;
\right)
\simeq
\sverlxyt\otimes\stashf\bigl(\Qclcz\bigr)
\\
\label{eq:qusot}
\sfvirtclxyt
& \simeq &
\stashf
\left(\dgbxyt \otimes
\begin{tikzpicture}[baseline=-0.25em]
\matrix (m) [ssmt]
{
\saqt\Qz
&
\Qz
\\};
\path[->,font=\scriptsize]
(m-1-1) edge node[auto] {$z$} (m-1-2);
\node[draw,fit=(m-1-1) (m-1-2) ] {};
\end{tikzpicture}
\;
\right)
\simeq
\stashf\;
\sverlxyt
\end{eqnarray}
In the last line we used a splitting $\Qz \cong \Qzz \oplus \zQz$, which isolates a contractible subcomplex:
\[
\begin{tikzpicture}[baseline=-0.25em]
\matrix (m) [ssmt]
{
\saqt\Qz
&
\Qz
\\};
\path[->,font=\scriptsize]
(m-1-1) edge node[auto] {$z$} (m-1-2);
\node[draw,fit=(m-1-1) (m-1-2) ] {};
\end{tikzpicture}
\cong
\Qzz \oplus
\begin{tikzpicture}[baseline=-0.25em]
\matrix (m) [ssmt]
{
\saqt\Qz
&
\zQz
\\};
\path[->,font=\scriptsize]
(m-1-1) edge node[auto] {$\xId$} (m-1-2);
\node[draw,fit=(m-1-1) (m-1-2) ] {};
\end{tikzpicture}
\sim\Qzz\cong\IQ.
\]

Now we apply the 1-strand closures to the definitions~\eqref{eq:rqco} and~\eqref{eq:rqct}. We begin with the complex~\eqref{eq:rqco}. After the closure, the morphism $\vsdm$ becomes homotopically trivial, hence
\begin{eqnarray}
\nonumber
\sfbrdpcl
& \cong &
\stasaqt
\left[
\begin{tikzpicture}[baseline=-0.25em]
\matrix (m) [ssmth]
{
&
(\stat\ssqt)^{-1}\sfvirtcl
\\
\sfparlcl
&
\stati\sfparlcl
\\};
\node  (PH) at (m-2-2) [xshift=-1.75cm] {};
\path[->,font=\scriptsize]
(m-1-2) edge node[auto] {$0$} (m-2-2)
(m-2-1) edge node[auto] {$\xId$} (PH);
\node[draw,fit=(m-1-2) (m-2-2) ] {};
\end{tikzpicture}
\;\;
\right]
\\
\nonumber
& \cong &
\xsht^{-1}\sfvirtcl \oplus
 \xsht\saqt
 \;
 \left[
\begin{tikzpicture}[baseline=-0.25em]
\matrix (m) [ssmtf]
{
\sfparlcl
&
\stati\sfparlcl
\\};
\path[->,font=\scriptsize]
(m-1-1) edge node[auto] {$\xId$} (m-1-2);
\end{tikzpicture}
\right]
\sim\sfverl
\end{eqnarray}
The last \hec\ is due to the contractibility of the complex in square brackets, and we also used the \qiso~\eqref{eq:qusot}. Thus we proved the first relation of~\eqref{eq:heqlm}.

Now we apply closure to the complex~\eqref{eq:rqct}:
\[
\sfbrdmclxyt \cong \sverlxyt\otimes B,
\]
where
\begin{eqnarray}
\nonumber
B
& \cong &
\stasaqti
\left[\;
\begin{tikzpicture}[baseline=-0.25em,ampre]
\matrix (m) [ssmth]
{
\Qclczbr
\&
\xsht\Qclczbr
\\
\ssqt\,\Qzzbr
\&
\\};
\node  (PH) at (m-1-1) [xshift=+1.75cm] {};
\path[->,font=\scriptsize]
(PH) edge node[auto] {$\xId$} (m-1-2)
(m-1-1) edge node[auto] {($\begin{smallmatrix} 1 & 0 \end{smallmatrix})$} (m-2-1);
\node (bx) [draw,fit=(m-1-1) (m-2-1) ] {};
\path (bx) node [xshift=-5em] {$\xsht\stat$};
\end{tikzpicture}
\;\;
\right]
\\
\nonumber
& \cong &
\left[\;
\begin{tikzpicture}[baseline=-0.25em]
\matrix (m) [ssmts]
{\sha\,\Qzz
&
\Qzz
\\
\shs\,\Qzz
\\};
\node (PH) at (m-1-1) [xshift=+2.75em]{};
\path[->,font=\scriptsize]
(PH) edge node[auto] {$\xId$} (m-1-2)
(m-1-1) edge node[auto] {$\xId$} (m-2-1);
\node (bx) [draw, fit=(m-1-1) (m-2-1) ] {};
\path (bx) node [xshift=-3.5em] {$\sht$};
\end{tikzpicture}
\right]
\\
&&\quad
\oplus
\left[
\begin{tikzpicture}[baseline=-0.25em]
\matrix (m) [ssmts]
{
\stat
\bigl(\zQz\oplus
\saqt^{-1}
\Qz
\bigr)
&
\bigl(\zQz\oplus
\saqt^{-1}
\Qz
\bigr)
\\
};
\path[->,font=\scriptsize]
(m-1-1) edge node[auto] {$\xId$} (m-1-2);
\end{tikzpicture}
\right]
\end{eqnarray}
Since the complex in the last line is contractible, the second \hec\ of~\eqref{eq:heqlm} follows.
\end{proof}

\section{Filtered homology of two-strand torus braids and links}
\label{s:fhtbr}

Denote $\ysgmm= \zbrdmv{\lncf}{5}$.
\begin{theorem}
\label{th:fhtbr}
For $n\geq 0$, the complex of the \brwd\ $\ysgmmn=(\ysgmm)^{n}$ has a form
\begin{align}
\ctsgmtn &\sim
(\stasaqt)^{-2n}\stat
\left[
\underbrace{
\stot^{2n-1}\sfblbg\xrightarrow{\xpsip}
\cdots
\xrightarrow{\xpsip}
\stot^2\sfblbb
\xrightarrow{\xpsim}
\stot\sfblbg
\xrightarrow{\xpsip}
\sfblbb
}_{2n}
\xrightarrow{\chm}
\stat^{-1}\sfparl
\right]
\\
\ctsgmtno &\sim
(\stasaqt)^{-(2n+1)}\stat
\left[
\underbrace{
\stot^{2n}\sfblbb\xrightarrow{\xpsim}
\cdots
\xrightarrow{\xpsim}
\stot\sfblbg
\xrightarrow{\xpsip}
\sfblbb
}_{2n+1}
\xrightarrow{\chm}
\stat^{-1}\sfparl
\right]
\end{align}
where $\stot = \stat\ssqt$.
\end{theorem}
\begin{proof}
This theorem is proved by induction over $n$. It obviously holds for $n=0$ and for $n=1$. The step of induction is proved with the help of the following lemma whose proof we leave for the readers.
\end{proof}
\begin{lemma}
There are \hec s
\begin{equation}
\label{eq:twhecs}
\sfbrmblb \sim
\stashf
\sfblbg,\qquad
\sfbrmblg \sim
\stashf
\sfblbb.
\end{equation}
which make the following diagram commutative:
\begin{equation*}
\begin{tikzcd}
\sfbrmblb
\arrow{d}{\sim}
\arrow{r}{\xId\otimes\xpsim}
&
\sfbrmblg
\arrow{d}{\sim}
\arrow{r}{\xId\otimes\xpsip}
&
\sfbrmblb
\arrow{d}{sim}
\arrow{r}{\chm\otimes\xId}
&
\sfbrmpar
\arrow{d}{\cong}
\\
\sfblbg
\arrow{r}{\xpsip}
&
\sfblbb
\arrow{r}{\xpsim}
&
\sfblbg
\arrow{r}{\xthet}
&
\sfbrdm
\end{tikzcd}
\end{equation*}
where $\xthet$ is the morphism
\[
\begin{tikzpicture}[baseline=-0.25em]
\matrix (m) [ssmt]
{
\sfblbg
&
\sfblbb
\\
&
\sfparl
\\};
\path[->,font=\scriptsize]
(m-1-1) edge node[auto] {$\xpsip$} ([xshift=-1.75em] m-1-2)
(m-1-2) edge node[auto] {$\chm$} (m-2-2);
\node (P1) [fit=(m-1-2) (m-2-2), left delimiter={[},right delimiter={]},inner sep=-0.1cm ] {};
\path (P1) node [xshift=3.5em] {$\cong \sfbrdm$};
\end{tikzpicture}
\]
\end{lemma}

\begin{theorem}
\label{th:fhmtbr}
The \ftghm\ homology of the circular closure of the \brwd\ $\ysgmmn$ is
\begin{gather*}
\Htg(\xLsgmtno)
 \cong (\stashf\saqt)^{-2n}\Htg(\bigcirc)\otimes
\left(\IQ\oplus
\stat\bigoplus_{k=0}^{n-1}\stot^{2k+1}\bigl(\ssqt\IQ\oplus\saqt^{-1}\IQ\bigr)
\right),
\\[2\jot]
\Htg(\xLsgmtn)  \cong
(\stasaqt)^{-1}\Htg(\xLsgmtnmo) \oplus \stashf^{2n-1} \saqt^{-1}\stat\Htg(\bigcirc)\otimes\bigl(
\saqt\ssqt\Qz\oplus\ssqt\Qz\oplus\IQ \bigr).
\end{gather*}
In the latter formula it is assumed that $n\geq 1$.
\end{theorem}
This theorem is an easy corollary of the following lemma:
\begin{lemma}
An application of the \Hchh\ functor $\HHm$  to
the sequence of morphisms
\begin{equation}
\label{eq:sqfr}
\begin{tikzcd}
\sfblbb
\arrow{r}{\xpsim}
&
\sfblbg
\arrow{r}{\xpsip}
&
\sfblbb
\arrow{r}{\chm}
&
\sfparl
\end{tikzcd}
\end{equation}
apart from the common factor $\stashf^{-1}\Htg(\bigcirc)$, where
\[
\Htg(\bigcirc) = \stashf\Bigl (\saqt\Qw \oplus \Qw \Bigr)
\]
is the homology of the unknot diagram,
yields the following sequence of maps between graded, filtered vector spaces
\begin{equation}
\label{eq:chvs}
\begin{tikzpicture}[commutative diagrams/every diagram]
\matrix (m) [smmt, column sep=2.75em]
{
\saqt\shq^2\,
\Qz
&
\saqt\ssqt\,
\Qz
&
\saqt\shq^2\,
\Qz
&
\saqt\,
\Qz
\\
\Qz
&
\ssqt\,
\Qz
&
\Qz
&
\Qz
\\
&
\IQ
\\};
\path[->,font=\scriptsize]
(m-1-1) edge node[auto]{$z$} (m-1-2)
(m-1-3) edge node[auto]{$z$} (m-1-4)
(m-2-1) edge node[auto]{$1$} (m-2-2)
(m-2-3) edge node[auto]{$1$} (m-2-4)
;
\path (m-1-1) -- (m-2-1) node[midway] {$\oplus$};
\path (m-1-2) -- (m-2-2) node[midway] {$\oplus$};
\path (m-1-3) -- (m-2-3) node[midway] {$\oplus$};
\path (m-1-4) -- (m-2-4) node[midway] {$\oplus$};
\path (m-2-2) -- (m-3-2) node[midway] {$\oplus$};
\end{tikzpicture}
\end{equation}
where $\xdmz$ and $\xdmw$ are `dummy' variables, and each column in the diagram represents the Hochschild homology of the corresponding module in the sequence~\eqref{eq:sqfr}.
\end{lemma}
\begin{proof}[Sketch of the proof]
After we pass from the variables $\bfx$ and $\bfy$ to the variables~\eqref{eq:dfdfv}, the \cres\ of the \bmdl\ of every diagram $\zdgrfv{\lncf}{5}$ appearing in~\eqref{eq:sqfr} splits into the tensor product (over $\IQ$) of the complexes of $\Qv{\xp,\yp}$-modules and complexes of $\Qv{\xm,\ym}$-modules: $\sfndgrf \cong \dgbyxp\otimes \sfmdgrf$ (\cf \ex{eq:cmfct}), so that apart from the common factor $\dgbyxp$, on which the morphisms~\eqref{eq:sqfr} act as identity, the sequence~\eqref{eq:sqfr} becomes
\begin{equation}
\label{eq:chnmrl}
\begin{tikzpicture}[baseline=-0.25em]
\matrix (m) [ssmt, column sep=6em]
{
\sfmblbb
&
\sfmblbg
&
\sfmblbb
&
\sfmparl
\\
\sfmparl
&
\sfmvirt
&
\sfmparl
&
\sfmparl
\\
\sfmvirt
&
\sfmparl
&
\sfmvirt
\\};
\node (P21) at (m-2-1) [xshift=2em] {};
\node (P22) at (m-2-2) [xshift=2em] {};
\node (P23) at (m-2-3) [xshift=2em] {};
\path[->,font=\scriptsize]
(m-1-1) edge node[auto] {$\xpsim$} (m-1-2)
(m-1-2) edge node[auto] {$\xpsip$} (m-1-3)
(m-1-3) edge node[auto] {$\chm$} (m-1-4)
(m-2-1) edge node[auto] {$\vsdp$} (m-3-1)
(m-2-2) edge node[auto] {$\vsdm$} (m-3-2)
(m-2-3) edge node[auto] {$\vsdp$} (m-3-3)
(m-1-1) edge node[auto] {$=$} ([yshift=1.5em] m-2-1)
(m-1-2) edge node[auto] {$=$} ([yshift=1.5em] m-2-2)
(m-1-3) edge node[auto] {$=$} ([yshift=1.5em] m-2-3)
(m-1-4) edge node[auto] {$=$} (m-2-4)
(P21) edge node[auto] {$\xId$} ([xshift=-2.5em] m-3-2)
(P22) edge node[auto] {$\xId$} ([xshift=-2.5em] m-3-3)
(P23) edge node[auto] {$\xId$} ( m-2-4)
;
\node [draw, fit=(m-2-1) (m-3-1) ] {};
\node [draw, fit=(m-2-2) (m-3-2) ] {};
\node [draw, fit=(m-2-3) (m-3-3) ] {};
\end{tikzpicture}
\end{equation}
where the $\Qv{\xm,\ym}$-modules $\sfmparl$ and $\sfmvirt$ are defined by \ex{eq:dfvrtm}, while the homomorphisms $\vsdp$ and $\vsdm$ are defined by \ex{eq:mcns}.
The application of \Hchh\ to the resolutions amounts to taking the quotient over the relation $\xp=\yp$, which transforms the common factor $\dgbyxp$ into $\stashf^{-1}\Htg(\bigcirc)$, as well as over the relation $\xm=\ym$, which transforms the chain of \bmdl s~\eqref{eq:chnmrl} into the chain of vector spaces~\eqref{eq:chvs}.
\end{proof}


\end{document}
